\documentclass{article} % Document font size and equations flushed left

\usepackage{a4wide}

\usepackage[usenames,dvipsnames]{xcolor}
\usepackage[english]{babel} % Specify a different language here - english by default
\usepackage[color=cyan]{todonotes}
\usepackage{float}
\usepackage{breakcites}
\newcommand{\R}{\mathbb R}
\usepackage{bbold}
\usepackage[usenames,dvipsnames]{xcolor}
\usepackage{tcolorbox}
\usepackage{colortbl}
\usepackage{pgfplots}
\tcbuselibrary{skins}

\usepackage{mathtools}
\usepackage{tabularx}
\usepackage{array}
\usepackage{colortbl}
\usepackage{graphicx}
\tcbuselibrary{skins}
\usepackage{subfigure}
\usepackage{pgfplots}
\pgfplotsset{width=10cm,compat=newest}

\usepackage[utf8]{inputenc}
\usepackage[T1]{fontenc}  

\definecolor{color1}{RGB}{0,139,0} % Color of the article title and sections
\definecolor{color2}{RGB}{154,255,154} % Color of the boxes behind the abstract and headings

\numberwithin{equation}{section}

\newcommand{\N}{\mathbb{N}} 
\newcommand{\ra}{\rightarrow}

\usepackage{xspace}
\newcommand{\push}{\emph{push}\xspace}
\newcommand{\pull}{\emph{pull}\xspace}
\newcommand{\pushpull}{\emph{push\&pull}\xspace}
\usepackage{hyperref} % Required for hyperlinks
\usepackage{url}
\usepackage[framemethod=tikz]{mdframed}
\usepackage{cuted}
\usepackage{amsmath}
\usepackage{enumitem}
\usepackage{datetime}
\usepackage{bm}

\hypersetup{hidelinks,colorlinks,breaklinks=true,urlcolor=color2,citecolor=color1,linkcolor=color1,bookmarksopen=false,pdftitle={Title},pdfauthor={Author}}
\usepackage{amssymb,amsmath,amsthm,amsfonts}
\makeatletter
\def\@endtheorem{\endtrivlist}
\makeatother

\newtheoremstyle{abcd}% name
  {}%      Space above, empty = `usual value'
  {}%      Space below
  {\itshape}% Body font
  {}%         Indent amount (empty = no indent, \parindent = para indent)
  {\bfseries}% Thm head font
  {.}%        Punctuation after thm head
  {.5em}% Space after thm head: \newline = linebreak
  {}%         Thm head spec

\theoremstyle{abcd}
\newtheorem{theorem}{Theorem}
\numberwithin{theorem}{section}
\newtheorem{definition}[theorem]{Definition}

\newtheorem{example}[theorem]{Example}

\newtheorem{proposition}[theorem]{Proposition}
\newtheorem{lemma}[theorem]{Lemma}
\newtheorem{remark}[theorem]{Remark}
\usepackage[margin=1in]{geometry}
\usepackage{tablefootnote}

\begin{document}

\title{Robustness of Randomized Rumour Spreading}

\author{Rami Daknama$^1$ \and Konstantinos Panagiotou$^1$ \and Simon Reisser$^1$}

%\flushbottom % Makes all text pages the same height
%\Keywords{Rumour Spreading --- Local Resilience --- Robustness --- Szem\`eredis regularity lemma}
\date{$^1$Ludwig-Maximilians-Universität München\\[2ex] 12\textsuperscript{th} April, 2019}%}
%\date{\today}
\maketitle % Print the title and abstract box

\begin{abstract}
In this work we consider three well-studied broadcast protocols: \push, \pull and \pushpull. A key property of all these models, which is also an important reason for their popularity, is that they are presumed to be very robust, since they are simple, randomized, and, crucially, do not utilize explicitly the global structure of the underlying graph. While sporadic results exist, there has been no systematic theoretical treatment quantifying the robustness of these models. Here we investigate this question with respect to two orthogonal aspects: (adversarial) modifications of the underlying graph and message transmission failures.  

We explore in particular the following notion of \emph{local resilience}: beginning with a graph, we investigate up to which fraction of the edges an adversary has to be allowed to delete at each vertex, so that the protocols need significantly more rounds to broadcast the information.
Our main findings establish a separation among the three models. It turns out that \pull is robust with respect to all parameters that we consider. On the other hand, \push may slow down significantly, even if the adversary is allowed to modify the degrees of the vertices by an arbitrarily small positive fraction only. Finally, \pushpull is robust when no message transmission failures are considered, otherwise it may be slowed down. 

On the technical side, we develop two novel methods for the analysis of randomized rumour spreading protocols. First, we exploit the notion of self-bounding functions to facilitate significantly the round-based analysis: we show that for any graph the variance of the growth of informed vertices is bounded by its expectation, so that concentration results follow immediately. Second, in order to control adversarial modifications of the graph we make use of a powerful tool from extremal graph theory, namely   Szemer\`edi's Regularity Lemma.

%For all these algorithms, we will additionally assume that there are independent message transmission failures; i.e.~each transmission succeeds independently with probability $q \in (0,1]$.
%We show that if we start with a complete graph on $n$ vertices and delete edges such that the resulting graph satisfies a certain expansion property, then, for any $q$, all three protocols remain as fast as before.
% In particular all three protocols are robust with respect to changes of the network topology and message transmission failures simultaneously.
% Afterwards, we explore the following notion of resilience for graphs with that expansion property: We start with such a graph and investigate up to which fraction of the edges an adversary has to be allowed to delete at each vertex, such that the respective broadcast protocol needs significantly, i.e.~$\Omega(\log n)$,  more rounds on the resulting graph.
%Surprisingly, only \pull turns out to be robust with respect to adversarial edge deletions for all $q \in (0,1]$; for $q=1$ also \pushpull is robust in this sense. Unexpectedly, \pushpull (for $q<1$) and \push (for any $q$) are slowed down by $\Omega(\log n)$ rounds, even if the adversary may only delete up to an arbitrarily small linear fraction of the edges at each vertex.
\end{abstract}
%----------------------------------------------------------------------------------------

%\tableofcontents % Print the contents section

\thispagestyle{empty} % Removes page numbering from the first page

%----------------------------------------------------------------------------------------
%	ARTICLE CONTENTS
%----------------------------------------------------------------------------------------

\section{Introduction}

Randomized broadcast protocols are highly relevant for data distribution in large networks of various kinds, including technological, social and biological networks. Among many others there are three basic models in the literature, introduced in \cite{frieze1985shortest,Demers1988,Karp2000}, namely \push, \pull and \emph{push\&pull} (or short \emph{pp}). Consider a connected graph in which  some vertex holds a piece of information; we call this vertex (initially) informed. All three models have the common characteristic that they proceed in rounds. In the  \push model, in every round every informed vertex chooses a neighbour independently and uniformly at random (iuar) and informs it; this of course has only an effect if the target vertex was previously uninformed. Contrary, in the \pull model every round every \emph{un}informed vertex chooses a neighbour iuar and asks for the information. If the asked vertex has the information, then the asking vertex becomes informed as well. The third model \pushpull combines both worlds: in each round, each vertex chooses a neighbour iuar, and if one of both vertices is informed, then afterwards both  become so. We additionally assume that each message transmission succeeds independently with probability $q\in (0,1]$. For these algorithms, the main parameter that we consider is the random variable that counts how many rounds are needed until all vertices are informed, and we call these quantities the \textit{runtimes} of the respective algorithms.

In the remainder we will denote the runtime of \push by $T_{push}(G,v,q)$ where $G$ is the underlying graph, initially the vertex $v$ is informed and we have a transmission success probability of $q\in (0,1]$. Analogously we denote the runtimes of \pull and \pushpull by $T_{pull}(G,v,q)$ and $T_{pp}(G,v,q)$ respectively. If the choice of $v$ does not matter we will omit it in our notation. The most basic case is when $G$ is the complete graph $K_n$ with $n$ vertices. Then, see for example Doerr and Kostrygin \cite{doerr2017randomized}, it is known that for ${\cal P} \in \{push, pull, pp\}$ and $q\in (0,1] $ in expectation and with probability tending to $1$ as $n\to\infty$
$$T_{\cal P}(K_n,q)=c_{\cal P}(q)\log n+o(\log n),$$
where, for $q\in (0,1)$,
\begin{align*}
c_{push}(q):=\frac{1}{\log (1+q)}+\frac{1}{q}, \hspace{0.2cm}
c_{pull}(q):=\frac{1}{\log (1+q)}-\frac{1}{\log (1-q)}, \hspace{0.2cm}
c_{pp}(q):=\frac{1}{\log (1+2q)}+\frac{1}{q-\log (1-q)},
\end{align*}
and where we set $c_{\cal P}(1):=\lim_{q \ra 1}c_{\cal P}(q)$. If $q$ is clear from the context, we write $c_{\cal P}$ instead of $c_{\cal P}(q)$. Actually, the results in~\cite{doerr2017randomized} and also  \cite{doerr2014tight} are much more precise,  but the stated forms will be sufficient for what follows. 
%Some related results concerning randomized rumor spreading
%TODO These rumour spreading protocols and variants of them have been subject of extensive research since the 1980s, see e.g.~\cite{Frieze1985, demers_epidemic_1987, pittel_on_1987, feige1990randomized, giakkoupis_tight_2011, giakkoupis_tight_2014, clementi_rumor_2016, daum_rumor_2016, panagiotou_asynchronous_2017, doerr2017randomized}.  

\paragraph{Contribution \& Related Work}
In this article our focus is on quantifying the \emph{robustness} of all three models. Indeed, robustness is a key property that is often attributed to them, since they are simple, randomized, and, crucially, do not exploit explicitly the structure of the underlying graph (apart, of course, from considering the neighborhoods of the vertices). 
% as one of their key strengths and the aim is to provide quantitative statements. Obtaining respective results for general graphs other than $K_n$ and to investigate the robustness of these runtimes constitute important research objectives.
Clearly, the runtime can vary tremendously between different graphs with the same number of vertices. Hence it is essential to understand which structural characteristics of a graph influence in what way the runtime of rumour spreading algorithms.
%; comprehending this relation can allow to prove precise runtime bounds for graph classes in which graphs share certain structural properties. 

One result in this spirit for the \push model was shown in~\cite{Panagiotou2015}. Roughly speaking, in that paper it is shown that even on graphs with low density, if the edges are distributed rather uniformly, then \push is as fast as on the complete graph. This can be interpreted as a robustness result: starting with a complete graph, one can delete a vast amount of edges and as long as this is done rather uniformly, the runtime of \push is affected insignificantly. To state the result more precisely, we need the following notion.
\begin{definition}[$(n,\delta,\Delta,\lambda)$-graph]
\label{def}
Let G be a connected graph with $n$ vertices that has minimum degree $\delta$ and maximum degree $\Delta$. Let $\mu_1\geq\mu_2\geq\cdots\geq\mu_n$ be the eigenvalues of the adjacency matrix of G, and set $\lambda=\max_{2\leq i\leq n}|\mu_i|=\max\{|\mu_2|,|\mu_n|\}$. We will call $G$ an $(n,\delta,\Delta,\lambda)$-graph.  
\end{definition}
In this paper we are interested in the case where $G$ gets large, that is, when $n\ra \infty$. Hence all asymptotic notation in this paper is with respect to $n$; in particular ``with high probability'', or short whp,  means with probability $1-o(1)$ when $n\ra \infty$.
\begin{definition}[Expander Sequence]
\label{expander_sequence_definition}
Let $\mathcal G = (G_n)_{n \in \N}$ be a sequence of graphs, where $G_n$ is a $(n,\delta_n,\Delta_n,\lambda_n)$-graph for each $n \in\N$. We say that $\mathcal G$ is an \emph{expander sequence} if $\Delta_n/\delta_n=1+o(1)$ and $\lambda_n=o(\Delta_n)$.
\end{definition}
Note that if we consider any sequence $\mathcal G = (G_n)_{n \in \N}$ of graphs this always implicitly defines $\delta_n,\Delta_n$ and $\lambda_n$ as in Definition~\ref{expander_sequence_definition}. 
Expander graphs have found numerous applications in computer science and mathematics, see for example the survey~\cite{hoory2006expander}. If $\mathcal G$ is an expander sequence, then intuitively this means that for $n$ large enough, the edges of~$G_n$ are rather uniformly distributed. For a more formal statement compare Lemma~\ref{expmix}.
Moreover, note that our definition of expander sequences excludes the case when $d$ is bounded; this is actually a necessary condition for our robustness results to hold, compare \cite{elsasser2009runtime}.
%Furthermore note that a sequence of random $d$-regular graphs with $n$-independent $d$ is not an expander sequence as we define it, which is necessary as it was shown in \cite{fountoulakis2010reliable} that the runtime on random $d$-regular graphs significantly differs from that on the complete graph\marginpar{\tiny Änderung}.
With all these definitions at hand we can state the result from \cite{Panagiotou2015} that quantifies the robustness of \push with respect to the network topology, that is, the runtime is asymptotically the same as on the complete graph~$K_n$. 
\begin{theorem}
\label{push}
Let $\mathcal G=(G_n)_{n \in \N}$ be an expander sequence. Then whp 
\begin{align*}
T_{push}(G_n) = c_{push}(1) \log n + o(\log n).
\end{align*}
\end{theorem}
%In this work we will provide respective results for \pull and \pushpull, always in the presence of independent message transmission failures that occur with probability $1-q$.
%Moreover we will investigate a new kind of robustness of rumour spreading graphs: While previous results already investigated settings where independent message transmission failures and node failures occur (cf. e.g. \cite{feige1990randomized}), to our best knowledge this is the first time where the robustness of rumor spreading algorithms with respect to adversarial edge deletions is considered. In particular we consider the following problem: We know the runtimes of \push, \pull and \pushpull on graphs from expander sequences. Now we allow an adversary to delete up to a certain fraction of the edges at each node from these graphs. We investigate whether the adversary can slow down the rumour spreading process. Informally speaking, if such an increase of the runtime can only achieved by the deletion of many edges, we say the graph is robust. 
%
%We show that for $q=1$ \pull and \pushpull are robust but \push is not robust. However, \push can still inform almost all vertices fast. For $q\neq 1$ we show that $\pull$ is robust and that \push and \pushpull are not robust. However, \push and \pushpull can still inform almost all vertices fast. We provide precise formulations in subsection \ref{results_subsec}.
Apart from expander sequences, results in the form of Theorem \ref{push} (where the asymptotic runtime is determined) were also shown for sufficiently dense Erdös-Renyi random graphs  \cite{fountoulakis2010rumor}, random regular graphs \cite{fountoulakis2010reliable} as well as hypercubes \cite{Panagiotou2015}. Moreover, the order of the runtime on various models that describe social networks was investigated. In \cite{fountoulakis2012ultra} the Chung-Lu model was studied, \cite{doerr2011social} explored preferential attachment graphs and \cite{Friedrich2013} examined geometric graphs. A somewhat different approach is to derive general runtime bounds that hold for all graphs and depend only on some graph parameter, e.g.~conductance \cite{giakkoupis_tight_2011, chierichetti2018rumor}, vertex expansion \cite{giakkoupis_tight_2014} or diameter \cite{censor2012global, haeupler2015simple}. Furthermore, several variants of \push,\pull and \pushpull were studied. These include vertices being restricted to answer only one \pull request per round \cite{daum_rumor_2016}, vertices being allowed to contact multiple neighbours per round \cite{Panagiotou2015,doerr2017randomized}, vertices not calling the same neighbour twice  \cite{doerr2011social} and asynchronous versions \cite{boyd2006randomized,panagiotou2017asynchronous, acan2017push,angel2017string}.
Finally, besides \cite{doerr2017randomized}, robustness of these rumor spreading algorithms with respect to message transmission failures  was also studied by Elsässer and Sauerwald in \cite{elsasser2009runtime}. It was shown for any graph that if a message fails with probability $1-p$, then the runtime of \push increases at most by a factor of $6/p$.  

In this work our focus is on three subjects concerning the robustness of rumour spreading.  
Our first (and not unexpected) result  extends the validity of Theorem  \ref{push} to the runtimes of \pull and \pushpull.
In particular, we show that none of the three protocols slows down or speeds up on graphs with good expansion properties compared to its runtime on the complete graph. This motivates to investigate how severely a graph with good expansion properties has to be modified to increase the respective runtimes. 

In our second contribution, which is also the main result   and which differs from what was treated in previous works, we propose and study an unworn  approach to quantifying robustness. In particular, we investigate the impact of adversarial edge deletions, where we use the well-known concept of \emph{local resilience}, see e.g.~\cite{Sudakov2008, Dellamonica2008}. To be specific, we explore up to which fraction of edges an adversary needs to be allowed to delete at each vertex to slow down the process by a significant amount of time, i.e., by $\Omega(\log n)$ rounds. Here we discover a surprising dichotomy in the following sense. On the one hand, we show that both \pull and \pushpull  cannot be slowed down by such adversarial edge deletions -- in essentially all but trivial cases, where the fraction is so large that the graph may become (almost) disconnected. On the other hand, we demonstrate that even  a small number of edge deletions is sufficient to slow down \push by $\Omega(\log n)$ rounds. In other words, we find that in contrast to \pull and \pushpull, the \push protocol is not resilient to adversarial deletions and lacks (in this specific sense) the robustness of the other two protocols.
%First, we investigate the runtimes of \pull and \pushpull on expander sequences, i.e.~we prove that \pull and \pushpull are in the same way robust with respect to the network topology as this is the case for \push according to Theorem \ref{push}. 
%Apart from this, we turn our attention to two different important aspects regarding the robustness of all three models. 
%On the one hand, we explore the impact of message transmission failures that occur independently with probability $1-q\in [0,1)$.  
%On the other hand, we investigate the impact of adversarial edge deletions.
%To adress this latter issue we use the well known concept of local resilience (cf.~e.g.~\cite{Sudakov2008, Dellamonica2008}).
%Let $\mathcal G = (G_n)_{n \in \N} =((V_n,E_n))_{n \in \N}$ denote a sequence of graphs with $|V_n|=n$ and let $P$ denote a property of graphs. We say that $\mathcal G$ has local resilience $\alpha \in [0,1]$ with respect to the property $P$ if $\alpha$ is the infimum over all $\tilde \alpha$ such that for these $\tilde \alpha$ it holds that for all sequences of graphs $\tilde{\mathcal G}= (\tilde{G}_n)_{n \in \N}$ that arise from $\mathcal G$ by deleting edges of each $G_n$ and thus obtaining a graph $\tilde G_n$  such that each node keeps at least an $(1-\alpha)$ fraction of its edges, whp, $\tilde G_n$ has the property $P$. 

As our third subject, we generalise the previous results by additionally considering message transmission failures that occur independently with probability $1-q\in [0,1)$. On the positive side, we show that for arbitrary $q \in (0,1]$ all three algorithms inform \textit{almost} all vertices at least as fast as in an expander sequence in spite of adversarial edge deletions. However, if we want to inform all vertices, only \pull is not slowed down by adversarial edge deletions for all values of $q$;  \push can be slowed down as before; and \pushpull is a mixed bag, for $q=1$ it can not be slowed down, for $q<1$ it can. Furthermore, in general it is also possible to speed \pushpull up by deleting edges, which is however not surprising as the star-graph deterministically finishes in at most 2 rounds.  

Summarizing, this work enhances previous (robustness) results, particularly the ones concerning precise asymptotic runtimes and random transmission failures. Crucially, we introduce and study the concept of local resilience as a method to investigate robustness. However, apart from that, in this paper we develop two new general methods for the analysis of rumour spreading algorithms.
\begin{itemize}
\item \vspace{-2pt} The most common approach in the current literature for the study of the runtime is to determine the expected number of newly informed vertices in one or more rounds and to show concentration, for example by bounding the variance. Achieving this, however, is often quite complex and makes laborious and lengthy technical arguments necessary. Here we use the theory of \emph{self-bounding} functions, see Section~\ref{basiclemmas_sec}, that allows us to cleanly upper bound the variance by the \emph{expected value}. The argument works for all three investigated algorithms and the bound is valid for all graphs. We are certain that this method will also facilitate future work on the analysis of rumour spreading algorithms.
\item \vspace{-4pt} Studying the robustness of the protocols is a challenging task, as the adversary (as described previously) has various opportunities to modify the graph, for example by introducing a high variance in the degrees of the vertices; this turns out to be particularly problematic in the case of \pushpull. Here we demonstrate that such types of irregularities can be handled universally  by applying a powerful tool from a completely different area, namely extremal graph theory. In particular, we use  Szemer\'edi's regularity lemma (see e.g.~\cite{rodl2010regularity}), which allows us to partition the vertex set of a graph such that nearly all pairs of sets in the partition behave nearly like perfect regular bipartite graphs. This allows us to apply our methods on these regular pairs; eventually we obtain a linear recursion that can be solved by analysing the maximal eigenvalue of the underlying matrix.
\end{itemize}

\subsection{Results}
\label{results_subsec}

Our first result addresses the question about how fast rumours spread on expander graphs; in order to obtain a concise statement also the occurrence of independent message transmission failures is considered.

%For $q=1$ we take the limits of the above constants.
\begin{theorem} 
\label{fast_thm}
Let $\mathcal G=(G_n)_{n \in \N}$ be an expander sequence and let $q \in (0,1]$. Then whp
\begin{enumerate}[label={(\alph*)},ref={\thetheorem~(\alph*)}]
\itemsep0em 
\item $T_{push}(G_n,q) =c_{push}(q) \log n + o(log(n)),$\label{pushFast}
\item $T_{pull}(G_n,q) =c_{pull}(q) \log n + o(log(n)),$\label{pullFast}
\item $T_{pp}(G_n,q) =c_{pp}(q) \log n + o(log(n)).$\label{pushPullFast}
\end{enumerate}
\end{theorem}
 The first statement is an extension of Theorem \ref{push} and its proof is a straigthforward adaption of the proof in \cite{Panagiotou2015}. We omit it. The   contribution here is the proof of (b) and (c). Next we consider the case with edge deletions in addition to the message transmission failures.
\begin{theorem}
\label{robust_thm}
Let $0<\varepsilon <1/2, q\in (0,1]$ and $\mathcal G=(G_n)_{n \in \N}$ be an expander sequence. Let $\tilde{\mathcal G}=(\tilde G_n)_{n \in \N}$ be such that each $\tilde G_n$ is obtained by deleting edges of $G_n$ such that each vertex keeps at least a $(1/2 + \varepsilon)$ fraction of its edges. Then whp
\begin{enumerate}[label={(\alph*)},ref={\thetheorem~(\alph*)}]
\itemsep0em 
\item $T_{pull}(\tilde G_n,q)= c_{pull}(q)\log n + o(\log n).$ \label{pullIsRobust}
\item  $T_{pp}(\tilde G_n,1)\leq c_{pp}(1)\log n + o(\log n),$ when additionally assuming that $\delta(G_n)\geq \alpha n$ for some $0<\alpha\leq 1.$ \label{pushPullIsRobust}
\end{enumerate} 
\end{theorem}
This result demonstrates uncoditionally the robustness of \pull, and conditionally on $q=1$ the robustness of \pushpull on dense graphs, in the case of edge deletions, that is, the runtime is asymptotically the same as in the complete graph. It even shows that \pushpull may potentially profit from edge deletions in contrast to being slowed down. The proof of this result, especially the statement about \pushpull, is rather involved, since the original graph may become quite irregular after the edge deletions. Here we use, among many other ingredients, the aforementioned decomposition of the graph given by Szemeredi's regularity lemma.
%In fact we show something stronger: even if before each round $t\in \N$ the underlying graph in round $t$ is chosen as a subgraph of $G_n$ such that each vertex keeps at least a $(1/2 + \varepsilon)$ fraction of its edges, the conclusions of Theorem \ref{robust_thm} remain valid. In other words, the robustness results of Theorem \ref{robust_thm} remain true even if we consider an \emph{adaptive} adversary.

Note that Theorem \ref{robust_thm} does not consider \push and \pushpull (when $q\neq 1$) at all. Indeed, our next result states that in these cases the behaviour is rather different and that the algorithms may be  slowed down.
\begin{theorem}
\label{not_robust_thm}
Let $\varepsilon >0$ and $q \in (0,1]$. Then there is an expander sequence $\mathcal G=(G_n)_{n \in \N}$ and a sequence of graphs $\tilde{\mathcal G}=(\tilde G_n)_{n\in\N}$ with the following properties. Each $\tilde G_n$ is obtained by deleting edges of $G_n$ such that each vertex keeps at least a $(1-\varepsilon)$ fraction of its edges. Moreover, whp 
\begin{enumerate}[label={(\alph*)},ref={\thetheorem~(\alph*)}]
\itemsep0em 
\item $T_{push}(\tilde G_n,q)\geq c_{push}(q) \log n+\varepsilon/(2q)\log n + o(\log n).$ \label{lastPhasePushNotRobust}
\item% there exists $q>0$, such that 
 $T_{pp}(\tilde G_n,q)\geq c_{pp}(q)\log n +\left(\varepsilon/(8q)-\varepsilon q^3/5\right)\log n+o(\log n).$
 \label{pushPullIsNotRobustq} 
\end{enumerate}
\end{theorem}
Nevertheless, not all hope is lost. On the positive side, the next result states that \push and \pushpull  are able to inform \emph{almost} all vertices as fast as on the complete graph in spite of adversarial edge deletions. In this sense, we obtain an almost-robustness result for these cases.
\begin{theorem}
\label{first_robust_thm}
% \label{first_phase_as_fast_as_before}
% \label{fast_phase_remains_fast}
Let $0<\varepsilon <1/2, q\in (0,1]$ and $\mathcal G=(G_n)_{n \in \N}$ be an expander sequence. Let $\tilde{\mathcal G}=(\tilde G_n)_{n \in \N}$ be such that each $\tilde G_n$ is obtained by deleting edges of $G_n$ such that each vertex keeps at least a $(1/2 + \varepsilon)$ fraction of its edges. For ${\cal P}\in\{push, pp\}$ let $\tilde{T}_{\cal P}$ denote the number of rounds needed  to inform at least $n-n/\log n$ vertices. Then whp
\begin{enumerate}[label={(\alph*)},ref={\thetheorem~(\alph*)}]
\itemsep0em 
\item $\tilde{T}_{push}(\tilde G_n)= \log_{1+q}(n) + o(\log n).$ \label{firstPhasePushRobust} 
\item $\tilde{T}_{pp}(\tilde G_n)\leq  \log_{1+2q}(n) + o(\log n),$ when additionally assuming that $\delta(G_n)\geq \alpha n$ for some $0<\alpha\leq 1$. \label{firstPhasePushPullRobust}
\end{enumerate} 
\end{theorem}
We conjecture that there is also a version of Theorem \ref{firstPhasePushPullRobust} that is true for \pushpull on sparse graphs; to be precise we conjecture that in the setting of Theorem \ref{firstPhasePushPullRobust}  $\tilde{T}_{pp}(\tilde G_n)\leq\log_{1+2q}(n) + o(\log n),$
without further restrictions on $G_n$, i.e.~that \pushpull can not be slowed down informing \emph{almost} all vertices.
 
As a final remark note that Theorems \ref{robust_thm} and \ref{first_robust_thm} are tight in the sense that if an adversary is allowed to delete up to half of the edges at each vertex, then there are expander graphs that become disconnected such that their components have linear size. On those graphs a linear fraction of the vertices will remain uninformed forever. 

\paragraph{Outline}
The rest of this paper is structured as follows. In Section \ref{basiclemmas_sec} we collect and prove several important facts;  this part of the paper also contains our technical contribution concerning the analysis through self-bounding functions. In Subsection \ref{pull_fast_sec} we show that \pull is as fast on expanders with (or without) deleted edges as it is on the complete graph. Subsection \ref{push_pull_fast_subsec} treats \pushpull on expanders without deleted edges. In the remaining subsections we focus on the cases that may be slowed down by edge deletions.
In Subsection \ref{push_robust_sec} we show that adversarial edge deletions cannot slow down the time until \push has informed almost all vertices, by giving a coupling to the case without edge deletions. Contrary in Subsection \ref{push_slow_sec} we show that the time until \push has informed all vertices can be slowed down by edge deletions, even if only few edges are deleted.   Then, in Subsection \ref{push_pull_informs _almost_all nodes_fast} we show that \pushpull informs almost all vertices of dense graphs fast in spite of adversarial edge deletions. We utilize a version of Sz\'emeredis Regularity Lemma to get a well-behaved partition of the vertex set that is suitable for performing a round based analysis. However, if $q<1$, adversarial edge deletions can slow down or speed up the time until \pushpull has informed all vertices for nearly all values of $q$; we show this in Section \ref{edge_deletions_slow_down_push_pull_for_q_smaller_than_1}.  

\paragraph{Further Notation} Let $G = (V,E)$ denote a graph with vertex set $V$ and edge set $E\subseteq \binom{V}{2}$. Consider $v\in V$ and $U,W\subseteq V$ with $U\cap W = \emptyset$. We will denote the set of neighbours of $v$ in $G$ by $N_G(v)$ or by $N(v)$ and we will denote its degree by $d_G(v) := |N_G(v)|$ or by $d(v)$; $\delta_G$ or $\delta$ and $\Delta_G$ or $\Delta$ denote minimum and maximum degree of $G$. Similarly the neighbourhood of any set of vertices $S\subseteq V$ is defined by $N_G(S) := \cup_{v\in S}N_G(v)$. Furthermore let $E(U,W)=E_G(U,W)$ denote the set of edges with one vertex in $U$ and one vertex in $W$ and let $e(U,W):=e_G(U,W):=|E_G(U,W)|$.  With $E_G(U)$ we denote the set of edges with both vertices in $U$; $e_G(U)=|E_G(U)|$. 
For any round $t \in \N$ and ${\cal P}\in\{push, pull, pp\}$, we denote by $I_{t}^{({\cal P})}(G)$ the set of vertices of $G$ informed by \push, \pull and \pushpull respectively at the beginning of round $t$ and $|I^{({\cal P})}_1|=1$; if the underlying graph is clear from the context we will omit it; if we consider a sequence of graphs $\mathcal{G}=(G_n)_{n \in \N}$ and a sequence of times $t = (t(n))_{n\in \N}$, then $I_t^{({\cal P})}(\mathcal{G})=(I_{t(n)}^{({\cal P})}(G_n))_{n\in\mathbb{N}}$ is also a sequence.  Similarly, $U_{t}^{({\cal P})}:= V\backslash I_{t}^{({\cal P})}$ denotes the set of uninformed vertices. With $\log$ we refer to the natural logarithm. For any event $A$ we will write $\mathbb{E}_{t}[A]$ instead of $\mathbb{E}[A\big|I_t]$ for the conditional expectation and $P_{t}[A]$ instead of $P[A\big|I_t]$ for the conditional probability. Finally we want to clarify our use of Landau symbols. Let $a,b\in \mathbb{R}$ and $f$ be a function. The terms $a\leq b+o(f)$ and $a \geq b-o(f)$ mean that there exist positive functions $g,h\in o(f)$ such that  $a\leq b+g$ and $a \geq b-h$. Consequently $a=b+o(f)$ means that there exists a positive function $g\in o(f)$ such that $a\in [b-g,b+g]$
%For simplicity of exposition we ignore rounding issues.

\section{Tools \& Techniques}
%Usually, when rumour spreading algorithms are analysed, given the number of informed vertices in one round, one aims to compute the expected number of informed vertices in the next round and then one tries to show concentration around expectation. To obtain such concentration results, in a lot of previous work, the variance was bounded using long and complicated computations. Here we use a new approach that provides elegant and clean bounds for the variance. In particular we will use a concentration result for selfbounding functions (lemma \ref{lugosi}) which will allow us to bound the variance by the expectation, yielding the main result of this section, namely Lemma \ref{lugosiappl}.  
In this section we collect and prove statements about our protocols and properties of expander sequences. We begin with applying the previously mentioned notion of self-bounding functions to derive universal and simple-to-apply concentration results for our random variables, i.e., the number of informed vertices after a particular round. Then we extend the concentration results to more than one round. In the last part we recall the well known Expander Mixing Lemma and utilize it to derive properties (weak expansion, path enumeration) for the case where we delete edges from our graphs.
 
\paragraph{Self-bounding functions.} 
Our main technical new result in this section is the following bound on the variance for the number of informed vertices in any given round; it is true for any graph and any set of informed vertices.
\label{basiclemmas_sec}
\begin{lemma}\label{lugosiappl}
Let $G$ be a graph, $t \in \mathbb{N}$ and $I_t=I_t^{({\cal P})}(G)$ for ${\cal P} \in \{push, pull, pp\}$. Then
 $$\textrm{Var}\big[|I_{t+1}|\big|I_t\big]\leq \mathbb{E}\big[|I_{t+1}|\big|I_t\big].$$
\end{lemma}
Lemma \ref{lugosiappl} follows directly from Lemmas \ref{lugosi} and \ref{appllugsi}. Before stating them we introduce the notion of self-bounding functions.
\begin{definition}[Self-bounding function]
Let $X$ be a set and $m\in\mathbb{N}$. A non-negative function $f:X^m\rightarrow \mathbb{R}$ is self-bounding, if there exist functions $f_i:X^{m-1}\rightarrow \mathbb{R}$ such that for all $x_1, ..., x_m \in X$ and all $i=1, ...,m$ 
$$0\leq f(x_1,...,x_m)-f_i(x_1, ..., x_{i-1},x_{i+1},...,x_m)\leq 1$$
and
$$\sum\limits_{1\leq i \leq m}(f(x_1,...,x_m)-f_i(x_1, ..., x_{i-1},x_{i+1},...,x_m))\leq f(x_1,...,x_m).$$
\end{definition}
As striking property of self-bounding function is the following bound on the variance.
\begin{lemma}[\cite{Boucheron2004}]
\label{lugosi}
For a self-bounding function $f$ and independent random variables $X_1,...,X_m$, $m \in\mathbb{N}$
$$\textrm{Var}\left[f(X_1, ..., X_m)\right]\leq \mathbb{E}\left[f(X_1, ..., X_m)\right].$$
\end{lemma}
\begin{lemma}\label{appllugsi} 
 Let $G$ be a graph, $t\in \mathbb{N}$, and let $I_t=I_t^{({\cal P})}(G)$ for  ${\cal P} \in \{push, pull, pp\}$. Then, conditional on~$I_t$, there exist $m\in\N$, independent random variables $X_1, ...,X_m$ and a self-bounding function $f=f^{({\cal P})}$ such that 
$|I_{t+1}| = f(X_1,...,X_m).$
\end{lemma}
\begin{proof}
We will prove in detail the result for \push, and then we show what needs to be modified in order to obtain the statement in the case of \pull and \pushpull. Let $I_t=I_t^{(push)}$, $n\in \mathbb{N}$ be number of vertices of $G$ and $f:[n]^{|I_t|}\to \mathbb{R}$ with
$$(x_1, ..., x_{|I_t|}) \mapsto
 |I_t|+\sum\limits_{1\le k \le |I_t|} \mathbb{1}[x_{k}\in U_t]\mathbb{1}[\forall~\ell <k:x_{k}\neq x_{\ell}].$$
Moreover, let $(X_i)_{1 \le i \le |I_t|}$ be independent random variables, where $X_i$ is a uniformly random neighbour of the $i$th vertex -- according to an arbitrary ordering -- in $I_t$. We argue that $f(X_1, \dots, X_{|I_t|})=|I_{t+1}|$. Consider $v\in I_t$, then $v$ is counted by the $|I_t|$ term in $f$. For $v\in I_{t+1}\backslash I_t$ let $v_1, \dots ,v_s\in I_t, s\in \mathbb{N}$ be the informed vertices with random neighbour $v$ in round $t$, i.e. $X_{v_1}=\dots =X_{v_s}=v$ and $X_u \neq v$ for all other $u \in I_t$. Assume further that $v_1< v_2 <\dots < v_s$. For $k=v_1$ the term  $ \mathbb{1}[X_{k}\in U_t]\mathbb{1}[\forall~\ell <k:x_{k}\neq x_{\ell}]=1$ as $X_{v_1}=v\in U_t$ and for all $i\leq v_1$ it holds that $X_i\neq X_{v_i}$. For $k=v_r,~2\leq r\leq s$ the term $ \mathbb{1}[\forall~\ell <k:x_{k}\neq x_{\ell}]=0$ as $v_1<v_r$ and $X_{v_1}=X_{v_r}=v$. Thus every vertex $v\in I_{t+1}\backslash I_t$ is counted exactly once  by $f$.  Set further  
\begin{equation*}
f_{i}(x_1, ...,x_{i-1},x_{i+1},..., x_{|I_t|}) = |I_t|+\sum\limits_{k=1, k\neq i}^{|I_t|} \mathbb{1}[x_{k}\in U_t]\mathbb{1}[\forall~j <k,j\neq i:x_{j}\neq x_{k}], 
\quad 1\leq i\leq |I_t|.
\end{equation*}
The function $f_i$ arises from $f$ by leaving the $i$th variable out of consideration, i.e., the push of the $i$th vertex has no effect. Then by definition $f - f_i \in \{0,1\}$ for all $1 \le i \le |I_t|$, and actually we have
% Now we verify that $f$ has the self-bounding property. We have 
\begin{equation*}
\begin{aligned}
f-f_{i} = \mathbb{1}[x_{i}\in U_t]\mathbb{1}[\forall~ j\neq i:x_{i}\neq x_{j}].
\end{aligned}
\end{equation*}
%\begin{equation*}
%\begin{aligned}
%f-f_{i} = \mathbb{1}[x_{i}\in U_t]\bigg(\mathbb{1}[\forall~ j<i:x_{i}\neq x_{j}] -\mathbb{1}[\exists j>i:  (x_i=x_j)\land (\forall~\ell<i:x_j\neq x_\ell)]\bigg).
%\end{aligned}
%\end{equation*}
This quantity is precisely the difference in informed vertices after round $t$, assuming the $i$th vertex did not push.
%This difference is $0$ or $1$ as $\mathbb{1}[\exists j>i:  (x_i=x_j)\land (\forall~\ell<i:x_j\neq x_\ell)]=1$ implies that $\mathbb{1}[\forall~ j<i:x_{i}\neq x_{j}]=1$.
Furthermore 
\begin{align*}\sum\limits_{1 \leq i \leq |I_t|} (f-f_{i}) \leq 
\sum\limits_{1 \leq i \leq |I_t|}\mathbb{1}[x_{i}\in U_t]\mathbb{1}[\forall~ j\neq i:x_{i}\neq x_{j}] \leq f.
\end{align*}
Thus $f$ has the self-bounding property, which establishes the claim in the case of  \push. The proof for \pull is completely analogous, where we use 
\begin{align*}
f^{(pull)}:[n]^{|U_t|}&\to \mathbb{R}, ~~
(x_1, ..., x_{|U_t|}) \mapsto |I_t|+\sum\limits_{k \in U_t} \mathbb{1}[x_{k}\in I_t]
\end{align*}
and, similarly, for \pushpull we use $f^{(pp)}:[n]^{n} \to \mathbb{R}$ with
\begin{align*}
(x_1, ..., x_n) &\mapsto |I_t|+\sum\limits_{1 \leq k \leq n} \mathbb{1}[k\in I_t]\mathbb{1}[x_{k}\in U_t]\mathbb{1}[\forall ~ j\in \{1,\dots, k\}\cap I_t:x_{k}\neq x_{j}]\\&~~~~~+\sum\limits_{1 \leq k \leq n} \mathbb{1}[k\in U_t]\mathbb{1}[x_{k}\in I_t]\mathbb{1}[\forall ~w \in I_t: x_{w}\neq k].
\end{align*}
Here it is useful to see that the two sums in $f^{(pp)}$ are complementary, i.e.~that only one of the summands for index $k$ can be 1. Thus the functions $f^{(pull)}_i$ and $f_i^{(pp)}$ are obtained analogously to the push case.
\end{proof}
\begin{remark}
Let $G=(V,E)$ be a graph. Lemma \ref{appllugsi} also applies to subsets of $I_{t+1}$, i.e for any $U\subset V$ and conditioned on $I_t$ we have that $|I_{t+1}\cap U|$ and $|(I_{t+1}\cap U) \setminus I_t|$ are self-bounding.
\end{remark}
%\begin{definition}\label{def}
%Let $G$ be a connected graph with $n$ vertices, minimum degree $\delta$ and maximum degree $\Delta$. Let $\lambda_1,...,\lambda_n$ be the eigenvalues of the adjacency matrix of $G$ and set $\lambda = \max \{|\lambda_2|, |\lambda_n|\}$. We call $G$ an $(n,\delta,\Delta,\lambda)$-graph. 
%\end{definition}
The following lemma gives a tool   that we will use in order to extend our round-wise analysis  to longer phases. 
\begin{proposition}\label{concentration}
Let ${\cal P}\in\{push,pull,pp\}, I_t=I_t^{(\cal P)}$ and $t_1\geq t_0\geq 1$ such that $|I_{t_0}|\geq \sqrt{\log n}$. Let further $(\mathcal{A}_i)_{i\in \mathbb{N}}$ be a sequence of events, $c>1$, and $\delta>0$ such  that
$$P_{t_0}[\mathcal{A}_{t}\mid \mathcal{A}_{t_0}, \dots, \mathcal{A}_{t-1}]\geq 1-\delta\left(c^{t-t_0}  |I_{t_0}|\right)^{-1/3} \quad  \text{for all } t_0\le t\le t_1.$$
Then
$$P_{t_0}\left[\bigcap_{t=t_0}^{t_1}\mathcal{A}_t\right]\geq 1-O(|I_{t_0}|^{-1/3})$$
\end{proposition}
\begin{proof}
Using the definition of conditional probability we obtain, as $c>1$,
\begin{align*}
P_{t_0}\left[\bigcap_{t=t_0}^{t_1}\mathcal{A}_t\right]
&=\prod_{t=t_0}^{t_1}P_{t_0}\left[\mathcal{A}_t \mid\mathcal{A}_{t_0}, \dots, \mathcal{A}_{t-1} \right]\geq  \prod_{t=t_0}^{t_1} \left(1-(c^{t-t_0}\delta  |I_{t_0}|)^{-1/3}\right)\\&\geq 1-\sum_{t=t_0}^{t_1} \left(c^{t-t_0}\delta  |I_{t_0}|\right)^{-1/3}=1-|I_{t_0}|^{-1/3}\sum_{t=0}^{t_1-t_0} \delta ^{-1/3} c^{-t/3}= 1-O(|I_{t_0}|^{-1/3}).
\end{align*} 
\end{proof}
We give two typical example applications of this lemma below.
The first example addresses the case where we have a lower bound for the expected number of informed vertices after one round.
\begin{example}\label{concetrationRemark}~
Let $\mathcal{P}\in \{push,pull, pp\}, I_t=I_t^{(\cal P)}$. Assume that there is some $c>1$ such that $\mathbb{E}_t\left[\left|I_{t+1}\right|\right]\geq c\left|I_t\right|$ for all $t$ as long as ${n/f(n)}\leq|I_t|\leq n/g(n)$ for some functions $1\le f,g \le n, f=o(n)$. Let $t_0$ be such that $|I_{t_0}|\geq {n/f(n)}$. Then according to Lemma \ref{lugosiappl} we have that $\text{Var}_t\left[\left|I_{t+1}\right|\right]\leq \mathbb{E}_t\left[\left|I_{t+1}\right|\right]$ and applying Chebychev's inequality  gives
\begin{equation}\label{cheby}
%P_t\left[|I_{t+1}|\geq \left(1-\mathbb{E}_t\left[|I_{t+1}|\right]^{-1/3}\right)\mathbb{E}_t\left[|I_{t+1}|\right]\right]\geq 1-%\mathbb{E}_t\left[|I_{t+1}|\right]^{-1/3}\geq 1-|I_t|^{-1/3}.
P_t\left[\big||I_{t+1}|-\mathbb{E}_t\left[|I_{t+1}|\right]\big|\leq \mathbb{E}_t\left[|I_{t+1}|\right]^{2/3}\right]\geq 1-\mathbb{E}_t\left[|I_{t+1}|\right]^{-1/3}\geq 1-|I_t|^{-1/3}.
\end{equation}
Consider the events
\[
	\mathcal{A}_t = \text{``}|I_{t}| \ge \mathbb{E}_{t-1}\left[|I_{t}|\right]-  \mathbb{E}_{t-1}\left[|I_{t}|\right]^{2/3} ~~\text{or}~~ |I_t| \ge n/g(n)\text{''}
\]
The intersection of $\mathcal{A}_{t_0+1}, \dots, \mathcal{A}_{t}$ implies inductively that either $|I_{t}| \ge n/g(n)$ or
$$|I_{t}|\geq \left(1-\mathbb{E}_{t-1}[|I_{t}|]^{-1/3}\right) \mathbb{E}_{t-1}[|I_{t}|] \ge \left(1-(c|I_{t-1}|)^{-1/3}\right) c|I_{t-1}|\geq \left(\big(1-(c|I_{t_0}|)^{-1/3}\big) c\right)^{t-t_0}|I_{t_0}|.$$
We obtain with~\eqref{cheby}
\begin{align*}
P_{t_0}[\mathcal{A}_{t+1}\mid \mathcal{A}_{t_0+1}, \dots, \mathcal{A}_{t}, |I_t| < n/g(n)]
&\geq 1- \left(\big(1-(c|I_{t_0}|)^{-1/3}\big) c\right)^{-(t-t_0)/3}|I_{t_0}|^{-1/3},
\end{align*}
and otherwise $P_{t_0}[\mathcal{A}_{t+1}\mid \mathcal{A}_{t_0+1}, \dots, \mathcal{A}_{t}, |I_t| \ge n/g(n)] = 1$.
%thus for all $t$
%\begin{equation}
%\label{eq:lbAt+1}
%P_{t_0}[\mathcal{A}_{t+1}\mid \mathcal{A}_{t_0+1}, \dots, \mathcal{A}_{t}]
%\geq 1-\left(\big(1-(c|I_{t_0}|)^{-1/3}\big) c\right)^{-(t-t_0)/3}|I_{t_0}|^{-1/3}.
%\end{equation}
Choose $\tau := t - t_0 = \log_c(f(n)/g(n)) + o(\log n)$ as small as possible such that this lower bound for $|I_{t+1}|$ is $\ge n/g(n)$, that is, this lower bound is $< n/g(n)$ for $t=t_0+\tau$. %From~\eqref{eq:lbAt+1}
Combining the two conditional probabilities we obtain for all $t_0 \le t \le t_0 + \tau$ 
\begin{equation*}
%\label{eq:lbAt+1}
P_{t_0}[\mathcal{A}_{t+1}\mid \mathcal{A}_{t_0+1}, \dots, \mathcal{A}_{t}]
\geq  1-  \left(\big(1-(c|I_{t_0}|)^{-1/3}\big) c\right)^{-(t-t_0)/3}|I_{t_0}|^{-1/3}.
\end{equation*}
Applying Proposition~\ref{concentration} then yields whp
\[
	|I_{t_0 + \tau+1}| \ge n/g(n).
\]
\end{example}
In the second example we make the stronger assumption that we can determine asymptotically the expected number of informed vertices after one round. Here we assume that we begin with a ``small'' set of informed vertices, say of size $\sqrt{\log n}$, and want to reach a set of size nearly linear in $n$.
\begin{example}\label{concetrationExample}
Assume that there is some $c>1$ such that $\mathbb{E}_t\left[\left|I_{t+1}\right|\right]= (1+o(1))c\left|I_t\right|$ for all $t$ as long as
 $\sqrt{\log n}\leq|I_t|\leq  n/ {\log n}$.
Let $\mathcal{A}_t$ be the event ``$\left||I_{t}|-\mathbb{E}_{t-1}\left[|I_{t}|\right]\right|\leq \mathbb{E}_{t-1}\left[|I_{t}|\right]^{2/3}$'' and let $t_0$ be such that $|I_{t_0}|\geq \sqrt{\log n}$. There is $h(n)\in o(1)$ such that for $c^-:= (1-h(n))c$ and $c^+:=(1+h(n))c$ we have that $\mathbb{E}_t\left[\left|I_{t+1}\right|\right]\leq c^+\left|I_t\right|$ and $\mathbb{E}_t\left[\left|I_{t+1}\right|\right]\geq c^-\left|I_t\right|$. Using this notation, the events $\mathcal{A}_{t_0+1}, \dots, \mathcal{A}_{t+1}$ imply together inductively that
$$|I_{t+1}|\leq \left(1+\mathbb{E}_t[|I_{t+1}|]^{-1/3}\right) \mathbb{E}_t[|I_{t+1}|] \le \left(1+(c^-|I_t|)^{-1/3}\right) c^+|I_t|\leq \left(\big(1+(c^-|I_{t_0}|)^{-1/3}\big) c^+\right)^{t-t_0}|I_{t_0}|$$
for all $t$ such that the right-hand side is bounded by $n/\log n$. Moreover, for all such $t$
$$|I_{t+1}|\geq \left(1-\mathbb{E}_t[|I_{t+1}|]^{-1/3}\right) \mathbb{E}_t[|I_{t+1}|] \ge \left(1-(c^-|I_t|)^{-1/3}\right) c^-|I_t|\geq \left(\big(1-(c^-|I_{t_0}|)^{-1/3}\big) c^-\right)^{t-t_0}|I_{t_0}|.$$ 
Thus, as $\mathcal{A}_{t}$ only depends on $I_t$ it follows with \eqref{cheby}
\begin{align*}
P_{t_0}[\mathcal{A}_{t+1}\mid \mathcal{A}_{t_0+1}, \dots, \mathcal{A}_{t}]
%&=P_{t_0}\left[\mathcal{A}_t\mid \left(\big(1-(c|I_{t_0}|)^{-1/3}\big) c\right)^{t-t_0}|I_{t_0}|\leq|I_{t}|\leq \left(\big(1+(c|I_{t_0}|)^{-1/3}\big) c\right)^{t-t_0}|I_{t_0}|\right]\\
&\geq 1- \left(\big(1-(c^-|I_{t_0}|)^{-1/3}\big) c^-\right)^{-(t-t_0)/3}|I_{t_0}|^{-1/3}.
\end{align*}
Applying Proposition~\ref{concentration} then immediately gives that there is $\tau_1= \log_c (n/|I_{t_0}|) + o(\log n)$  such that whp
%Choosing
%$$\tau_1= \frac{\log (n/(|I_{t_0}|\log n) )}{\log\left(\left(1-(c^-|I_{t_0}|)^{-1/3}\right) %c^-\right)} \leq \log_c n+o(\log n),
%~~
%\tau_2= \frac{\log (n/(|I_{t_0}|\log n) )}{\log\left(\left(1+(c^-|I_{t_0}|)^{-1/3}\right) c^+\right)}\geq \log_c n-o(\log n) $$
%gives that whp
$ |I_{t_0+\tau_1}|\leq n/\log n.$
Example \ref{concetrationRemark}, setting $f=n/\sqrt{\log n}$ and $g=\log n$, gives an additional $\tau_2= \log_c (n/|I_{t_0}|) + o(\log n)$ such that $|\tau_1 - \tau_2| = o(\log n)$ and whp
$$|I_{t_0+\tau_1}|\leq \frac{n}{\log n}\leq |I_{t_0+\tau_2}|.$$
\end{example}
\paragraph{Expander Sequences.} In this section we collect some important properties of expander sequences that we are going to use later. We start by stating a version of the well-known expander mixing lemma applied to our setting of expander sequences.
\begin{lemma}[{\cite[Cor.~2.4]{Panagiotou2015}}]
\label{expmix}
Let $\mathcal G=(G_n)_{n \in \N}=((V_n,E_n))_{n \in \N}$ be an expander sequence. Then for $S_n\subseteq V_n$ such that $1 \leq |S_n| \leq n/2$ it is
\begin{align*}
\bigg |e(S_n,V_n\backslash S_n) - \frac{\Delta_n |S_n| (n- |S_n|)}{n}\bigg | = o(\Delta_n) |S_n|.
\end{align*}
\end{lemma}
The following result is a consequence of the Expander Mixing Lemma that applies to graphs in which some edges were removed. It seems very simple but it turns out to be surprisingly useful.
%It will be used to show that phases running very fast (in $o(\log n)$ rounds) stay very fast in spite of edge deletions.
\begin{lemma}
\label{linearFraction}
Let $\mathcal G=(G_n)_{n \in \N}=((V_n,E_n))_{n \in \N}$ be an expander sequence. Let $\varepsilon>0$ and set $\tilde{\mathcal G}=(\tilde G_n)_{n\in \N}$, where each $\tilde G_n$ it is obtained from $G_n$ by deleting edges such that each vertex keeps at least a $(1/2 + \varepsilon)$ fraction of its edges. For each $n \in \N$ let further $S_n\subseteq V_n$, then there is $n_0 \in \N$  such that for all $n \geq n_0$
\begin{align*}
e_{\tilde G_n}(S_n,V_n\backslash S_n)\geq \varepsilon e_{G_n}(S_n,V_n \backslash S_n).
\end{align*}
\end{lemma}
\begin{proof}
Without loss of generality we assume that $|S_n|\leq n/2$. Since at most $(1/2-\varepsilon)\Delta_n$ edges are deleted at each vertex, we immediately obtain that 
\begin{align*}
 e_{\tilde G_n}(S_n,V_n \backslash S_n)\geq e_{G_n}(S_n,V_n \backslash S_n) - \Delta_n \left(1/2 - \varepsilon\right)|S_n|.
\end{align*}
Using Lemma \ref{expmix} and choosing $n_0$ large enough such that $\frac{o(\Delta_n)}{\Delta_n} \frac{n}{n-|S_n|}<\varepsilon$ for all $n \ge n_0$, we obtain that
\begin{align*}
(1-\varepsilon) e_{G_n}(S_n,V_N \backslash S_n) - \Delta_n \left(1/2 - \varepsilon\right)|S_n|
~&\geq (1-\varepsilon)\frac{\Delta_{n} |S_n|(n-|S_n|)}{n} - o(\Delta_n)|S_n| -  \Delta_n \left(1/2 - \varepsilon\right)|S_n| \\
&= \frac{\Delta_n |S_n|(n-|S_n|)}{n} \left (1-\varepsilon- \frac{o(\Delta_n)}{\Delta_n} \frac{n}{n-|S_n|}-\frac{n(1/2-\varepsilon)}{n-|S_n|}  \right ).
\end{align*}
As $n-|S_n| \geq n/2$ the last expression is $>0$. Hence
\begin{align*}
e_{\tilde G_n}(S_n,V_n \backslash S_n)\geq \varepsilon e_G(S_n,V_n \backslash S_n) +(1-\varepsilon) e_G(S_n,V_n \backslash S_n)- \Delta_n \left(1/2 - \varepsilon\right)|S_n| \geq \varepsilon e_{G_n}(S_n,V_n \backslash S_n).
\end{align*}
\end{proof}
Next we give a lemma that counts the number of paths between two arbitrary vertices of a dense graph satisfying a weak expander property (as for example guaranteed by Lemma \ref{linearFraction}). This will later be used to give a lower bound on the probability of any vertex to be informed after a given constant number of rounds.
\begin{lemma}\label{CountingPaths}
Let $G=(V,E), |V|=n.$ Assume that there is $\alpha>0$ such that $d(v)\geq \alpha n$ for all $v\in V$ and  $e(W,V\backslash W)\geq \alpha |W||V\backslash W|$ for all $W\subseteq V$. Then for all $u,w \in V$ there is $1\le d\leq 8/\alpha^2 +2$ such that there are at least $ (\alpha^4/64)^{d+1} n^{d-1}$ paths of length $d$ from $u$ to $w$. 
\end{lemma}
\begin{proof}
Assume $\alpha\leq 1/2$ as otherwise the claim is trivial (with $d \in\{1,2\}$). We define sequences $(U_i)_{i\in\mathbb{N}}$ and $(H_i)_{i\in\mathbb{N}}\subseteq V$ as follows. Set $U_1=\{u\}\cup N(u), W=\{w\}\cup N(w)$ and $H_1=V\backslash (U_1\cup W)$ and proceed for $i\ge 1$ as follows. Let $\tilde U_{i+1}\subseteq H_i$ be the set of vertices $v\in H_i$ with $|N(v)\cap U_i|\geq \alpha^2 n/8$. Set $U_{i+1}=U_i\cup \tilde{U}_{i+1}$ and $H_{i+1}=H_i\backslash \tilde U_{i+1}$. Then we claim that for all $i \ge 1$
\begin{equation}
\label{eq:manyEdgesORlargeIncrease}
%\textrm{either}\quad 
e(U_i,W)\geq \alpha^3 n^2/2 \quad\textrm{or}\quad
|U_{i+1}|  \ge |U_i|+\alpha^2 n/8 . %\textrm{ and } |H_i| \geq \alpha n/2.
\end{equation}
To see this, assume that $e(U_i,W)\leq \alpha^3 n^2/2 $; since $|U_i|, |W| \ge \alpha n$, the weak expansion property guarantees that
$$e(U_i,H_i)=e(U_i,H_i\cup W)-e(U_i,W)\geq \alpha|U_i||H_i\cup W|-\alpha^3n^2/2\geq \alpha^2(1-\alpha)n^2-\alpha^3n^2/2,$$
and using $\alpha \leq 1/2$ we obtain that
$e(U_i,H_i)\geq \alpha^2n^2/4.$
%\end{equation*}
To complete the proof of~\eqref{eq:manyEdgesORlargeIncrease} we compute the size of $\tilde U_{i+1}$. As $|N(v)\cap U_i|\leq \alpha^2 n/8$ for all $v\in H_i\backslash \tilde{U}_{i+1}$ and $|N(v)\cap U_i|\le n$ we get
\begin{align*}
\frac{\alpha^2 n^2}{4}\leq e(U_i,H_i)\le|\tilde{U}_{i+1}| n + |H_i|\frac{\alpha^2 n}{8}.
\end{align*}
Since $|H_i| \le n$ we immediately get that $|\tilde{U}_{i+1}| \ge \alpha^2 n /8$, which shows~\eqref{eq:manyEdgesORlargeIncrease}.
We next show that there are (sufficiently) many paths for each vertex in $U_i$ to $u$. More precisely, let $1\leq j\leq 8/\alpha^2$ be such that $e(U_i,W) < \alpha^3 n^2/2 $ for all $1\leq i\leq j$. For those $i$ we  have by~\eqref{eq:manyEdgesORlargeIncrease} that $|U_i|\geq i\cdot\alpha^2 n/8$. We claim that for all $v\in U_i \setminus\{u\}$ there is $d\leq i$ such that $v$ has at least $(\alpha^4/64)^d \cdot n^{d-1}$ paths of length $d$ with endpoint $u$.
We show the claim by induction on $i$. The base case $v \in U_1 \setminus  \{u\}$ is clear, as $1\geq \alpha^4/64$. For the induction step assume that $v\in U_{i+1}\backslash U_i$, $v\neq u$. Then by construction $|N(v)\cap U_i|\geq \alpha^2 n/8$. Thus by induction hypothesis there is $d\leq i$ such that $v$ has at least $\alpha^2 n/(8i)$ neighbours with at least $(\alpha^4/64)^d n^{d-1}$ paths with endpoint $u$. As $i\leq 8/\alpha^2$ this gives that $v$ has at least $\alpha^2 n/(8i)\cdot (\alpha^4/64)^d n^{d-1} \geq (\alpha^4/64)^{d+1} n^{d}$ paths of length $d+1\leq i+1$ with endpoint $u$, and this accomplishes the induction step.
With all these facts at hand we finally show the claim of the lemma. Let $j\leq 8/\alpha^2$ be the first index such that $e(U_j,W)\geq \alpha^3 n^2/2$ and let $W'\subseteq W$ be such that $|N(v)\cap U_j|\geq \alpha^3 n/4$ for all $v\in W'$. Thus
\begin{align*}
\frac{\alpha^3 n^2 }{2}\leq e(U_j,W)\leq |W'| n + |W| \frac{\alpha^3 n}4,
\end{align*}
and thus $|W'|\geq \alpha^3 n/4$. Then there is $d\leq j$ and $W''\subseteq W'$ such that $|W''|\geq |W'|/j$ and every $v$ in $W''$ has at least $\alpha^3 n/(4j)$ neighbours with at least $(\alpha^4 /64)^{d} n^{d-1}$ paths of length $d$ with endpoint $u$. Therefore every $v\in W''$ hast at least $(\alpha^4 /64)^{d} n^{d-1} \cdot \alpha^3n /(4j) \ge  (\alpha^4 /64)^{d+1} n^{d}/j $ paths of length $d+1$ with endpoint $u$. This in turn gives that there are at least
$|W'|/j \cdot  (\alpha^4 /64)^{d+1} n^{d}/j
\geq {\alpha^3}/4 \cdot (\alpha^4 /64)^{d+2} n^{d+1}$
paths of length $d+2$ from $w$ to $u$, and the proof is completed.
\end{proof}
Next comes a technical lemma that given a small set quantifies the number of vertices for which only a small fraction of their neighbourhood intersects that given set. 
\begin{lemma}\label{fraction}
Let $\mathcal G=(G_n)_{n \in \N}=((V_n,E_n))_{n \in \N}$ be an expander sequence. Let $\varepsilon>0$ and let $\tilde{\mathcal G}=(\tilde G_n)_{n \in \N}$, where   each $\tilde G_n$ it is obtained from $G_n$ by deleting edges such that each vertex keeps at least a $(1/2 + \varepsilon)$ fraction of its edges. Let further $A_n\subseteq V_n$ with $|A_n| = o(n)$.
\begin{enumerate}[label={(\alph*)},ref={\thetheorem~(\alph*)}]
\item \label{fraction_a} There is $B_n\subseteq A_n$ with $|B_n|=(1-o(1))|A_n|$  such that for all $u\in B_n$
\begin{equation*}
\frac{|N_{\tilde G_n}(u)\cap A_n|}{|N_{\tilde G_n}(u)|} = o(1).
\end{equation*}\item \label{fraction_b} There is $ B_n\subseteq V_n\setminus A_n$ with $|V_n\setminus (A_n\cup B_n)|=o(| A_n|)$  such that for all $v\in B_n$
\begin{equation*}
\frac{|N_{\tilde G_n}(v)\cap A_n|}{|N_{\tilde G_n}(v)|} =  o(1).
\end{equation*}
\end{enumerate}
\end{lemma}
\begin{proof}
Let $\delta_n, \Delta_n$ denote the minimum and maximum degree of $G_n$. Lemma \ref{expmix} yields that
\begin{align*}
	e_{G_n}(A_n,V_n\setminus A_n)
	= \frac{\Delta_n |A_n||V_n\setminus A_n|}{n}+o(\Delta_n)|A_n|
	%=\Delta_n|A_n|\left(\frac{|V_n\setminus A_n|}{n}+o(1)\right)
	=(1+o(1))\Delta_n |A_n|.
\end{align*}
As there are a maximum of $\Delta_n |A_n|$ edges with at least one point in $A_n$, we get that 
$e_{G_n}(A_n)=o(\Delta_n)|A_n|$.
Since we obtain $\tilde G_n$ from $G_n$ by deleting edges 
%With Lemma \ref{linearFraction} we know that between the two sets $U$ and $U^c$ there remain at least a constant fraction of their edges after edge deletions. Thus we have
\begin{equation}
\label{eq:etgn}
e_{\tilde G_n}(A_n)=o(\Delta_n)|A_n|.
\end{equation}
With this fact at hand we show a). Let $\eta > 0$ and call a vertex $u\in A_n$ bad if $|N_{\tilde G_n}(u) \cap A_n| \ge \eta |N_{\tilde G_n}(u)|$. Since $N_{\tilde G_n}(u) \ge \delta_n/2$ we obtain for any bad $u$ that $|N_{\tilde G_n}(u) \cap A_n| \ge \eta \delta_n/2$. As  $\delta_n = (1- o(1))\Delta_n$ we infer from~\eqref{eq:etgn} that the number of bad vertices is $o(|A_n|)$.

To see the b) let again $\eta > 0$ and call this time a vertex $v\in V_n\setminus A_n$ bad if $|N_{\tilde G_n}(v) \cap A_n| \ge \eta |N_{\tilde G_n}(v)|$. Then for any such bad $v$ we know that $|N_{\tilde G_n}(v) \cap A_n| \ge \eta \delta_n/2$. As before, using~\eqref{eq:etgn} we readily get that the number of bad $v$'s is $o(|A_n|)$.
\end{proof}
We conclude our preparational section by giving a lemma that  bounds crudely the time needed until at least $\omega(1)$ vertices are informed. %This is needed as concentration using self-bounded functions can not be applied in this range.
\begin{lemma}\label{startup}
Let $0<\varepsilon \leq 1/2, q\in (0,1]$ and $\mathcal G=(G_n)_{n \in \N}$ be an expander sequence. Let $\tilde{\mathcal G}=(\tilde G_n)_{n \in \N}$ be such that each $\tilde G_n$ is obtained by deleting edges of $G_n$ such that each vertex keeps at least a $(1/2 + \varepsilon)$ fraction of its edges. Let further ${\cal P} \in \{push,pull,pp\}$ and suppose that $ |I_t|<\sqrt{\log n}$. Then there is $\tau=o(\log n)$ such that whp $|I_{t+\tau}^{(\cal P)}|\geq\sqrt{\log n}$. 
\end{lemma}
\begin{proof}
Recall that the probability that  $v\in U_t$ gets informed by \pull  is $q|N(v)\cap I_t|/|N(v)|$. Thus
\begin{align*}
P_t[|I_{t+1}^{(pull)}\backslash I_t| = 0] = \prod\limits_{u\in N(I_t)\cap U_t}\left(1-\frac{q|N(u)\cap I_t)|}{|N(u)|}\right)
%\geq 1-\left(1-\frac{q}{\Delta_n}\right)^{|N(I_t)\cap U_t|}
\leq e^{-qe(U_t,I_t)/\Delta_n}.
\end{align*}
Similarly we obtain for  \push
\begin{align*}
P_t[|I_{t+1}^{(push)}\backslash I_t| = 0] = \prod_{v \in I_t} \frac{|N(v) \cap I_t|}{|N(v)|} 
= \prod_{v \in I_t} \left(1-\frac{|N(v) \cap U_t|}{|N(v)|}\right)
\le e^{-qe(I_t,U_t)/\Delta_n}.
\end{align*}
The same bound is obviously also true for \pushpull. Thus, for all ${\cal P}\in \{push,pull,pp\}$  
$$ P_t[|I_{t+1}^{(\cal P)}\backslash I_t| \geq 1]\geq 1-e^{-qe(U_t,I_t)/\Delta_n}.$$
As Lemma \ref{expmix} and Lemma \ref{linearFraction} imply that $e(U_t,I_t)\geq (1+o(1))\varepsilon\Delta_n|I_t|$, there is $c\in(0,1)$ such that $P[|I_{t+1}^{(\cal P)}\backslash I_t|\geq 1] >c$.
Define $\tau := \lceil 2\sqrt{\log n} \rceil$ and $X=\text{Bin}(\tau, c)$ with $\mathbb{E}[X]=c\tau$ and $\textrm{Var}[X]=\tau (1-c)c$. Then, using Chebyshev
\begin{align*}
P_t\left[|I_{t+\tau}^{(\cal P)}| \leq \sqrt{\log n}\right]\leq P_t\left[X\leq \sqrt{\log n}\right]\leq  P_t\left[|X-\mathbb{E}[X]| \leq \mathbb{E}[X]/2\right]\leq {4\textrm{Var}[X]}/{\mathbb{E}[X]^2}=o(1).
\end{align*}
\end{proof}
\section{Proofs}

\subsection{Proof of Theorems \ref{pullFast}, \ref{pullIsRobust} --- edge deletions do not slow down \pull}\label{pull_fast_sec}

Let $0 < \varepsilon \le 1/2$. In this section we study the runtime of \pull in the case in which the input graph is an expander, and where at each vertex at most an $(1/2 - \varepsilon)$ fraction of the edges is deleted. The runtime on expander sequences without edge deletions, that is, the setting in Theorem~\ref{pullFast}, is included as the special case where we set $\varepsilon=1/2$. In contrast to previous proofs, in the analysis of \pull the `standard' approach that consists of showing, for example, that $\mathbb{E}_t[|I_{t+1} \setminus I_t|] \approx |I_t|$ fails. The main reason is that the graph between $I_t$ and $U_t$ might be quite irregular, so that, depending on the actual state, $\mathbb{E}_t[|I_{t+1} \setminus I_t|] \approx c|I_t|$ for some $c < 1$. However, we discover a different invariant that is preserved, namely that the number of edges between $I_t$ and $U_t$ behaves in an exponential way. With Lemmas~\ref{expmix} and \ref{linearFraction} we can then relate this to the number of informed vertices.
\begin{lemma}\label{expResPull}
Consider the setting of Theorem \ref{pullIsRobust} and let $I_t=I_t^{(pull)}$.
\begin{enumerate}[label={(\alph*)},ref={\thetheorem~(\alph*)}]
\itemsep0em 
\item \label{expResPull_a}Let  $\sqrt{\log n}\leq |I_t|\leq n/\log n$. Then 
$|e(U_{t+1},I_{t+1}) - (1+q)e(U_t,I_t)| \le |I_t|^{-1/3}e(U_t,I_t)$ with probability at least  $1-O(|I_t|^{-1/3})$.
\item \label{expResPull_b}Let $|U_t|\leq n/\log n$. Then
$\mathbb{E}_{t}[|U_{t+1}|] = (1-q+o(1))|U_t|. $
\end{enumerate}
\end{lemma}
\begin{proof}
We start with $a)$. Let $D_t=e(U_{t+1},I_{t+1})-e(U_{t},I_{t})$ and for $u\in U_t$ let $X_u$ be the random variable that indicates whether $u$ gets informed in round $t+1$. Then 
\begin{align*}
\mathbb{E}_t[D_t] &= \sum\limits_{u \in U_t}\sum\limits_{v \in N(u)\cap U_t}\mathbb{E}_t[X_u(1-X_v)]-\sum\limits_{u \in U_t}\mathbb{E}_t[X_u]\cdot |N(u)\cap I_t|\\& =  
\sum\limits_{u \in U_t}q\frac{|N(u)\cap I_t|}{|N(u)|}\left(\left(\sum\limits_{v \in N(u)\cap U_t}1-q\frac{|N(v)\cap I_t|}{|N(v)|}\right)- |N(u)\cap I_t|\right).
%\\
%& =  
%\sum\limits_{u \in U_t}q|N(u)\cap I_t|\left(\left(\sum\limits_{v \in N(u)\cap U_t}\frac{1}{|N(u)|}-q\frac{|N(v)\cap I_t|}{|N(v)||N(u)|}\right)- \frac{|N(u)\cap I_t|}{|N(u)|}\right).
\end{align*}
The second sum is at most $|N(u)|$, so obviously $\mathbb{E}_t[D_t] \le qe(U_t,I_t)$. To get a lower bound consider a largest set $\tilde U\subseteq U_t$  such that $|N(u)\cap I_t|/|N(u)|=o(1)$ for all $u\in \tilde U$.  From Lemma \ref{fraction_b} we obtain that $|U_t\setminus \tilde U|=o(|I_t|)$, and so
\begin{align*}
\mathbb{E}_t[D_t] &\geq
\sum\limits_{u \in U_t}q|N(u)\cap I_t|\left(\left(\sum\limits_{v \in N(u)\cap \tilde U}\frac{1}{|N(u)|}-o\left(\frac{1}{|N(u)|}\right)\right)-\frac{|N(u)\cap I_t|}{|N(u)|}\right).
\end{align*}
Consider furthermore $\hat U \subseteq \tilde U$ such that $|N(u)\cap \tilde U|/|N(u)|=1-o(1)$ and thus also  $|N(u)\cap I_t|/|N(u)|=o(1)$ for all $u\in \hat U$. Lemma \ref{fraction_b} again yields that we can choose $\hat U$ such that $|U_t\backslash \hat U|=o(|I_t|)$, thus
\begin{align*}
\mathbb{E}_t[D_t]&\geq (1-o(1))
\sum\limits_{u \in \hat U}q|N(u)\cap I_t|\left(\frac{|N(u)\cap \tilde U|}{|N(u)|}-\frac{|N(u)\cap I_t|}{|N(u)|}\right)-\sum_{u\in U_t\setminus \hat U}{|N(u)\cap I_t|}\\&\geq (q-o(1))e(U_t,I_t)-2e(U_t\backslash\hat{U},I_t).
\end{align*}
According to Lemmas  \ref{expmix} and \ref{linearFraction} we have that $e(U_t,I_t)=\Theta(|I_t|\Delta_n)$. But $e(U_t\backslash\tilde{U},I_t)\leq |U_t\backslash \tilde{U}|\Delta_n=o(|I_t|\Delta_n)$. Thus, $\mathbb{E}_t[e(U_{t+1},I_{t+1})]=(1+q-o(1))e(U_t,I_t)$.
In the next step we bound the variance. For each edge $e$ let $X_e$ be the indicator random variable that denotes the events that $e\in E(U_{t+1},I_{t+1})$. Thus 
$$e(U_{t+1},I_{t+1})=\sum\limits_{e\in E}X_e=\frac{1}{2}\sum\limits_{u\in V}\sum\limits_{v\in N(u)}X_{\{u,v\}}.$$
Using that $X_e$ and $X_{e'}$ are independent for all $e,e'\in E$ with $e\cap e' =\emptyset$,
\begin{align*}
 \textrm{Var}[e(U_{t+1},I_{t+1})] &= \textrm{Var}\left[\sum\limits_{e\in E}X_e\right] =\sum\limits_{e,e'\in E}\mathbb{E}[X_eX_{e'}]-\mathbb{E}[X_e]\mathbb{E}[X_{e'}]\\&\leq \sum\limits_{u\in V}\sum\limits_{v,v'\in N(u)}\mathbb{E}[X_{\{u,v\}}X_{\{u,v'\}}]\leq \Delta_n\sum\limits_{u\in V}\sum\limits_{v\in N(u)}\mathbb{E}[X_{\{u,v\}}]= 2\Delta_n \mathbb{E}[e(U_{t+1},I_{t+1})].
\end{align*}
Since $\mathbb{E}_t[e(U_{t+1},I_{t+1})] = (1 + q - o(1))e(U_t,I_t) =   \Theta(\Delta_n |I_t| )$ by Lemmas \ref{expmix} and \ref{linearFraction} and Var$[e(U_{t+1},I_{t+1})] \le 2\Delta_n \mathbb{E}_t[e(U_{t+1},I_{t+1})]$ we obtain for $|I_t|\geq \sqrt{\log n}$ with Cheybyshev's inequality immediately that 
\begin{equation*}
\begin{aligned}
&P\left[|e(U_{t+1},I_{t+1})-\mathbb{E}_t[e(U_{t+1},I_{t+1})]|\geq e(U_t,I_t)|I_t|^{-1/3}\right] \leq O(|I_t|^{-1/3}).
\end{aligned}
\end{equation*}
Next we show $b)$. 
We bound the expected number of uninformed vertices after one additional round. Lemma \ref{fraction_a}  asserts that there is a set $\tilde U\subseteq U_t$ such that $|\tilde U| = (1-o(1))|U_t|$ and $|N(u)\cap I_t|/|N(u)|=1-o(1)$ for all $u\in \tilde{U}$. Thus,
\begin{equation*}
\mathbb{E}_{t}[|U_{t+1}|]=
\sum\limits_{u\in U_t} 1-q\frac{|N(u)\cap I_t|}{|N(u)|} \leq|U_t|-q \sum\limits_{u\in \tilde U}\frac{|N(u)\cap I_t|}{|N(u)|}=  |U_t|-q(1-o(1))|\tilde U|
= \left(1-q -o(1)\right) |U_t|.
\end{equation*}
As $|N(u)\cap I_t|\leq |N(u)|$ we also have
\begin{equation*}
\mathbb{E}_{t}[|U_{t+1}|]=
\sum\limits_{u\in U_t} 1-q\frac{|N(u)\cap I_t|}{|N(u)|} \geq \sum\limits_{u\in U_t} (1-q)=\left(1-q\right) |U_t|.
\end{equation*}
\end{proof}
Lemma \ref{pullResUpper} and Lemma \ref{startup} give lower bounds, that together with an upper bound provided by Lemma \ref{pullLower} imply Theorems \ref{pullFast} and \ref{pullIsRobust}.
\begin{lemma}[Upper bound in Theorem \ref{pullIsRobust}]\label{pullResUpper}
Consider the setting of Theorem \ref{pullIsRobust} and let $I_t=I_t^{(pull)}$, then the following statements hold whp.
\begin{enumerate}[label={(\alph*)},ref={\thetheorem~(\alph*)}]
\itemsep0em 
%\item Let $1\leq|I_t|$. Then there is $\tau=o(\log n)$ such that $|I_{t+\tau}| >\sqrt{\log n}$.
\item \label{pullResUpper_a} Let $\sqrt{\log n}\leq|I_t| \leq n/\log n$. Then there are $\tau_1, \tau_2= \log_{1+q} (n/|I_t|)+o(\log n)$ such that $|I_{t+\tau_2}|<n/\log n<|I_{t+\tau_1}|.$
\item \label{pullResUpper_b} Let $n/\log n\leq|I_t| \leq n -n/\log n$. Then there is $\tau = o(\log n)$ such that $|I_{t+\tau}| >n-n/\log n.$
\item \label{pullResUpper_c}Let $|I_t| ~\geq n - n/\log n.$
\begin{enumerate}
\itemsep0em 
\item[1.] Case $q= 1$: Then there is $\tau = o(\log n)$ such that $|I_{t+\tau}| =n.$
\item[2.] Case $q\neq1$: Then there is $\tau \leq -\log n/\log\left(1-q\right) +o(\log n)$ such that $|I_{t+\tau}| =n$.
\end{enumerate}
\end{enumerate}
\end{lemma}
\begin{proof}
%An inspection of the respective parts in the proof of Lemma \ref{pullUpper} together with Lemma \ref{linearFraction} shows that those parts which take $o(\log n)$ rounds (i.e.~$b)$) in the original setting need only a constant factor more rounds in the setting with deleted edges; we do not repeat the arguments. Therefore, it suffices to show $a)$ and $c)$.  
We start with $a)$. Let $|I_t|\in [\log n,n/\log n].$  First note that any bound on $e(U_t,I_t)$ translates to a bound for $|I_t|$, as with Lemmas \ref{expmix}, \ref{linearFraction} we obtain 
\begin{equation}
\label{eq:auxIte}
	(1-o(1))\varepsilon \Delta_n |I_t|  \le e(U_t,I_t) \le \Delta_n |I_t|.
\end{equation}
In particular, up to constant factors, $|I_t|$ is $e(U_t, I_t) / \Delta_n$ and vice versa. From Lemma \ref{expResPull_a} we obtain that $e(U_{t+1},I_{t+1}) = (1+q\pm |I_t|^{-1/3})e(U_t,I_t)$ with probability $1-O(|I_t|^{-1/3})$.
%
%Iterating this bound we get that for all $t_0\leq t$ such that $|I_{t_0}|\ge \log n$
%\marginpar{\tiny verstehe ich nicht: du möchtest dass die rechte Seite nicht zu groß ist, oder (?) also $e(U,I) \le \varepsilon \Delta_n n / \log n$ oder so}
%$$(1+q- |I_{t_0}|^{-1/3})^{t-t_0}e(U_{t_0},I_{t_0})\leq e(U_{t+1},I_{t+1}) \leq  (1+q+ |I_{t_0}|^{-1/3})^{t-t_0}e(U_{t_0},I_{t_0}).$$
%Lemma \ref{expmix} and \ref{linearFraction} yield that $e(U_t,I_t)\geq (1-o(1))\varepsilon\Delta_n|I_t|$, which we can use to relate the number of edges between informed and uninformed vertices to the number of informed vertices. Indeed,  
%$$(1-o(1))\varepsilon\Delta_n|I_{t+1}|\leq e(U_{t+1},I_{t+1}) \leq (1+q+|I_{t_0}|^{-1/3})^{t-t_0} e(U_{t_0},I_{t_0})\leq (1+q+|I_{t_0}|^{-1/3})^{t-t_0} \Delta_n |I_{t_0}|,$$
%and therefore 
%\begin{equation}\label{pull_con_upper}
% |I_{t+1}|\leq (1+q+|I_{t_0}|^{-1/3})^{t-t_0} \frac{1}{\varepsilon}|I_{t_0}|.
%\end{equation} 
%Conversely
%\begin{equation*}
%\begin{aligned}
%\Delta_n |I_{t+1}|\geq e(U_{t+1},I_{t+1})\geq (1+q-|I_{t_0}|^{-1/3})^{t-t_0} e(U_{t_0},I_{t_0})\geq (1+q-|I_{t_0}|^{-1/3})^{t-t_0} \varepsilon\Delta_n|I_{t_0}|,
%\end{aligned}
%\end{equation*}
%which implies
%\begin{equation}\label{pull_con_lower}
%|I_{t+1}|\geq (1+q-|I_{t_0}|^{-1/3})^{t-t_0} \varepsilon|I_{t_0}|.
%\end{equation}
Proceeding as in Examples \ref{concetrationRemark} and \ref{concetrationExample}, where we replace the events   ``$|I_{t}| \ge \mathbb{E}_{t-1}\left[|I_{t}|\right]-  \mathbb{E}_{t-1}\left[|I_{t}|\right]^{2/3}$ or $|I_t| \ge n/g(n)\text{''}$ and ``$\left||I_{t}|-\mathbb{E}_{t-1}\left[|I_{t}|\right]\right|\leq \mathbb{E}_{t-1}\left[|I_{t}|\right]^{2/3}$'' with ``$e(U_{t},I_{t}) \geq  (1+q- |I_{t-1}|^{-1/3})e(U_{t-1},I_{t-1})$ or $|I_{t}|\geq n/\log n$'' and ``$e(U_{t+1},I_{t+1}) = (1+q\pm |I_t|^{-1/3})e(U_t,I_t)$'' we obtain the statement.

We continue with $b)$. Consider first the case $|I_t|\in[n/\log n,n/2]$. Using Lemmas \ref{expmix}, \ref{linearFraction}, i.e. $e(U_t,I_t)\geq \varepsilon|U_t||I_t|\Delta_n/n+o(\Delta_n)|I_t|$, together with $|U_t|\geq n/2$ implies
 \begin{align*}
\mathbb{E}_{t}[|I_{t+1}\backslash I_t|] =\sum_{u\in U_t}q\frac{|N(u)\cap I_t|}{|N(u)|}\geq \frac{q\cdot e(U_t,I_t)}{\Delta_n}\geq
 \frac{q\varepsilon|U_t||I_t|\Delta_n/n+o(\Delta_n)|I_t|}{\Delta_n(1+o(1))} \geq \left(\frac{q\varepsilon}{2}+ o(1)\right)|I_t|.
\end{align*}
Proceeding as in Example \ref{concetrationRemark}, where we set $g=2,f=\log n$ and $c=q\varepsilon/2+o(1)$, we are finished with this part as well. Now let $|I_t| \in [n/2,n-n/\log n]$. We switch our focus to the set of uninformed vertices. Using again that $e(U_t,I_t)\geq \varepsilon|U_t||I_t|\Delta_n/n+o(\Delta_n)|U_t|$, we have
\begin{align*}
\mathbb{E}_{t}[|U_{t+1}|] &= \sum\limits_{u\in U_t}1-q\frac{|N(u)\cap I_t|}{|N(u)|}
= \sum\limits_{u\in U_t}1-q\frac{|N(u)\cap I_t|}{\Delta_n(1+o(1))} \\ &= 
|U_t|-\frac{q\cdot e(U_t,I_t)}{\Delta(1+o(1))}  =|U_t|-
 \frac{q\varepsilon{|U_t||I_t|\Delta_n}/{n}+o(\Delta_n)|U_t|}{\Delta_n(1+o(1))} \leq \left(1-\frac{q\varepsilon}{2}+ o(1)\right)|U_t|.
\end{align*}
Inductively we obtain for any integer $\tau \geq 1$ the bound $\mathbb{E}_{t}\left[|U_{t+\tau}|\right] \leq \left(1-q\varepsilon/2+ o(1)\right)^\tau|U_t|,$
and so for some $\tau := 2\log\log n/\log(1/(1-q\varepsilon/2+ o(1))) = o(\log n)$ we have 
$$\mathbb{E}_{t}\left[|U_{t+\tau}|\right] \leq |U_t|/\log^2n=o(n/\log n). $$
Hence, by Markov's inequality, $P_{t}[|U_{t+\tau}|\geq n/\log n] =  o(1).$

In order to show $c)$ let $|I_t|\in [n-n/\log n,n]$.
As for $q=1$ the term $1-q$ in Lemma \ref{expResPull_b} vanishes, we distinguish the cases $q=1$ and $q \neq 1$. We start with $q=1$. By induction, it follows that for any round $\tau >0$ and suitable $f=o(1)$, $$\mathbb{E}_{t}[|U_{t+\tau}|] \leq \left(f(n)\right)^{\tau}|U_t|. $$ 
We choose $\tau = \log_{1/f(n)}(n) = o(\log n)$, as $1/f=\omega(1)$.
Hence we obtain 
$\mathbb{E}_{t}[|U_{t+\tau}|]\leq |U_t|/n \leq {1}/{\log n}.$
Therefore we have $P_{t}[|U_{t+\tau}|\geq 1]\leq o(1)$ by Markov's inequality. 
For $q\neq 1$ we have by induction, for any number of rounds $\tau \ge 1$, $$\mathbb{E}_{t}[|U_{t+\tau}|] \leq \left(1-q+o(1)\right)^{\tau}|U_t|. $$ 
We choose $\tau = \log_{1/(1-q+o(1))}(n) =-\log n/\log(1-q) + o(\log n).$ Thus using Markov's inequality, analogously to the case $q=1$, we obtain the desired upper bound. 
\end{proof}

Note that for $q=1$ this already implies Theorems \ref{pullFast}, \ref{pullIsRobust}. This leaves the case for $q\neq 1$.
\begin{lemma}\label{pullLower}
Let $0< \varepsilon \le 1/2, q\in (0,1]$ and $\mathcal G=(G_n)_{n \in \N}$ be an expander sequence. Let $\tilde{\mathcal G}=(\tilde G_n)_{n \in \N}$ be such that each $\tilde G_n$ is obtained by deleting edges of $G_n$ such that each vertex keeps at least a $(1/2 + \varepsilon)$ fraction of its edges and abbreviate $I_t=I_t^{(pull)}$. Let further $q \in (0,1)$ and $ |I_t|\, \leq n/2$. Then for  $\tau = -\log n/\log\left(1-q\right)$ and all $c<1$ whp $|I_{t+c\tau}|<n$.
\end{lemma}
\begin{proof} We consider a modified  process in which vertices have a higher chance of getting informed. In particular, note that the probability that $u \in U_t$ gets informed is at most $q |N(u) \cap I_t|/|N(u)| \le q$ and that all these events are independent; now we assume that each such $u$ gets independently informed with probability exactly $q$. Then the runtime in this modified model constitutes a lower bound for the runtime in the original model.

Let $c<1,u\in U_t$ and $E_u$ be the event that $u$ does not get informed in $c\tau$ rounds in this  model. Thus
$$P[E_u]=(1-q)^{c\tau}=(1-q)^{-c\log n/\log(1-q) }={n^{-c}}=\omega\left(1/n\right) $$
and as the events $E_u$ are independent and $|U_t|=\Theta(n)$
\begin{align*}
P\left[\bigwedge\limits_{u \in U_t} \overline{E_{u}}\right] 
%\leq P\left[\bigwedge\limits_{u \in V} \overline{E_{u}}\right]
\le \prod\limits_{u \in U_t} P[\overline{E_{u}}]
%= \prod\limits_{u \in V} (1-P[E_{u}])
 \leq \exp\left(- \sum\limits_{u \in U_t} P[E_{u}]\right)=o(1).
\end{align*}
%implies the claim.
%Note that the probability of $E_u$ for \pull is then also $\omega(n^{-1})$; applying Lemma \ref{finish} gives the bound.
\end{proof}

\label{proofs_sec}

\subsection{Proof of Theorem \ref{pushPullFast} ---  \pushpull is fast on expanders}
\label{push_pull_fast_subsec}
As we are now in the case without edge deletions, we begin with a lemma that determines the expected number of informed vertices in one round.  Intuitively we will show that \push and \pull do not interact badly and therefore \pushpull is given as a straightforward  combination of \push and \pull.  
\begin{lemma}
\label{expPushPull}
Let $\mathcal G$ be an expander sequence and abbreviate $I_t=I_t^{(pp)}$.
\begin{enumerate}[label={(\alph*)},ref={\thetheorem~(\alph*)}]
\itemsep0em 
\item Let $|I_t|\leq n/\log n$. Then
$\mathbb{E}_{t}[|I_{t+1}\setminus I_t|] = (2q+o(1))|I_t|. $\label{expPushPull_a}
\item Let $|U_t|\leq n/\log n$. Then
$\mathbb{E}_{t}[|U_{t+1}|] = (1+o(1))e^{-q}(1-q)|U_t|. $\label{expPushPull_b}
\end{enumerate}
\end{lemma}
%For the proof we  state a fact from \cite[Section 2.2]{Panagiotou2015} first. There, only the case $q = 1$ was treated, but the statement for $q\in(0,1)$ is also implicit.
%\begin{lemma}
%\label{Panagiotou}
%Let $\mathcal G=(G_n)_{n \in \N}=((V_n,E_n))_{n\in \mathbb{N}}$ be an expander sequence 
%\begin{enumerate}[label={(\alph*)},ref={\thetheorem~(\alph*)}]
%\itemsep0em 
%\item Let $M\subseteq V_n, |M|\leq n/\log n$ and $d_n$ be the average degree of $G_n$. For $\tilde \lambda_n := \lambda_n/d_n + 1/\log n$ define the set $A(M):=\{v \in N(M) \backslash M: |N(v)\cap M| \geq 2  {\tilde \lambda_n^{1/2}}d_n\}$. Then $|A(M)| < \tilde \lambda_n |M|.$
%\label{Panagiotou_a}
%\item
%and $|I_t| \leq n/\log n$. Then 
%$\mathbb{E}_{t}\big[|I_{t+1}^{(push)} \backslash I_t|\big] = (q+o(1))|I_t|.$\label{Panagiotou_b}
%\item Let $|U_t| \leq n/\log n$. Then 
%$\mathbb{E}_{t}\big[|U_{t+1}^{(push)}|\big]\leq (e^{-q}-o(1))|U_t|.$ \label{Panagiotou_c}
%\end{enumerate}
%\end{lemma}
\begin{proof}%[Proof of Lemma \ref{expPushPull}]~
To begin with $a)$. The probability that  $v\in U_t$ gets informed by \pull  is $q|N(v)\cap I_t|/|N(v)|$. Thus, using Lemma \ref{expmix} 
\begin{equation}\label{basicEquPull}
\begin{aligned}
\mathbb{E}_{t}[|I_{t+1}^{(pull)}\backslash I_t|] &= \sum\limits_{u\in U_t}q\frac{|N(u)\cap I_t|}{|N(u)|} 
 = q\sum\limits_{u\in U_t}\frac{|N(u)\cap I_t|}{\Delta_n(1+o(1))}\\ 
& = (q+o(1))\frac{e(U_t,I_t)}{\Delta_n} 
= (q+o(1)) \frac{|U_t||I_t|\Delta_n/n+o(\Delta_n)|I_t|}{\Delta_n}.
\end{aligned}
\end{equation}
Since $|I_t| = o(n)$ we obtain that $|U_t| = (1-o(1))n$ and this expression simplifies to $(q + o(1))|I_t|$. The probability that  $v\in U_t$ gets informed by \push  is $1-\prod_{i\in N(v)\cap I_t}(1-1/|N(v)|)$. Using $e^{-1/n +o(1/n)}= 1-1/n, e^{-1/n}= 1-1/n + o(1/n),$ and $|I_t|=o(n)$ we obtain in a similar fashion 
\begin{equation}\label{basicEquPush}
\begin{aligned}
\mathbb{E}_t[|I_{t+1}^{(push)}\setminus I_t|]&= \sum\limits_{u\in U_t}1-\prod_{i\in N(u)\cap I_t}\left(1-\frac{q}{|N(i)|}\right)=\sum\limits_{u\in U_t}1-\exp\left(-(1 - o(1))\frac{q|N(u)\cap I_t|}{\Delta_n}\right)\\
&=q\sum\limits_{u\in U_t}\frac{|N(u)\cap I_t|}{\Delta_n(1+o(1))}=(q + o(1))|I_t|.
\end{aligned}
\end{equation}
 We express the expected number of vertices informed by \pushpull after one additional round in terms of the expected values we just calculated (\eqref{basicEquPull} and \eqref{basicEquPush}):
\begin{equation}\label{equPushPull}
\begin{aligned}
\mathbb{E}_{t}[|I_{t+1}\backslash I_t|]
&= \mathbb{E}_{t}\left[|I_{t+1}^{(pull)}\backslash I_t|\,+\,|I_{t+1}^{(push)}\backslash I_t|\,-\,|(I_{t+1}^{(push)}\backslash I_t)\cap (I_{t+1}^{(pull)}\backslash I_t)|\right] \\
&=(2q-o(1))|I_t|\,-\, \mathbb{E}_{t}\left[|(I_{t+1}^{(push)}\backslash I_t)\cap (I_{t+1}^{(pull)}\backslash I_t)|\right].
\end{aligned}
\end{equation}
Lemma \ref{fraction_a} gives a set $I\subseteq I_{t+1}^{(push)}, |I|=(1-o(1))|I_{t+1}^{(push)}|$, such that $|N(u)\cap I_{t+1}^{(push)}|=o(1)|N(u)|$ for all $u\in I.$ Since push and pull happen independently 
%Let $A := A(I_t)$ be the set defined in Lemma \ref{Panagiotou_a} (with $M = I_t$) and set $A^c=V_n\backslash A$. Then
\begin{align*}
\mathbb{E}_{t}\left[|(I_{t+1}^{(pull)} \backslash I_t) \cap (I_{t+1}^{(push)} \backslash I_t)|~\big| ~I_{t+1}^{(push)} \right] &= 
\sum\limits_{u \in I_{t+1}^{(push)} \backslash I_t} P_{t}[u \in I_{t+1}^{(pull)} \backslash I_t]=
\sum\limits_{u \in I_{t+1}^{(push)}\backslash I_t}q\frac{|N(u)\cap I_t|}{|N(u)|} \\&\le
\sum\limits_{u \in I}q\frac{|N(u)\cap I_t|}{|N(u)|}+ \sum\limits_{u \in I_{t+1}^{(push)}\setminus I}q\frac{|N(u)\cap I_t|}{|N(u)|} .
\end{align*}
Using that $|N(u)\cap I_t| =o(|N(u)|)$ for all $u \in I$ we obtain
\begin{align*}
\mathbb{E}_{t}\left[|(I_{t+1}^{(pull)} \backslash I_t) \cap (I_{t+1}^{(push)} \backslash I_t)|\right] & \leq \mathbb{E}_t[ o(|I|) + |I_{t+1}^{(push)}\setminus I|]=o(|I_t|),
\end{align*}
as $|I|\le |I_{t+1}^{(push)}|\le 2|I_t|$ and $|I_{t+1}^{(push)}\setminus I|=o(|I_{t+1}^{(push)}|)=o(|I_t|)$.
%The first term is at most $q|A(I_t)|$ and with Lemma~\ref{Panagiotou_a} we can bound it with $o(|I_t|)$. Similarly, the second term is at most $o(1) \cdot |I_{t+1}^{(push)}\backslash I_t|$, and since $|I_{t+1}^{(push)}\backslash I_t| \le |I_t|$  we can also bound it with $o(|I_t|)$.
Combining this with \eqref{equPushPull} we get $\mathbb{E}_{t}[|I_{t+1} \setminus I_t|] = (2q+o(1))|I_t|$, as claimed.

%For the upper bound we use bounds for the expected number of informed vertices by  \pull and \push. First, as  $|I_{t+1}^{(push)}\backslash {I_t}|\leq |I_t|$ without message transmission  and message transmission failures  being independent, applying a Chernoff bound gives
% $$\mathbb{E}_{t}\left[|I_{t+1}^{(push)}\backslash I_t|\right]\leq(q+o(1))|I_t|.$$
%For the \pull algorithm, using Lemma \ref{expPull_a}, we readily get that $\mathbb{E}_{t}[|I_{t+1}^{(pull)}\backslash I_t|]\leq(q+o(1))|I_t|.$
%By linearity of expectation
%\begin{equation*}
%\mathbb{E}_{t}[|I_{t+1}\backslash I_t|]
%\leq \mathbb{E}_{t}[|I_{t+1}^{(pull)}\backslash I_t|] + \mathbb{E}_{t}[|I_{t+1}^{(push)}\backslash I_t|]
%\leq (1+o(1))|I_t|+(1+o(1))|I_t| = (2q+o(1))|I_t|.
%\end{equation*}
Next we show $b)$. 
Let $A_u$ be the event that an uninformed vertex $u$ does not get informed by the \push algorithm, let $B_u$ be the corresponding event for \pull. Then $A_u$ and $B_u$ are independent and $A_u\cap B_u$ is the event that $u$ does not get informed in the current round. We obtain 
\begin{equation*}
\begin{aligned}
 P_t[A_u] =\prod_{i\in N(u)\cap I_t}\left(1-\frac{q}{|N(i)|}\right)\le \left(1-\frac{q}{\Delta_n}\right)^{|N(u)\cap I_t|} \le 
\exp\left(-q\frac{|N(u)\cap I_t|}{\Delta_n}\right)=\exp\left(\frac{-q|N(u)\cap I_t|}{(1+o(1))|N(u)|}\right)
%1-(1+o(1))q\frac{|N(u)\cap I_t|}{\Delta_n}.
\end{aligned}
\end{equation*}
and
%from Lemma \ref{expPull_b} \marginpar{\tiny b? Im display Reihenfolge vertauschen?} and Lemma \ref{Panagiotou}
$$ P_t[B_u]= 1-\frac{q|N(u)\cap I_t|}{|N(u)|}.$$
According to Lemma \ref{fraction_a} there is a set $U\subseteq U_t, |U|=(1-o(1))|U_t|$ such that $|N(u)\cap I_t|=(1-o(1))|N(u)|$ for all $u\in U.$ As $P_t[A_u\cap B_u]\leq 1$ we get therefore
\begin{equation*}
\begin{aligned}
\mathbb{E}_{t}[|U_{t+1}|]=&\sum\limits_{u\in U_t} P_{t}[A_u\cap B_u]\le\sum\limits_{u\in U}P_{t}[A_u] \cdot P_{t}[B_u]+|U_t\setminus U|\le(1+o(1))e^{-q}(1-q) |U_t|.
\end{aligned}
\end{equation*}
For the lower bound we need to find a lower bound on the probability of a single uninformed vertex not getting informed in one round by \push. Indeed, for any $u \in U_t$ and sufficiently large $n$
\begin{equation}\label{pushLower}
\begin{aligned}
P_t[A_u]&=\prod\limits_{v\in N(u)\cap I_t}\left(1-\frac{q}{|N(v)|}\right)\geq \left(1-\frac{q}{\delta_n}\right)^{|N(u)\cap I_t|}
\geq e^{-q\Delta_n/\delta_n}.
\end{aligned}
\end{equation}
Combining this inequality with the trivial bound $P[B_u]\ge 1-q$, we get a lower bound on the expected number of uninformed vertices after one round using \pushpull: 
\[
\mathbb{E}_{t}[|U_{t+1}|]=\sum\limits_{u\in U_t} P_{t}[A_u\cap B_u]= \sum\limits_{u\in U_t} P_{t}[A_u]\cdot P_{t}[B_u]\geq e^{-q\Delta_n/\delta_n}(1-q) |U_t|= (1+o(1))e^{-q}(1-q) |U_t|.
\]
\end{proof}

Next we show upper and lower bounds that together with Lemma \ref{startup} imply Theorem \ref{pushPullFast}.
\begin{lemma}\label{PushPullUpper}
Let $\mathcal G$ be an expander sequence and abbreviate $I_t=I_t^{(pp)}$. Let $q \in (0,1]$. Then the following statements hold whp.
\begin{enumerate}
\itemsep0em 
%\item Let $1\leq|I_t|$. Then there is $\tau=o(\log n)$ such that $|I_{t+\tau}|>\sqrt{\log n}$.
\item Let $\sqrt{\log n}\leq|I_t|\leq n/\log n$. Then there are $\tau_1, \tau_2= \log_{1+2q} (n/|I_t|)+o(\log n)$ such that $|I_{t+\tau_2}|<n/\log n<|I_{t+\tau_1}|.$
\item Let $n/\log n\leq |I_t|\leq n -n/\log n$. Then there is $\tau = o(\log n)$ such that $|I_{t+\tau}|>n-n/\log n.$
\item Let $|I_t| \geq n - n/\log n.$
\begin{enumerate}
\item[1.] Case $q= 1$: Then there is $\tau = o(\log n)$ such that $|I_{t+\tau}|=n.$
\item[2.] Case $q\neq1$: Then there is $\tau \leq \log n/(q-\log\left(1-q\right))+o(\log n)$ such that $|I_{t+\tau}|=n$.
\end{enumerate}
\end{enumerate}
\end{lemma}
\begin{proof}
Since $|I_t|\geq |I^{(pull)}_t|$ the statements  $b)$ and $c)$ for $q=1$ follow immediately from Lemma \ref{pullResUpper}. To see $a)$, note that by using Lemma \ref{expPushPull} we get $\mathbb{E}_{t}[|I_{t+1}\backslash I_t|] = (2q+o(1))|I_t|$,  and proceeding as in Example \ref{concetrationExample}  implies the claim.
 
Finally we show $c)$ for $q\neq 1$. Let $|I_t|\, \geq n-n/\log n$. 
By Lemma \ref{expPushPull}, we obtain that for any $\tau \in \mathbb{N}$, $$\mathbb{E}_{t}[|U_{t+\tau}|] = \left( (1+o(1))e^{-q}(1-q)\right)^{\tau}|U_t|. $$ 
Thus we may choose $\tau =  \log n / (q-\log(1-q))+o(\log n)$
such that, say, $\mathbb{E}_{t}[|U_{t+\tau}|]\leq |U_t|/n \leq {1}/{\log n}.$
Thus $P_{t}[|U_{t+\tau}|\geq 1]\leq o(1)$ by Markov's inequality.
\end{proof}
Note that for $q=1$ this already implies Theorem \ref{pushPullFast}. This leaves the case for $q\neq 1$.
\begin{lemma}\label{PushPullLower}
Let $\mathcal G$ be an expander sequence and abbreviate $I_t=I_t^{(pp)}$, let $q \in (0,1)$ and $ |I_t|\, \leq n/2$. Then for  $\tau = \log n/(q-\log\left(1-q\right))$ and all $c<1$ whp $|I_{t+c\tau}|<n$.
\end{lemma}
\begin{proof}
We consider a modified  process in which vertices have a higher chance of getting informed. In particular, note that the probability that $u \in U_t$ gets informed by \pull is at most $q |N(u) \cap I_t|/|N(u)| \le q$ and that all these events are independent; according to \eqref{pushLower} the probability that $u \in U_t$ gets informed by \push is at most $1-e^{-q\Delta_n/\delta_n}$. Now we assume that each such $u$ gets independently informed with probability exactly $1-e^{-q\Delta_n/\delta_n}(1-q)$. Then the runtime in this modified model constitutes a lower bound for the runtime in the original model.
Let $u\in U_t$ and $E_u$ be the event that $u$ does not get informed in this modified model in $c\tau$ rounds. Thus for $c<1$,
$$P[E_u]\geq((1-q)e^{-q\Delta_n/\delta_n})^{c\tau} =\omega\left(n^{-1}\right)$$ 
and as the events $E_u$ are independent and $|U_t|=\Theta(n)$
\begin{align*}
P\left[\bigwedge\limits_{u \in U_t} \overline{E_{u}}\right] 
%\leq P\left[\bigwedge\limits_{u \in V} \overline{E_{u}}\right]
\le \prod\limits_{u \in U_t} P[\overline{E_{u}}]
%= \prod\limits_{u \in V} (1-P[E_{u}])
 \leq \exp\left(- \sum\limits_{u \in U_t} P[E_{u}]\right)=o(1).
\end{align*}
%Clearly the probability of $E_u$ for \pushpull is then also $\omega(n^{-1})$; applying Lemma \ref{finish} gives the bound.
\end{proof}

\subsection{Proof of Theorem \ref{firstPhasePushRobust} --- \push informs almost all vertices fast in spite of edge deletions}\label{push_robust_sec}

%\marginpar{\tiny Änderung}We want to show that the number of informed nodes doubles in ``early'' rounds. Intuitively this holds true in expander sequences, as the degree is very large and the number of informed nodes is small. As the edges are spread very evenly among the vertices the probability is very small that any informed node chooses an already informed neighbour or that an uninformed node is chosen twice. We show that the same holds if we consider adversarial edge deletion by giving a suitable coupling.

To shorten the notation let us call the setting with deleted edges  ``new model''  and the setting without  ``old model'', that is, the term new model corresponds to the graphs in $\tilde{\mathcal G}$, while old model refers to the (original) graphs in $\mathcal G$. %More specifically, when considering the old model, we look at the information spreading process in the unmodified expander graph; in contrast, when considering the new model, we look at the information spreading process in the modified graph that is obtained by deleting edges from the unmodified graph such that each vertex keeps at least a $(1/2+\varepsilon)$ fraction of its edges.
We prove Lemma \ref{propositionFirstPhasePushRobust} that directly implies Theorem \ref{firstPhasePushRobust}.
%In order to do this, we will use Lemmas \ref{increaseByFactor2lemma}--\ref{binomial_exceeds_expectation}.
We write $I_t=I_t^{(push)}$ throughout. 
\begin{lemma}
\label{propositionFirstPhasePushRobust}
Under the assumptions of Theorem \ref{firstPhasePushRobust} the following holds for the new model:
\begin{itemize}
\itemsep0em 
\item[a)] There are $\tau, \tilde \tau = \log_{1+q}(n) + o(\log n)$ such that whp $|I_{\tilde \tau}| < n/\log n < |I_{\tau}|$.
\item[b)] Assume $|I_t|  \geq n/\log n$. Then there is a $\tau=o(\log n)$ such that whp $|I_{t+\tau}| \geq n-n/\log n$.
\end{itemize}
\end{lemma}
%Note that directly implies Theorem \ref{firstPhasePushRobust}.
For the proof of Lemma \ref{propositionFirstPhasePushRobust} we will need the following statements, the first one taken from~\cite{Panagiotou2015}.
\begin{lemma}[Proof of Lemma 2.5 in \cite{Panagiotou2015}]
\label{increaseByFactor2lemma}
Consider the old model. Assume $|I_t| < n/\log n$ and $q=1$. Then   
\begin{align}
\label{increaseByFactor2}
P_{t}\big[|I_{t+1}|= |I_t| \, +\,  (1-o(1))|I_t|\big] = 1 - o(1).
\end{align}
\end{lemma}

\begin{lemma}
\label{increaseByFactor1+q_lemma}
\label{final_step_lemma}
Consider \push on a sequence of graphs $(G_n)_{n \in \N}$,
%=((V_n,E_n))_{n\in\N}$ with $|V_n|=n$
where $G_n$ has $n$ vertices. Assume that $|I_t| =\omega(1)$ and that \eqref{increaseByFactor2} holds for $q=1$, that is, assume that $P_{t}\big[|I_{t+1}|= |I_t| \, +\,  (1-o(1))|I_t|\big] = 1 - o(1)$  for $q=1$.
Then for $q \in (0,1]$ 
\begin{equation}
\label{increaseByFactor1+q}
P_{t}\big[|I_{t+1}| = |I_t| \, +\, (q-o(1))|I_t|\big] = 1 - o(1).
\end{equation}
Moreover, assume that whenever $|I_t| < n/\log n$, for $q = 1$, \eqref{increaseByFactor2} holds. Then there are $\tau, \tilde \tau = \log_{1+q}(n) + o(\log n)$ such that whp
\begin{equation}
\label{final_step}
|I_{\tilde \tau}| < n/\log n < |I_{\tau}|.
\end{equation}
\end{lemma}

\begin{proof}
For a graph $G$ and for $v \in I_t$ let $X_v(G)$ denote the vertex to which $v$ pushes in round $t$.
Let $$N_{t+1}:=\{X_v(G_n) \mid v \in I_t\}\cap U_t.$$
Note that whenever $|I_t| < n/\log n $ whp $|N_{t+1}| = (1 - o(1))|I_t|$
 from \eqref{increaseByFactor2}. For $q\in (0,1]$ each vertex in $N_{t+1}$ has a probability of at least $q$ to get informed and all these events are independent;  thus ~\eqref{increaseByFactor1+q} follows directly by applying the Chernoff bounds whenever $|I_t| = \omega(1)$. 
 
 In order to prove the second statement we call a round $t$ that does not satisfy \eqref{increaseByFactor1+q} a \emph{failed} round.
 Note that we just argued that the probability that a round fails is $o(1)$ whenever $|I_t| = \omega(1)$ and $|I_t| < n/\log n$, and the events that distinct rounds fail are independent. In particular, the number of failed rounds among the next $R$ rounds, assuming that $|I_t|$ stays below $n/\log n$, is whp $o(R)$. Moreover, if a round does not fail, the number of informed vertices increases by a factor of $(1 + q +o(1))$ and otherwise it may increase by an arbitrary factor in the interval $[1,2]$. Finally, Lemma \ref{startup} yields that there  is $t^*=o(\log n)$ such that whp $|I_{t^*}|=\omega(1)$, which implies that after $R + t^*$ rounds, the number of informed vertices is whp in the interval
 \[
 	[(1 + q +o(1))^{R - o(R)}, (1 + q +o(1))^{R - o(R)} \cdot 2^{o(R)}]
  \]
 and choosing $R = \log_{1+q}(n) + o(\log n)$ in two ways establishes~\eqref{final_step}.
\end{proof}
In the subsequent proof of Lemma \ref{propositionFirstPhasePushRobust} we will use the simple observations that for any $n\in\mathbb{N}_0$
\begin{equation}
\label{binomial_exceeds_expectation}
	P\left[\textrm{Bin}(n,1/2) \ge {n}/2\right] \ge 1/2
	~~\textrm{and}~~
	P\left[\textrm{Bin}(n,1/4) \ge {n}/4\right]
	\ge 1/4,
\end{equation}
see for example~\cite{greenberg2014tight} when $n > 4$, and the other cases are checked easily.
\begin{proof}[Proof of Lemma \ref{propositionFirstPhasePushRobust}]
We first show $a)$. 
We assume $q=1$ and prove that, for $|I_t|<n/\log n$, \eqref{increaseByFactor2} also holds in the new model; then claim $a)$  follows directly from Lemma \ref{increaseByFactor1+q_lemma}. Let $G=(V,E)$ be a graph. For $v \in I_t$ let $X_v(G)$ denote the vertex to which $v$ pushes in round $t$.
For $u\in V$ let $c_u(G):=|\{v \in I_t \mid X_v(G)=u\}|$ denote the number of times $u$ is pushed in round $t$.
Let $$\mathcal{Y}_t(G) := \{v \in I_t \mid c_v(G)=1\} \quad \text{ and }\quad \mathcal{H}_t(G):=\{v \in I_t \mid c_v(G)\geq 1\}$$ denote the set of informed vertices that are being pushed exactly once in round $t$ and 
the set of informed vertices that are being pushed at least once in round $t$ respectively. Let $$\mathcal{Z}_t(G):=\{v \in V \mid c_v(G)\geq 2\}$$ denote the set of vertices that are being pushed more than once in round $t$. 
Let $Y_t(G):=|\mathcal{Y}_t(G)|$ and
$H_t(G):=|\mathcal{H}_t(G)|$ and, in slight abuse of notation, let
 $Z_t(G):=\sum_{k \geq 2}(k-1)\cdot |\{v \in V \mid c_v(G)=k\}|$ denote the number of vertices that are being pushed multiple times in round $t$ counted with multiplicity. Note that the quantity $Y+Z$ denotes the number of pushes that have no effect in the respective round, i.e., there are $Y+Z$ pushes that are useless in the sense that even without them, the same number of vertices would become informed in the respective round.
 In the following paragraphs we condition on $I_t$ implicitly, that is, we write $P[\dots]$ instead of $P_{t}[\dots]$ etc.~to lighten the notation.
We want to show that (\ref{increaseByFactor2}) does hold in the new model; for contradiction we assume that this is not the case. Hence we can infer that there is a constant $c >0$ such that 
$$\limsup\limits_{n \ra \infty}P[Y_t(\tilde G_n)\geq c |I_t|]>0 \quad \text{ or } \quad \limsup\limits_{n \ra \infty} P[Z_t(\tilde G_n)\geq c |I_t|]>0.$$ 
Thus, w.l.o.g., we can assume that there is $f^*>0$ and $n_0 \in \N$ such that
$$P[Y_t(\tilde G_n)\geq c |I_t|]>f^* \text{ for all }n\geq n_0 \quad\text{ or }\quad P[Z_t(\tilde G_n)\geq c |I_t|]>f^* \text{ for all } n \geq n_0;$$
if this is not the case we can restrict ourselves to a suitable subsequence of $(n)_{n \in \N}$ on which it is true. Next, we describe an explicit coupling between the new and the old model. For any vertex $v$ consider $X_v(G_n)$. If $X_v(G_n) \in N_{\tilde G_n}(v)$, then set $X_v(\tilde G_n):=X_v(G_n)$ and otherwise choose $X_v(\tilde G_n)$ uniformly at random from $N_{\tilde G_n}(v)$. Note that $X_v(G_n),X_v(\tilde G_n)$ have by construction the correct marginal distribution. Moreover, note that by construction, the family
\begin{equation}
\label{eq:condIndep}
	\left( X_v(G_n) ~|~ (X_u(\tilde{G}_n))_{u \in V_n} \right)_{v \in V_n}
\end{equation}
of random variables is independent, since $X_v(G_n)$ depends only on $X_v(\tilde{G}_n)$ for all $v \in V_n$.
%Thus we have defined a family of random variables $(X_v(G_n),X_v(\tilde G_n))_{v \in V_n}$ on a joint probability space.
%In slight abuse of notation, for better readability we have used the same names for the ``coupled'' and the ``uncoupled'' random variables. The marginal distributions of the coupled pair $(X_v(G_n),X_v(\tilde G_n))$ are by construction the same as the distributions of the uncoupled random variables $X_v(G_n)$ and $X_v(\tilde G_n)$.

We begin with the case that $P[Y_t(\tilde G_n)\geq c |I_t|]>f^*$. We will show
\[
	P\big[H_t(G_n) \ge  Y_t(\tilde G_n)/2 ~|~ \mathcal{Y}_t(\tilde{G}_n) \big] \ge 1/2
\]
and then, since by assumption $P[Y_t(\tilde G_n) \geq c |I_t|]> f^*$, we can infer $P[H_t(G_n)\geq c|I_t|/2]\geq f^*/2$ which contradicts Lemma \ref{increaseByFactor2lemma}.
 Let $\mathcal{Y}_t(\tilde G_n) = \{y_1, \dots, y_{Y_t(\tilde G_n)}\}$, then there are distinct vertices $v_1, \dots, v_{Y_t(\tilde G_n)} \in I_t$ such that $X_{v_i}(\tilde{G}_n) = y_i$ for all $i\in \{1, \dots, Y_t(\tilde G_n)\}$. Due to~\eqref{eq:condIndep} the events $(\{X_{v_i}(G_n) = X_{v_i}(\tilde{G}_n)\})_{1 \le i \le Y_t}$ are independent. Moreover, for all $i\in \{1, \dots, Y_t(\tilde G_n)\}$,
\[
	P\big[X_{v_i}(G_n) = X_{v_i}(\tilde{G}_n) ~|~ \mathcal{Y}_t(\tilde G_n)\big]
	= \frac{d_{\tilde G_n}(v_i)}{d_{G_n}(v_i)} \ge 1/2 + \varepsilon
\]
and therefore, given $\mathcal{Y}_t(\tilde G_n)$, $H_t(G_n)$ dominates a binomially distributed random variable Bin$(Y_t(\tilde G_n),1/2)$.
In particular, this implies with~\eqref{binomial_exceeds_expectation} that $P[H_t(G_n) \ge  Y_t(\tilde G_n)/2 ~|~ \mathcal{Y}_t(\tilde{G}_n) ] \ge 1/2$, as claimed.

We continue with the case $P[Z_t(\tilde G_n)\geq c |I_t|]>f^*$.  Let $\mathcal{Z}_t(\tilde G_n)=\{z_1,\dots, z_{|\mathcal{Z}_t(\tilde G_n)|}\}$. Then, for any $i \in \{1,\dots, |\mathcal{Z}_t(\tilde G_n)|\}$ let $n_i:= c_{z_i}(\tilde G_n)\geq 2$, that is, there are distinct vertices $v_{i,1},\dots,v_{i,n_i}$ such that $X_{v}(\tilde G_n)=z_i$ for all $v \in \{v_{i,1},\dots,v_{i,n_i}\}$. We will show that 
\begin{equation}
\label{eq:Zcontr}
	P\big[Z_t(G_n)
			\geq
			Z_t(\tilde G_n)/8 \mid \mathcal Z_t(\tilde G_n), n_1, \dots, n_{|\mathcal{Z}_t(\tilde G_n)|}\big]
	\geq 1/8
\end{equation}
and then, since by assumption $P[Z_t(\tilde G_n)\geq c |I_t|]>f^*$, we obtain $P[Z_t(G_n)\geq c/8 |I_t|]\geq f^*/8$ which contradicts Lemma \ref{increaseByFactor2lemma}. Due to \eqref{eq:condIndep} the events 
\begin{align}
\label{independent_events_second_case}
\big(\{X_{v_{i,j}}(G_n)=X_{v_{i,j}}(\tilde G_n)\}\big)_{1\leq i \leq |\mathcal{Z}_t(\tilde G_n)|, 1 \leq j \leq n_i}
\end{align}
are independent. 
Moreover, for all $1\leq i \leq |\mathcal{Z}_t(\tilde G_n)|, 1 \leq j \leq n_i$,
\begin{align}
\label{same_push_second_case}
P\left [X_{v_{i,j}}(G_n)=X_{v_{i,j}}(\tilde G_n) \mid \mathcal{Z}_t(\tilde G_n), n_1, \dots, n_{|\mathcal{Z}_t(\tilde G_n)|}\right ] = \frac{d_{\tilde G_n}(v_{i,j})}{d_{G_n}(v_{i,j})} \geq 1/2 + \varepsilon.
\end{align}
For $1\leq i \leq |\mathcal{Z}_t(\tilde G_n)|$ let $B_i \sim$ Bin$(n_i,1/2)$ be independent random variables. Moreover, let $M_1:=\{i \mid 1\leq i \leq |\mathcal{Z}_t(\tilde G_n)|, n_i=2\}$ and  $M_2:=\{i \mid 1\leq i \leq |\mathcal{Z}_t(\tilde G_n)|, n_i>2\}$.
Using  \eqref{independent_events_second_case} and \eqref{same_push_second_case}, given $\mathcal{Z}_t(\tilde G_n)$, $n_1, \dots, n_{|\mathcal{Z}_t(\tilde G_n)|}$, we infer that $Z_t(G_n)$ dominates 
\begin{align*}
	\sum\limits_{i=1}^{|\mathcal{Z}_t(\tilde G_n)|}\max\{B_i-1,0\}
	%= \sum\limits_{i \in M_1} \max\{B_i-1,0\} + \sum\limits_{i \in M_2}\max\{B_i-1,0\}
	\ge \sum_{i \in M_1} \max\{B_i-1,0\} + \sum_{i \in M_2}B_i - |M_2|.
\end{align*}
We treat the two sums individually. Note that $\sum_{i \in M_1} \max\{B_i-1,0\} \sim$ Bin$(|M_1|, 1/4)$;
in particular, $P[\sum_{i \in M_1}\max\{B_i-1,0\}\geq |M_1|/4]\geq 1/4$ by  \eqref{binomial_exceeds_expectation}.
Regarding the second sum, since $\sum_{i \in M_2} B_i \sim$ Bin$(\sum_{i \in M_2}n_i,1/2)$
we obtain $P[\sum_{i \in M_2} B_i \geq 1/2 \sum_{i \in M_2}n_i]\geq 1/2$.
Thus, given $\mathcal{Z}_t(\tilde G_n)$, $n_1, \dots, n_{|\mathcal{Z}_t(\tilde G_n)|}$ and using $2|M_1|= \sum_{i \in M_1}n_i$ and $\sum_{i \in M_2}n_i \geq 3|M_2|$,
we infer that with probability at least $1/4\cdot 1/2 = 1/8$
\begin{align*}
Z_t(G_n) &\geq \frac{1}{4}|M_1| + \frac{1}{2} \sum\limits_{i \in M_2} n_i  -|M_2| ~ = \frac{1}{8} \sum\limits_{i \in M_1}n_i + \frac{1}{2}\sum\limits_{i \in M_2}n_i - |M_2| ~ \geq \frac{1}{8}\sum\limits_{i \in M_1}n_i + \frac{1}{6} \sum\limits_{i \in M_2}n_i\\
&\geq \frac{1}{8} \sum\limits_{i=1}^{|\mathcal Z_t(\tilde G_n)|}n_i = \frac{1}{8} \left (Z_t(\tilde G_n)+|\mathcal{Z}_t(\tilde G_n)|\right )\geq \frac{1}{8}Z_t(\tilde G_n).
\end{align*}
This establishes~\eqref{eq:Zcontr}. All in all, for $q=1$ we have shown that \eqref{increaseByFactor2} does also hold in the new model.
Hence claim $a)$  follows directly from Lemma \ref{increaseByFactor1+q_lemma}.
%Hence, to complete the proof of claim $a)$, we are left to show that for any $q \in (0,1]$ there is a $t^*=o(\log n)$ such that in the new model whp $|I_{t^*}|=\omega(1)$. Let $\tilde N_{t+1}:=\{X_v(\tilde G_n) \mid v \in I_t \} \cap U_t$; note that for $q=1$ it is $\tilde N_{t+1}=|I_{t+1}\backslash I_t|$. As each push succeeds independently of the other pushes with probability $q$, the number of vertices that become informed in round $t$ can be lower bounded by a binomially distributed random variable $Bin(\tilde N_{t+1},q)$. Therefore, since for $q=1$ we have shown that (\ref{increaseByFactor2}) holds in the new model, using Lemma \ref{binomial_exceeds_expectation} we can infer $P[|I_{t+1}|\geq I_t + (1-o(1))q |I_t|] \geq 1/4-o(1)$. This implies that there is a $t^*=o(\log n)$ such that in the new model whp $|I_{t^*}|=\omega(1)$.

Next we prove claim $b)$. We write $\Delta_n:=\Delta(G_n), \tilde \Delta_n:= \Delta(\tilde G_n), \delta_n:=\delta(G_n)$ and $\tilde \delta_n:=\delta (\tilde G_n)$; moreover we write $\tilde N(\cdot)$ instead of $N_{\tilde G_n}(\cdot)$. 
We assume that $|I_t| \in [n/\log n,n-n/\log n]$. We further distinguish two cases, namely $|I_t|\in [n/\log n,n/2]$ and $|I_t|\in [n/2,n-n/\log n]$. We start with the case~$|I_t|\in [n/\log n,n/2]$. Using Lemmas \ref{expmix} and \ref{linearFraction} and the assumption that $\Delta_n/\delta_n=1+o(1)$ we obtain, for any $0<\bar \varepsilon<\varepsilon/2$, for $n$ sufficiently large,
\begin{equation}
\label{bound_for_edges_when_I_t_is_small}
e(I_t,U_t) > \bar \varepsilon \delta_n |I_t|.
\end{equation}
Using that $e^x\geq (1+x/n)^n$ for $n\in \mathbb N$ and $|x|\leq n$ we obtain 
\begin{align*}
\mathbb{E}_{t}[|I_{t+1}\backslash I_t|] \geq \sum\limits_{u \in \tilde N(I_t)\backslash I_t}\left [ 1- \prod\limits_{v\in \tilde N(u) \cap I_t}\left(1- \frac{q}{\tilde \Delta_n}\right) \right ] \geq \sum\limits_{u \in \tilde N(I_t)\backslash I_t}1-e^{-|\tilde N(u)\cap I_t| q/\tilde \Delta_n}.
\end{align*}
Further, using that  $e^{-x}\leq 1-x/2$ for any $x\in(0,1)$ and \eqref{bound_for_edges_when_I_t_is_small} yields the bound
\begin{align*}
\mathbb{E}_{t}[|I_{t+1}\backslash I_t|] \geq \sum\limits_{u \in \tilde N(I_t) \backslash I_t} \frac{q|\tilde N(u)\cap I_t|}{2 \tilde \Delta_n} = \frac{q e(I_t,U_t)}{2 \tilde \Delta_n} \geq \frac{\bar \varepsilon q \delta_n}{2 \Delta_n}|I_t|.
\end{align*}
 For this case the claim follows by Example \ref{concetrationRemark}, when setting $f=n/\log n, g=\log n$ and $c=\bar \varepsilon q \delta_n/(2 \Delta_n)$.

Finally we consider the case $|I_t| \in [n/2,n-n/\log n]$; here we examine the shrinking of $U_t$. Using Lemmas \ref{expmix} and \ref{linearFraction} we obtain, for any $0<\bar \varepsilon < \varepsilon/2$, for $n$ sufficiently large, $e(I_t,U_t)> \bar \varepsilon \delta_n |U_t|.$
Hence, again using that for any $x\in(0,1)$ it holds $e^{-x}\leq 1-x/2$ and that for $n\in \mathbb N$ and $|x|\leq n$ it is $e^x\geq (1+x/n)^n$, we obtain
\begin{align*}
\mathbb{E}_{t}[|U_{t+1}|] &= \sum\limits_{u \in U_t} \prod\limits_{v \in \tilde N(u) \cap I_t}\left (1- \frac{q}{d_{\tilde G_n}(v)}\right ) \leq \sum\limits_{u \in U_t} e^{-|\tilde N(u)\cap I_t|q/ \tilde \Delta_n} \\ 
&\leq \sum\limits_{u \in U_t} 1- \frac{q|\tilde N(u)\cap I_t|}{2 \tilde \Delta_n} \leq |U_t| - \frac{\bar \varepsilon q \delta_n}{2 \tilde \Delta_n}|U_t|\leq  \left(1- \frac{\bar \varepsilon q\delta_n}{2 \Delta_n}\right)|U_t|. 
\end{align*}
Using the tower property of conditional expectation we immediately get
\begin{align*}
\mathbb{E}_{t}[|U_{t + \tau}|] \leq \left(1- \frac{\bar \varepsilon q \delta_n}{2 \Delta_n}\right)^{\tau}|U_t|, \quad \tau \in \N.
\end{align*}
Thus, for $\tau:= -2\log\log(n)/\log(1- \bar \varepsilon q\delta_n/(2 \Delta_n)) = o(\log n)$ we have $\mathbb{E}_t[|U_{t+\tau}|]=o(n/\log n)$.
Hence by Markov's inequality, $P[|U_{t+ \tau}|\geq n/\log n]=o(1)$.
\end{proof}

\subsection{Proof of Theorem \ref{lastPhasePushNotRobust} --- edge deletions slow down \push}\label{push_slow_sec}

Let $I_t^{(\push)} := I_t$. In order to show the claim we construct an explicit sequence of graphs that has the desired property. More precisely, for any $\varepsilon >0$, each $q \in (0,1]$ and $n \in \N$ we will define a graph $G_n(\varepsilon)$ that is obtained by deleting edges from the complete graph on $n$ vertices such that each vertex keeps at least an $(1-\varepsilon)$ fraction of its edges and such that \push slows down significantly. 

We define $G_n(\varepsilon)=(V_1 \cup V_2,E)$ with vertex set $V = V_1 \cup V_2$, where $V_1:=\{1,\dots,\lfloor n/2 \rfloor\}$ and $V_2:=\{\lfloor n/2 \rfloor+1, \dots, n\}$, as follows. We include in $E$ all pairs of vertices that intersect $V_1$ and moreover, we add edges (that now have endpoints only in $V_2$) such that all vertices in $V_2$ have degree $\lceil (1-\varepsilon)n\rceil +1 \pm 1$.
%Denote $U_t^{(i)}:=U_t\cap V_i, I_t^{(i)}:=I_t\cap V_i, i=1,2$.
According to Lemma \ref{propositionFirstPhasePushRobust} $a)$ there is a $t=\log_{1+q}(n) + o(\log n)$ such that whp $|I_t| < n/\log n$. It thus suffices to show that it takes whp at least $(1+\varepsilon/2)q^{-1} \log n $ more rounds to inform all  remaining vertices.

Let $U_t' :=U_t^{(\push)}\cap V_2$.
As $|I_t|<n/\log n$ we have $|U_t'| \geq n/4 $ with plenty of room to spare. 
In the remainder of this proof we will consider a modified process in which vertices have a higher chance of getting informed; in particular we assume that in each round, \emph{all} vertices choose a neighbour independently and uniformly at random and after this round the chosen vertices are informed.
% The runtime in this modified model is obviously a lower bound for the runtime of \push.
%Moreover, as this makes things only worse for us\marginpar{\tiny nicht klar}, we assume that uninformed vertices can obtain the rumour not only from informed neighbours but from all their neighbours. 
%Furthermore let $E_u$ denote the event that $u \in U_t$ does not get informed within the next $\tau := \log_{\exp(q(0.5+\frac{0.5-\varepsilon}{1-\varepsilon}))}(n) - h(n)$ rounds where $h = o(\log(n))$ and $h = \omega(1)$. Let $c_{\varepsilon} := \frac{1-\varepsilon}{1-1.5 \varepsilon}>1$ and note that $\tau =\frac{1-\varepsilon}{1-1.5 \varepsilon}\frac{1}{q}\log(n) -h(n) =c_{\varepsilon} \frac{1}{q} \log(n)-h(n)$.
%Let $h = o(\log n)$, $h = \omega(1)$ and $h \geq 0$.
%Let $c_{\varepsilon}:= 1+ \varepsilon/(2 - 3 \varepsilon)\geq 1+\varepsilon/2 $.
Let $E_u$ denote the event that $u \in U_t'$ does not get informed within the next $\tau := (1+\varepsilon/2)q^{-1} \log n$ rounds in this modified model. Each vertex $u \in U_t'$ has $\lfloor n/2 \rfloor$ neighbours that have degree $n-1$, at most $\lceil(1-\varepsilon)n\rceil+1 \pm 1 -\lfloor n/2\rfloor \le (1/2-\varepsilon)n+4$
neighbours that have at least degree $(1-\varepsilon)n$  and no further neighbours. 
%It suffices to show that $$P\Bigl[\bigwedge\limits_{u \in U_t^{(2)}} \overline{E_{u}}\Bigr]=o(1).$$
Therefore, using that for any $a\in \R$ we have $(1+a/n)^n=e^{a}+\mathcal{O}(1/n)$, we obtain for each $u \in U_t'$
\begin{align*}
%P[\bar{E}_u]\geq
P_t[E_{u}]&\geq \left(\left(1- \frac{q}{n-1}\right)^{n/2}\left(1-\frac{q}{\left(1-\varepsilon\right)n}\right)^{\left(1/2 - \varepsilon\right)n+4}\right)^{\tau}
= 
(1+o(1))\left(e^{-q\left(1/2+(1/2-\varepsilon)/(1-\varepsilon)\right)}\right)^{\tau} \\
&= (1+o(1)) \exp\left({-\frac{4-4\varepsilon-3 \varepsilon^2}{4-4 \varepsilon}\log n}\right) =\omega(n^{-1}).
\end{align*}
%and as the events $E_u$ are independent and $|U_t|=\Theta(n)$ 
In this modified model the events $\{\overline{E_{u}}\mid u\in U_t'\}$ also satisfy $P_t[\overline{E_{u}}\mid \{\overline{E_v}:v\in U\}]\leq 1-p$ for all $u\in V_2$ and $U\subseteq V\setminus \{u\}$ and for some $p =\omega(n^{-1})$. This follows immediately from the previous calculation, as conditioning on an event like ``$\{\overline{E_v}:v\in U\}$'' only decreases the number of vertices that can push to $u$. Thus as $|U_t'|=\Theta(n)$
\begin{align*}
P_t\left[\bigwedge\limits_{u \in U_t'} \overline{E_{u}}\right] 
%\leq P\left[\bigwedge\limits_{u \in V} \overline{E_{u}}\right]
\le \prod\limits_{u \in U_t'}(1-p)
%= \prod\limits_{u \in V} (1-P[E_{u}])
 \leq \exp\left(- \sum\limits_{u \in U_t'} p\right)=o(1).
\end{align*}

\subsection{Proof of Theorems \ref{pushPullIsRobust},~\ref{firstPhasePushPullRobust} -- \pushpull informs almost all vertices fast in spite of edge deletions}
\label{push_pull_informs _almost_all nodes_fast}

Before we show the actual proof we will first present an informal argument that contains all relevant ideas and important observations. Let $\sqrt{\log n}\leq |I_t| \leq n/\log n$ and assume $q=1$. In Section \ref{push_robust_sec} we proved that for \push the informed vertices nearly double in every round for an arbitrary expander sequence with edge deletions and an otherwise arbitrary set $I_t$. For \pull this is not true; however, we proved in Section \ref{pull_fast_sec} that the number of edges between the informed and the uninformed vertices nearly doubles in every round. The first attempt towards the proof of Theorems \ref{pushPullIsRobust},~\ref{firstPhasePushPullRobust} then seems obvious: one would try to show that either the vertices triple every round, or the the edges do so, or for example that the product of the two quantities increases by a factor of 9. As it turns out, this is in general not the case; indeed, it is possible to choose an expander sequence, to delete edges such that each vertex keeps at least an ($1/2 + \varepsilon$)-fraction of its neighbors, and to choose a (large) set of informed vertices $I_t$ such that after one round whp either $|I_{t+1}| < c |I_t|$ or $e(I_{t+1}, U_{t+1}) < ce(I_{t}, U_{t})$ or $|I_{t+1}|e(I_{t+1}, U_{t+1}) < c^2|I_t|e(I_{t}, U_{t})$ for some $c < 3$. On the other hand and although we have no explicit description of these `malicious' sets, it seems rather unlikely that such sets will occur several times during the execution of \pushpull.

%\setlength{\unitlength}{0.1cm}
%\linethickness{1pt}
%\begin{figure}[htbp]
%\centering
%\begin{picture}(70,70)
%\put(10,10){\circle{20}}
%\put(60,10){\circle{20}}
%\put(10,60){\circle{20}}
%\put(60,60){\circle{20}}
%\put(8,9){$V_1$}
%\put(58,9){$V_2$}
%\put(8,59){$V_3$}
%\put(58,59){$V_4$}
%\put(20, 10){\line(1, 0){30}}
%\put(20, 60){\line(1, 0){30}}
%\put(10, 20){\line(0, 1){30}}
%\put(60, 20){\line(0, 1){30}}
%\put(17, 53){\line(1, -1){35}}
%\put(17, 17){\line(1, 1){35}}
%\put(35,12){$d_{12}$}
%\put(35,62){$d_{34}$}
%\put(2,35){$d_{13}$}
%\put(62,35){$d_{24}$}
%\put(25,20){$d_{14}$}
%\put(25,47){$d_{23}$}
%\end{picture}
%\caption{Let $G=(\bigcup_{i=1}^4V_i,E)$. For $1\leq i < j\leq 4$ assume that the subgraph induced by the pair $(V_i,V_j)$ is a random $d_{ij}$ -regular bipartite graph. Then the density $d(V_i,V_j)$ is given by $d_{ij}\cdot4/n$ and $\forall~v\in V_i:d(v)=\sum_{j\neq i}d_{ij}$. }
%\end{figure}

In order to show the claimed running time of \pushpull we will impose some additional structure. Let $\varepsilon > 0$. In the subsequent exposition we assume that our graph $G$ -- obtained from an expander by deleting edges such that each vertex keeps at least an $(1/2+\varepsilon)$ fraction of the edges -- has a \emph{very  special} structure. In particular, we assume that there is a partition $\Pi=(V_i)_{i\in [k]}$ of the vertex set of $G$ into a bounded number $k$ of equal parts such that $E_G(V_i) = \emptyset$ for all $1\le i \le k$ and such that  the induced subgraph $(V_i,V_j)$ looks like a random regular bipartite graph for all $1 \le i < j \le k$. Of course, not every relevant $G$ admits such a partition; however, Szemeredi's regularity lemma guarantees that every sufficiently large graph has a partition that is in a well-defined sense \emph{almost} like the one described previously, and a substantial part of our proof is concerned with showing that being `almost special' does not hurt significantly.

Assuming that $G$ is very special let us collect some easy facts. 
Denote the degree of  $u\in V_i$ in the induced subgraph $(V_i,V_j)$ with $d_{ij}$; this immediately  gives that $d_G(u)=\sum_{\ell=1}^kd_{i\ell}$, and note that $d_{ii}=0$ as there are no edges in $V_i$. Moreover, regular bipartite random graphs fulfil an expander property, that is, 
\[
	e(W_i,W_j) = d_{i,j}|W_i||W_j|/|V_j|+o(d_{i,j})|W_i| \approx |W_i||W_j|d_{ij}k/n
	\quad
	\text{for all $W_i\subseteq V_i, W_j\subseteq V_j$}, 1\le i < j \le k,
\]
where we used that all $|V_i|$'s are of equal size. 
This is quite similar to the property that we used in our preceding analysis on expander sequences, see Lemma \ref{expmix}. As a pair in $\Pi$ behaves like a bipartite expander sequence we can easily compute the expected number of informed vertices like we did in Section \ref{push_pull_fast_subsec}. We do so now for \pull. Let $\big|I_{t+1}^{i,j}\big|$ be the number of vertices in $V_i$ informed after round $t+1$ by \pull from vertices only in $V_j$ and set $I_t^i:=I_t\cap V_i, U_t^i:=U_t\cap V_i~\forall ~1\leq i\leq k$. Thus, as long as $I_t^i$ is much smaller than $V_i$ (and thus also $U_t^i \approx |V_i| = n/k$) we get
\begin{align*}
\mathbb{E}_t\left[\big|I_{t+1}^{(pull), i,j}\backslash I_t\big|\right]=\sum_{u\in U_t^i}\frac{|N(u)\cap I_t^j|}{d(u)}=\frac{e(U_t^i,I_t^j)}{\sum_{1\le \ell\le k}d_{i\ell}} \approx %(1+o(1))\frac{d_{ij}|U_t^i||I_t^j|k/n}{\sum_{\ell=1}^kd_{i\ell}}=(1+o(1))
\frac{d_{ij}}{\sum_{1\le \ell\le k}d_{i\ell}}|I_t^j|.
\end{align*}
A similar calculation, which we don't perform in detail, yields for \push 
\begin{align*}
\mathbb{E}_t\left[\big|I_{t+1}^{(push), i,j}\backslash I_t\big|\right] \approx  \frac{d_{ij}}{\sum_{1\le \ell\le k}d_{\ell j}}|I_t^j|.
\end{align*}
Moreover, as in previous proofs it turns out that the number of vertices informed simultaneously  by \push as well as \pull is negligible, compare with the proof of Lemma \ref{expPushPull}. Thus we obtain that more or less
\begin{align*}
\mathbb{E}_t\left[\big|I_{t+1}^{(pp), i,j}\big|\right]
\approx
|I_t^i|+\left(\frac{d_{ij}}{\sum_{1\le \ell\le k}d_{i\ell}}+\frac{d_{ij}}{\sum_{1\le \ell\le k}d_{\ell j}} \right)|I_t^j|
\end{align*}
and by linearity of expectation
\begin{align*}
\mathbb{E}_t\left[\big|I_{t+1}^{(pp), i}\big|\right]
\approx 
|I_t^i|+ \sum_{1\le j\le k}\left(\frac{d_{ij}}{\sum_{1\le \ell\le k}d_{i\ell}}+\frac{d_{ij}}{\sum_{1\le \ell\le k}d_{\ell j}}\right)|I_t^j|.
\end{align*}
Set $X_t=(|I_t^i|)_{i\in [k]}$ and $A=(A_{ij})_{1\leq i,j\leq k}$, the matrix with entries
 $$A_{ij}=\frac{d_{ij}}{\sum_{1\le \ell\le k}d_{i\ell}}+\frac{d_{ij}}{\sum_{1\le \ell\le k}d_{\ell j}}\qquad \text{for }1 \le i \neq j \le k$$
and $A_{ii}=1$  for $1\leq i\leq k$. With this notation we obtain the recursive relation 
\begin{equation}
\label{eq:masterRec}
\mathbb{E}_t[X_{t+1}] \approx A\cdot X_t,
\end{equation}
that is, we may expect that $X_t \approx \mathbb{E}_t[X_{t}] \approx A^t X_0$. If we then denote by $\lambda_{\max}$ the greatest eigenvalue of $A$, then we obtain that in leading order
$$|I_t| \approx \lambda_{\max}^t.$$
Our aim is to show that \pushpull is (at least) as fast as on the complete graph, that is, $|I_t| \precsim 3^t$, and so  we take a closer look at the eigenvalues of $A$. By construction $A$ is symmetric, so that the largest eigenvalue equals $\sup_{\|x\|=1}\|x^TAx\|$, and the simple choice $x = k^{-1/2}\mathbf{1}$ yields
$$\lambda_{\max} \geq \frac{\sum_{(i,j)}A_{i,j}}{k}= \frac{\sum_{j=1}^k 1+\sum_{i=1}^k\sum_{j=1}^kd_{ij}/\left(\sum_{\ell=1}^kd_{i\ell}\right)+\sum_{j=1}^k\sum_{i=1}^k d_{ij}/\left(\sum_{\ell=1}^kd_{\ell j}\right)}{k}=3.$$ 
This neat property leads us to the expected result $T_{pp}(G) = (1+o(1))\log_{\lambda_{\max}}n \le (1+o(1))\log_3n$, and it also completes the informal argument that justifies the claim made in Theorems \ref{pushPullIsRobust} and ~\ref{firstPhasePushPullRobust}. In the rest of this section we will turn this argument step by step into a formal proof by filling in all missing pieces.

\paragraph{Obtaining an Appropriate  Regular Partition}

An important ingredient in the previous sketch was the assumption that the given graph has a partition into a bounded number of equal parts, such that the bipartite graph induced by any two different parts looks like a random regular graph. This assumption is quite strong and very much not true in general. However, restricting ourselves to dense graphs we can actually come quite close to that. Let us begin with some definitions; the statements are taken from~\cite{rodl2010regularity}. 
\begin{definition}[Density]
Given a graph $G=(V,E)$ and two disjoint non-empty sets of vertices $X,Y\subseteq V$, we define the \emph{density} of the pair $(X,Y)$ as
$$d_G(X,Y)=\frac{e_G(X,Y)}{|X||Y|}.$$
\end{definition}
As usual, if the graph is clear from the context the index will be omitted. The next definition gives a partition that is close to the previously described properties; all sets in the partition have nearly the same size and nearly all pairs behave in a well-defined sense like regular bipartite random graphs.   
\begin{definition}[$(\varepsilon,k_0,K_0)$-Szemer\'edi partition]\label{SzPartition}
Let $G=(V,E)$ and $k\in \mathbb{N}$. We call  $\Pi=\{V_i\}_{i\in[k]}$ an \emph{$(\varepsilon,k_0,K_0)$-Szemerd\'edi partition} of $G$ if the following conditions are fulfilled.
\begin{enumerate}
\itemsep0em 
\item $V_1\dot\cup\dots\dot\cup V_k=V$,
\item $k_0\leq k\leq K_0$,
\item $|V_1| \leq \dots\leq |V_k| \leq |V_1|+1$,
\item for all but at most $\varepsilon k^2$ pairs $(V_i,V_j)$ of $\Pi$ with $i<j$ we have that for all subsets $U_i\subseteq V_i$ and $U_j\subseteq V_j$ with $|U_i| \geq \varepsilon|V_i|$ and $|U_j| \geq\varepsilon|V_j|$
$$|d(U_i,U_j)-d(V_i,V_j)| \leq \varepsilon. $$
\end{enumerate}
A pair $(V_i,V_j)$ satisfying the last condition is called \emph{$\varepsilon$-regular}. For pairs $(V_i,V_j)$ in $\Pi$ we will abbreviate $d(V_i,V_j)$ with $d_{ij}$.
\end{definition}
Next we state Sz\'emeredis Regularity Lemma. It guarantees that we will have a Szemer\'edi partition if the underlying graph is large enough. 
\begin{lemma}[\cite{rodl2010regularity}, The Regularity Lemma]\label{Szemeredi}
For every $\varepsilon >$ 0 and every $k_0\in \mathbb{N}$ there exist $K_0 = K_0 (\varepsilon, k_0)$ and
$n_0$ such that every graph $G = (V, E)$ with at least $|V | = n \geq n_0$ vertices admits an \emph{$(\varepsilon,k_0,K_0)$-Szemerd\'edi partition}.
\end{lemma}
The next lemma gives a useful property of regular pairs. In particular, with the exception of a small set only, all other vertices have a degree that is close to $dN$, where $d$ is the density of the pair and $N$ the number of vertices in each part. Actually, the statement also is true for arbitrary but not too small subsets of the parts.
\begin{lemma}{\label{fineRegularity}}
Let $G=(V,E)$ be a graph, $\varepsilon>0$ and $U, U'\subseteq V$. Suppose that $(U,U')$ is an $\varepsilon$-regular pair, and let  $W\subseteq U',|W|\geq \varepsilon |U'|$. Let furthermore $\mathcal{E}(U,W)\subseteq U$ be the largest set such that 
$|d(u,W)-d(U,U')|\geq \varepsilon$ for all $u\in \mathcal{E}(U,W)$.
Then $|\mathcal{E}(U,W)|\leq 2\varepsilon|U|$.
\end{lemma}
\begin{proof}
We will prove this by contradiction. Assume that $|\mathcal{E}(U,W)|\geq 2\varepsilon|U|$. Let us write $\mathcal{E}(U,W) = S \cup L$, where $S = \{u \in {\cal E}(U,W): d(u,W)<d(U,U')-\varepsilon\}$ and $L = \{u \in {\cal E}(U,W): d(u,W)>d(U,U')+\varepsilon\}$. Then $|S| \ge \varepsilon |U|$ or $|L| \ge \varepsilon |U|$. In the former case
$$d\left(S,W\right)
%=  \frac{e(\bigcup_{u\in S}u,W)}{\left|\bigcup_{u\in S}u\right||U'|}
=\frac{\sum_{u\in S}e(u,W)}{|S| \, |W|}
=\frac{\sum_{u\in S}d(u,W)}{|S| }<{d(U,U')-\varepsilon}.
%=d(U,W)-\varepsilon
$$
As $|S|\geq \varepsilon |U|, |W|  \ge \varepsilon |U'|$, this contradicts the assumption that $(U,U')$ is an $\varepsilon$-regular pair. The case $|L| \ge \varepsilon |U|$ follows analogouesly by showing that $d\left(L,W\right) > d(U,U')+\varepsilon$.
\end{proof}
We call the set $\mathcal{E}(U,W)$ in Lemma \ref{fineRegularity} the exceptional set of $U$ with respect to $W$. In particular Lemma \ref{fineRegularity} implies that for every $\varepsilon$-regular pair $(U,U')$ and \emph{all} $W\subseteq U', |W|\geq (1-c\varepsilon)|U'|,c>0$ we have
\begin{equation}\label{fineRegularityappl}
	|d(u,W)-d(U,U')|
	\le |d(u,W)-d(u,U')| + |d(u, U') - d(U,U')|
	\leq (c+1)\varepsilon~~\text{for all}~~u\in U\backslash \mathcal{E}(U,U').
\end{equation}
Having done these preparations we can now determine a partition that comes close to the initially described properties.
\begin{lemma} \label{Partition}
Consider the setting of Theorems \ref{pushPullIsRobust},~\ref{firstPhasePushPullRobust}. Then for all $\eta > 0$ and $k_0>1/\sqrt{\eta}$ there exists $n_0, K_0\in \mathbb{N}$ such that for all $\tilde{G}_n$ with $n\geq n_0$ there is a $(\eta,k_0,K_0)$-Szemer\'edi partition $\Pi=\{V_i\}_{i\in [k]}$ of~$\tilde{G}_n$ with the following property. There is  $F\subseteq \Pi$ with $|F|\leq \eta k $ such that for all $V_i\in \Pi\backslash F$
\begin{itemize}
\itemsep0em 
\item  there are at most $\eta k$ non-$\eta$-regular pairs $(V_i,V_j)$, $j \in [k]$, and
\item  there exists an exceptional set $N_i, |N_i|\leq  \eta |V_i| $ such that
 $$d(u)\leq (1+ \eta)\frac{n}{k}\sum_{1\le j\le k}d(V_i,V_j)  \qquad \text{for all $u\in V_i\backslash N_i$.}$$ 
\end{itemize}
\end{lemma}
\begin{proof}
According to Lemma \ref{Szemeredi}, for all $\xi > 0$ and $k_0>1/\sqrt{\xi}$, there are $n_0,K_0\in \mathbb{N}$ such that for all $\tilde{G}_n$ with $n\geq n_0$ there is a $k\in \mathbb{N}$ and a $(\xi,k_0,K_0)$-Szemer\'edi partition $\Pi=\{V_i\}_{i\in [k]}$ of $\tilde{G}_n$.
Let $F\subseteq \Pi$ contain the parts $V_i\in\Pi$ such that there are at least $\sqrt{\xi}k$ other parts $V_j\in \Pi$ such that the pair $(V_i,V_j)$ is not $\xi$-regular.
As there are at most $\xi k^2$ non $\xi$-regular pairs, we infer that $|F|\leq \sqrt{\xi}k$. Let $V_i\in \Pi\setminus F$.
Let further $A_i\subseteq \Pi$ be such that $(V_i,V_j)$ is a $\xi$-regular pair for all $V_j\in \Pi\setminus A_i$ and $(V_i,V_j)$ is not $\xi$-regular for all $V_j\in A_i$.
The definition of $F$ implies that $|A_i|\leq \sqrt{\xi}k$. For these $V_j\in\Pi\setminus A_i$ let $\mathcal{E}_i(V_j)=\mathcal{E}(V_i,V_j)$ be the exceptional set of $V_i$ with respect to $V_j$.
On top of that let $N_i\subseteq V_i$ be the set of points in $V_i$ that are in at least $\sqrt{\xi}k$ exceptional sets with respect to parts in $\Pi\setminus A_i$. As there are at most $k$ exceptional sets and by Lemma \ref{fineRegularity} each exceptional set has at most $2\xi|V_i|$ vertices, we get that $|N_i|\leq 2\sqrt{\xi}|V_i|$. Let $V_i\in\Pi\backslash F$, $u\in V_i\backslash N_i$ and let $B(u)\subseteq \Pi\backslash A_i$ be the set of parts such that $u\in \mathcal{E}_i(V_j)$ for all $V_j\in B$. Then $|B|\leq \sqrt{\xi}k$ and 
\begin{align*}
d(u)&=\sum_{1\le j\le k}|V_j| d(u,V_j) =\left( \sum_{V_j\in A_i\cup B} |V_j|d(u,V_j)+\sum_{V_j\in \Pi\setminus(A_i\cup B)} |V_j|d(u,V_j)  \right)
\\&\leq \biggl| N(u)\cap (\bigcup_{V_j\in A_i\cup B\cup \{V_i\}}V_j)\biggr|+\sum_{1\le j\le k}|V_j| \left(d(V_i,V_j)+\xi\right).
\end{align*}
By the definition of $F$ and as $u\in V_i\backslash N_i$ we get that $|\bigcup_{V_j\in A_i\cup B\cup \{V_i\}}V_j|\leq (\sqrt{\xi}k+\sqrt{\xi}k+1)(n/k+1)\leq 3\sqrt{\xi}n$. With that at hand and by using $d(u)\geq \alpha n/2$ and that the sizes of the parts in $\Pi$ differ by at most one we obtain
\begin{align*}
d(u)&\leq 3\sqrt{\xi}n+\frac{n}{k}\sum_{1\le j\le k} d(V_i,V_j)+2\xi n\leq \frac{n}{k}\sum_{1\le j\le k} d(V_i,V_j)+10\sqrt{\xi}d(u)/\alpha  .
\end{align*}
Let $\eta>0$. Choosing $\xi$ small enough such that $\max\{\xi, 2\sqrt{\xi}, 1/(1-10\sqrt{\xi}/\alpha)-1\}\leq \eta$ implies the claim.
\end{proof}

\paragraph{The Recursion Relation.} 
In this section we exploit the properties of the partition to study the expected number of informed vertices after one additional round; our aim is to establish a precise version of~\eqref{eq:masterRec}. In the remainder let $\|A\|_F = (\sum_{1\leq i\leq n}\sum_{1\leq j\leq n}|a_{i,j}|^2)^{1/2}$ denote the Frobenius norm of a matrix $A\in \mathbb{R}^{n\times n}$.

For the next lemma  consider the setting of Theorems \ref{pushPullIsRobust}, \ref{firstPhasePushPullRobust}, i.e., we are given an expander sequence $(G_n)_{n\in \mathbb{N}}$ with minimal degree $\delta_n\geq \alpha n$ for some $\alpha>0$ and an $\varepsilon>0$. We obtain a sequence of graphs $(\tilde{G}_n)_{n\in \mathbb{N}}$ by deleting up to a $1/2-\varepsilon$ fraction of the edges at each vertex in $G_n$. Let further $\eta >0, k_0\in\mathbb{N}$  and $\Pi=\{V_i\}_{i\in[k]}$ be the $(\eta,k_0,K_0)$-Sz\'emeredi partition of $\tilde{G}_n$ as given by Lemma \ref{Partition}. For that partition define $\mathcal{E}_{i,j}:=\mathcal{E}(V_i,V_j) $ as the exceptional set of $V_i$ with respect to $ V_j$ given by Lemma \ref{fineRegularity}, $i\neq j\in [k]$,  $F$ and $N_i$ as the exceptional sets from Lemma \ref{Partition}, $i \in \Pi\setminus F$. Moreover, let $\Pi_i = \{V_j \in \Pi \setminus F: (V_i,V_j) \text{ is $\eta$-regular}\}$ and note that
\[
	|\Pi_i| \ge (1-2\eta)k,
	|N_i|\le \eta |V_i|,
	|\mathcal{E}_{i,j}|\leq 2\eta |V_i|
	\quad
	\text{for all } i\in \Pi\setminus F, j\in\Pi_i.
\]
Finally, define
\[
	X_{t,i,j} = \bigl| I_t \cap \big(V_i \setminus (N_i \cup {\cal E}_{i,j})\big)\bigr|, \quad i\in \Pi \setminus F \text{ and } j\in\Pi_i
\]
and
\[
	X_{t,i} =\min_{j\in\Pi_i} X_{t,i,j}, \quad i\in \Pi \setminus F.
\]
This definition guarantees that $|I_t| \ge \| X_{t} \|_1$. The cornerstone of our proof is the following lemma, which bounds the growth of $X_t = (X_{t,i})_{i\in \Pi\setminus F}$ after one round.
\begin{lemma}\label{RegularRecursion}
Consider the situation as described above and assume additionally that $|X_{t,i}|\geq \log\log n$ for all $i\in\Pi\setminus F$ and that $|I_t| \leq n/\log n$. Then for all $\nu>0$ and $n$ large enough there exists a symmetric matrix~$A$ with biggest eigenvalue $\lambda_{\max}\geq 1+2q -\nu$ and an error matrix $\Delta A$ with $\|\Delta A\|_F\leq \nu$ such that whp
$$
%\mathbb{E}_t[|I_{t+1}|]\geq\mathbb{E}_t[\|X_{t+1}\|_1]\quad \text{and}\quad 
X_{t+1} \geq (A+\Delta A)X_t. $$
\end{lemma}
\begin{proof}We set $I_t^{{\cal P},i}=I_t^{\cal P}\cap V_i,~U_t^{{\cal P},i}=U_t^{\cal P}\cap V_i$ for ${\cal P}\in \{push,pull, pp\}$ and let $$I_{t+1}^{{\cal P},i,j}\backslash I_t=\{u\in U_t\cap V_i\mid \text{there is } v\in I_t\cap V_j \text{ such that $u$ gets informed by $v$ using $\cal P$}\}$$ be the vertices in $V_i$ newly informed in round $t+1$  by operations involving only vertices from $V_i$ and $V_j$. Let $(i,j)\in\Pi\setminus F$. For all $u\in U_t^i$ we know that $d(u)\geq \alpha n/2 $. Moreover, $|I_t^i|\leq|I_t|\leq n/\log n$. Thus, the probability of $u\in U_t^i$ being informed by vertices in $I_t^j$ via \pull is 
${q|N(u)\cap I_t^j|}/{|N(u)|}=o(1)$. As the events of $u$ being informed by \push and \pull are independent  $P[u \in I_{t+1}^{(\push),i,j}\cap I_{t+1}^{(\pull),i,j}]=o(1)P[u \in I_{t+1}^{(\push),i,j}]$. Thus for any set $U\in V$
\begin{align}\label{combinePushPull}
\mathbb{E}\Big[\big|(I_{t+1}^{(pp),i,j}\setminus I_t)\cap U\big|\Big]
= (1-o(1))\left(\mathbb{E}\left[\bigl|(I_{t+1}^{(pull),i,j}\setminus I_t) \cap U\bigr|\right]+\mathbb{E}\left[\bigl|(I_{t+1}^{(push),i,j}\setminus I_t)\cap U\bigr|\right]\right).
\end{align}
Let $i\in \Pi\backslash F$ and $j\in \Pi_i$. We start by determining the expected number of vertices informed by \pull. Set further $D_i=(1+ \eta)\frac{n}{k}\sum_{1\le \ell \le k}d_{i\ell}$. According to Lemma \ref{Partition} all $v\in {U}^i_t\setminus N_i$ have degree less than $D_i$. Let $j' \in \Pi_i$  and set $\mathcal{H}_{i,j'}= N_i\cup \mathcal{E}_{i,j'}$. Then
\begin{equation*}
\begin{aligned}
\mathbb{E}_t\left[\bigl|I_{t+1}^{(\pull),i,j}\backslash (I_t\cup \mathcal{H}_{i,j'})\bigr|\right]
&=\sum_{u\in  U_t^i\setminus \mathcal{H}_{i,j'}}q\frac{|N(u)\cap I_t^j|}{|N(u)|} \geq q\frac{e(U_t^i\setminus \mathcal{H}_{i,j'},I_t^j)}{D_i}.
\end{aligned}
\end{equation*}
Since  $|I_t^i|\leq |I_t|\leq n/\log n$ we get with room to spare that $|{U_t^i}\setminus \mathcal{H}_{i,j'}|\geq (1-5\eta)n/k$ for $n$ large enough and all $j'\in \Pi_i$. Applying~\eqref{fineRegularityappl}, where we choose $W=U_t^i\setminus \mathcal{H}_{i,j'}$, yields $|d(U_t^i\setminus \mathcal{H}_{i,j'},u)-d_{ij}|\leq 6\eta$ for all $u\in V_j\setminus \mathcal{E}_{j,i}$. Thus 
\begin{align*}
\mathbb{E}_t\left[\bigl|I_{t+1}^{(\pull),i,j}\backslash (I_t\cup \mathcal{H}_{i,j'})\bigr|\right]
\geq q\frac{(d_{ij}-6\eta)|U_t^i\setminus \mathcal{H}_{i,j'}||I_t^j\backslash \mathcal{E}_{j,i}|}{D_i}\geq (1-5\eta)q\frac{(d_{ij}-6\eta)|I_t^j\backslash (\mathcal{E}_{j,i}\cup N_j)|}{D_ik/n}.
\end{align*}
As $D_i=(1+ \eta)n/k\sum_{1\le \ell \le k}d_{i\ell}$ we get for 
$$c_1:=(1-6\eta)(1+ \eta)^{-1}$$
with $X_{t,j,i} = |I_t^j\backslash (\mathcal{E}_{j,i}\cup N_j)|$ that
\begin{align}\label{pulleq}
\mathbb{E}_t\left[\bigl|I_{t+1}^{(\pull),i,j}\backslash (I_t\cup \mathcal{H}_{i,j'})\bigr|\right]\geq c_1 \cdot q\frac{(d_{ij}-6\eta)X_{t,j,i}}{\sum_{1\le \ell \le k}d_{i\ell}} \qquad \text{for all } i\in \Pi\setminus F \text{ and } j,j' \in \Pi_i.
\end{align}
We continue with \push. Let $i\in \Pi\backslash F$ and $j,j'\in \Pi_i$, and set (as before) $D_j=(1+ \eta)\frac{n}{k}\sum_{1\le \ell \le k}d_{\ell j}$ and $\mathcal{H}_{i,j'}= N_i\cup \mathcal{E}_{i,j'}$. Then 
\begin{align*}
\mathbb{E}_t\left[\bigl|I_{t+1}^{(\push),i,j}\backslash (I_t\cup \mathcal{H}_{i,j'})\bigr|\right]
&=\sum_{u\in U_t^i\setminus \mathcal{H}_{i,j'}}\Bigl(1-\prod_{v\in N(u)\cap I_t^j }\bigl(1-\frac{q}{|N(v)|}\bigr)\Bigr)
%\geq \sum_{u\in U_t^i\setminus \mathcal{H}_{i,j'}}\Bigl(1-\prod_{v\in N(u)\cap I_t^j\backslash N_j }\bigl(1-\frac{q}{D_i}\bigr)\Bigr).
\end{align*}
According to Lemma \ref{Partition} all $v\in I_t^j\backslash N_j$ have degree less than $D_j$. Using the inequalities $(1-1/n)^n\leq e^{-1}$ and $e^{-1/n}= (1-1/n)+o(1/n)$ we obtain the estimate 
\begin{equation}\label{starstar}
\begin{aligned}
&\mathbb{E}_t\left[\bigl|I_{t+1}^{(\push),i,j}\backslash (I_t\cup \mathcal{H}_{i,j'})\bigr|\right]\geq
\sum_{u\in U_t^i\setminus \mathcal{H}_{i,j'}}\left(1-\Bigl(1-\frac{q}{D_j}\Bigr)^{|N(u)\cap (I_t^j\backslash N_j)|}\right)\\
\geq & 
\sum_{u\in U_t^i\setminus \mathcal{H}_{i,j'}}\left(1-\exp\Bigl(-q\frac{|N(u)\cap (I_t^j\backslash N_j)|}{D_j}\Bigr)\right)\geq(1-o(1))
\sum_{u\in U_t^i\setminus \mathcal{H}_{i,j'}}q\frac{|N(u)\cap (I_t^j\backslash N_j)|}{D_j}.
\end{aligned}
\end{equation}
The remaining steps are similar to the previously considered case of \pull.
By assumption we have that $|I_t^j\backslash \mathcal{H}_{j,i}|= X_{t,j,i}$ and as $|I_t^i|\leq |I_t|\leq n/\log n$ we obtain that $|U_t^i\backslash \mathcal{H}_{i,j'}|\geq (1-5\eta)n/k$ for $n$ large enough and all $j'\in \Pi_i$.  Using \eqref{fineRegularityappl} we obtain that $|d(U_t^i\setminus \mathcal{H}_{i,j'},u)-d_{ij}|\leq 6\eta$ for all $u\in V_j\setminus \mathcal{E}_{j,i}$. Thus
\begin{align*}
\mathbb{E}_t\left[\left|I_{t+1}^{(\push),i,j}\backslash (I_t\cup\mathcal{H}_{i,j'})\right|\right]&\geq (q-o(1)) \frac{e(U_t^i\setminus \mathcal{H}_{i,j'},I_t^j\backslash (N_j\cup \mathcal{E}_{j,i}))}{D_j}\geq (q-o(1))\frac{(d_{ij}-6\eta)|U_t^i\setminus \mathcal{H}_{i,j'}|X_{t,j,i}}{D_j}.
\end{align*}
Using that $D_j=(1+ \eta)n/k\sum_{1\le \ell \le k}d_{\ell j}$, we get for the same constant $c_1$ as in \eqref{pulleq} and $n$ large enough 
\begin{align}\label{pusheq}
\mathbb{E}_t\left[\left|I_{t+1}^{(\push),i,j}\backslash (I_t\cup \mathcal{H}_{i,j'})\right|\right]\geq c_1 \cdot q\frac{(d_{ij}-6\eta)X_{t,j,i}}{\sum_{1\le \ell \le k}d_{\ell j}}\qquad \text{for all } i\in \Pi\setminus F \text{ and } j,j' \in \Pi_i.
\end{align}
With \eqref{combinePushPull}, we can combine the results for \pull, \eqref{pulleq}, and \push, \eqref{pusheq}, to get for $c_2 :=  c_1 - \eta$
%for all $i\in \Pi\setminus F$ and $ j,j' \in \Pi_i$
\begin{align}\label{pushpulleq}
\mathbb{E}_t\left[\left|I_{t+1}^{(pp),i,j}\backslash (I_t^i\cup \mathcal{H}_{i,j'})\right|\right]&\geq c_2 \cdot q
\left(\frac{d_{ij}-6\eta}{\sum_{1\le \ell \le k}d_{i\ell}}+\frac{d_{ij}-6\eta}{\sum_{1\le \ell \le k}d_{\ell j}}\right)X_{t,j,i}~~ \text{for all } i\in \Pi\setminus F,  j,j' \in \Pi_i.
\end{align}
Next we will show how we can exploit~\eqref{pushpulleq} to obtain (a lower bound for)  $\mathbb{E}_t[|(I_{t+1}^{(pp),i}\setminus I_t)|]$. Let $i\in\Pi\setminus F$ and $u\in U_t^i$. Using $e^{-1/n +o(1/n)}= 1-1/n, e^{-1/n}= 1-1/n + o(1/n),$ and $|I_t|=o(n)$ we obtain 
\begin{align*}
P_t[u\in I_{t+1}^{(push),i}\setminus I_t]&= 1-\prod_{i\in N(u)\cap I_t}\left(1-\frac{1}{|N(i)|}\right)=1-\exp\left(-(1 - o(1))\sum_{i\in N(u)\cap I_t}\frac{1}{|N(i)|}\right)\\
&=(1-o(1))\sum_{i\in N(u)\cap I_t}\frac{1}{|N(i)|}.
\end{align*}
Let $W\subseteq V$. Using \eqref{combinePushPull}, the previous equation and that $\Pi$ is a partition we get
 \begin{align*}
 \mathbb{E}_t[|(I_{t+1}^{(pp),i}\setminus I_t)\cap W|]&= (1-o(1))\sum_{u\in U_t^i\cap W}\left(\frac{|N(u)\cap I_t|}{|N(u)|}+\sum_{i\in N(u)\cap I_t}\frac{1}{|N(i)|}\right)
 \\&=(1-o(1))\sum_{u\in U_t^i\cap W}\left(\sum_{j\in [k]}\Biggl(\frac{|N(u)\cap I_t\cap V_j|}{|N(u)|}+\sum_{i\in N(u)\cap I_t\cap V_j}\frac{1}{|N(i)|}\Biggr)\right)\\
 &=(1-o(1))\sum_{j\in [k]} \mathbb{E}_t[|(I_{t+1}^{(pp),i,j}\setminus I_t)\cap W|].
 \end{align*}
Choose $W=V\setminus \mathcal{H}_{i,j'}$, then the previous equation implies
\begin{align*}
\mathbb{E}_t\left[\left|I_{t+1}^{(pp),i}\backslash (I_t\cup \mathcal{H}_{i,j'})\right|\right]
&\geq (1-o(1))
\sum_{j\in\Pi\setminus F}\mathbb{E}_t\left[\left|I_{t+1}^{(pp),i,j}\backslash (I_t\cup \mathcal{H}_{i,j'})\right|\right]
\quad
\text{for all } i\in\Pi\setminus F, j'\in \Pi_i,
\end{align*}
which in turn, using \eqref{pushpulleq} and $X_{t,j,i}\geq X_{t,j}$ for all $j\in\Pi\setminus F$ and $i\in \Pi_j$, implies for $c := c_2 - \eta$
\begin{align}\label{assumption}
\mathbb{E}_t\left[X_{t+1,i,j'}\right]&\geq X_{t,i}+ c \cdot q\sum_{j\in\Pi_i}
\left(\frac{d_{ij}-6\eta}{\sum_{1\le \ell \le k}d_{i\ell}}+\frac{d_{ij}-6\eta}{\sum_{1\le \ell \le k}d_{\ell j}}\right)X_{t,j} 
\quad
\text{for all } i\in\Pi\setminus F, j'\in \Pi_i.
\end{align}
Assume that $\eqref{assumption}$ does not only hold in expectation but also for a slightly smaller $c$, say $c - \eta$, with high probability. We are going to show this at the end of the proof. Using this assumption and a union bound over $j'\in\Pi_i$ gives whp 
%\begin{align}\label{extend}
%X_{t+1,i}=\min_{j'\in\Pi_i}X_{t+1,i,j'}\geq  \left\langle b_i, X_t\right\rangle\quad \text{ for all }i\in\Pi\setminus F,
%\end{align}
%where
%\begin{equation}\label{aij}
%\begin{aligned}
%b_{ij}=%\left\{
%%\begin{matrix}
%%cq\left(\frac{d_{ij}-6\eta}{\sum_{1\le \ell \le k}d_{i\ell}}+\frac{d_{ij}-6\eta}{\sum_{1\le \ell \le k}d_{\ell j}}\right),& j\neq i\\
%%1 +cq\left(\frac{d_{ij}-6\eta}{\sum_{1\le \ell \le k}d_{i\ell}}+\frac{d_{ij}-6\eta}{\sum_{1\le \ell \le k}d_{\ell j}}\right),& j=i.
% \mathbb{1}[i = j] + c \cdot q\left(\frac{d_{ij}-6\eta}{\sum_{1\le \ell \le k}d_{i\ell}}+\frac{d_{ij}-6\eta}{\sum_{1\le \ell \le k}d_{\ell j}}\right), ~~ j \in \Pi_i,
% \quad
% \text{and}
% \quad
% b_{ij} = 0, ~~ j \in \Pi \setminus (F \cup \Pi_i).
%%\end{matrix}
%%\right.
%\end{aligned}
%\end{equation}
%Let $B$ be the  $|\Pi\setminus F| \times |\Pi\setminus F|$ matrix with entries as in the previous equation. Then we obtain from~\eqref{extend}
%\begin{align*}
%X_{t+1}\geq B\cdot X_t;
%\end{align*}
%Let us write $B = A + \Delta A$, where
%$$(\Delta A)_{ij}=\left\{
%\begin{matrix}
%0 ,&i\in \Pi\backslash F \text{ and } j\in \Pi_i\\
%-\mathbb{1}[i = j] -  c \cdot q\left(\frac{d_{ij}-6\eta}{\sum_{1\le \ell \le k}d_{i\ell}}+\frac{d_{ij}-6\eta}{\sum_{1\le \ell \le k}d_{\ell j}}\right), &i\in \Pi\backslash F \text{ and } j \in \Pi\setminus (F \cup \Pi_i)
%\end{matrix}.
%\right.$$
%Note that $A$ is symmetric.
\begin{align}\label{extend}
X_{t+1,i}=\min_{j'\in\Pi_i}X_{t+1,i,j'}\geq  \left\langle a_i, (X_{t,j})_{j\in\Pi_i}\right\rangle\quad \text{ for all }i\in\Pi\setminus F,
\end{align}
where for $i\in\Pi\setminus F$ and $j\in\Pi_i$ we have
\begin{equation}\label{aij}
\begin{aligned}
a_{ij}= \mathbb{1}[i = j] + c \cdot q\left(\frac{d_{ij}-6\eta}{\sum_{1\le \ell \le k}d_{i\ell}}+\frac{d_{ij}-6\eta}{\sum_{1\le \ell \le k}d_{\ell j}}\right).
\end{aligned}
\end{equation}
Let $A$ be the  $|\Pi\setminus F| \times |\Pi\setminus F|$ matrix with entries as in the previous equation, i.e., $A=(a_{ij})_{(i,j)\in(\Pi\setminus F)^2}$ is given by \eqref{aij} for all $(i,j)\in(\Pi\setminus F)^2$.
Note that $A$ is symmetric. Then we obtain from~\eqref{extend}
\begin{align*}
X_{t+1}\geq B\cdot X_t;
\end{align*}
with $B = A + \Delta A$, where
$$(\Delta A)_{ij}=\left\{
\begin{matrix}
0 ,&i\in \Pi\backslash F \text{ and } j\in \Pi_i\\
-a_{ij}, &i\in \Pi\backslash F \text{ and } j \in \Pi\setminus (F \cup \Pi_i)
\end{matrix}.
\right.$$
Set $F':=\{(i,j)\in (\Pi\setminus F )^2\mid j \in \Pi\setminus (F \cup \Pi_i) \}.$ As $d(u)\geq \alpha n/2$ for all $u\in V$ and some $\alpha>0$, we also know that $\sum_{1\le\ell\le k}d_{\ell j}\geq k\alpha/2 $. Together with $0\leq d_{i,j}\leq 1$ for all $(i,j)\in [k]^2$ we get that 
$|(d_{ij}-6\eta)/\sum_{1\le \ell\le k}d_{i\ell}|\le 2/(\alpha k)$. Using that $|F'|\leq 2\eta k^2$ we obtain 
\begin{align*}
\|\Delta A\|_F^2&=\sum_{(i,j)\in F'} a_{ij}^2\leq \sum_{(i,j)\in F'}\left(\frac{2}{\alpha k}\right)^2\leq 2\eta k^2\left(\frac{2}{\alpha k}\right)^2=\frac{8\eta}{\alpha^2}
\end{align*}
and thus $\|\Delta A\|_F\leq 2\sqrt{2\eta}/\alpha$. This leaves us with bounding the biggest eigenvalue $\lambda_{\max}$ of $A$. 
Using the well-known inequality for symmetric matrices $\lambda_{\max}\geq \sum_{(i,j)\in (\Pi\setminus F)^2} A_{ij}/|\Pi\setminus F|$ we obtain
\begin{align*}
& \lambda_{\max}
\geq \frac1{|\Pi \setminus F|} {\sum_{(i,j)\in (\Pi\setminus F)^2} a_{ij}}
%=k^{-1} {\sum_{(i,j)\in \Pi\setminus F^2} a_{ij}'}
\\
&~~\geq \frac{1}{|\Pi \setminus F|}
	\left(
	\sum_{(i,i)\in (\Pi \setminus F)^2} 1
	+ \sum_{(i,j)\in [k]^2} \frac{cq(d_{ij}-6\eta)}{\sum_{1\le \ell \le k}d_{i\ell}}
	+\sum_{(i,j)\in [k]^2}\frac{ cq(d_{ij}-6\eta)}{\sum_{1\le \ell \le k}d_{\ell j}}
	-2\sum_{i\in [k]\setminus (\Pi\setminus F)} \sum_{j\in [k]}\frac{cq}{\sum_{1\le \ell \le k}d_{\ell j}}
	\right).
\end{align*}
Note that $|\Pi\setminus F| \ge (1-\eta)k$, $|[k] \setminus (\Pi \setminus F)| \le \eta k$. Moreover, $\sum_{1\le \ell \le k}d_{\ell j}\geq \alpha k/2$ for all $j \in [k]$. Thus
\begin{align*}
\lambda_{\max}
&\ge 1 + \frac{1}{k}
	\left(
	 cqk
	+ cqk
	- 12cq\sum_{(i,j)\in [k]^2}\frac{\eta}{\sum_{1\le \ell \le k}d_{\ell j}}
	- 2cq \frac{\eta k^2}{\alpha k/2}\right)
	\geq 1+2cq(1-8\eta/\alpha).
\end{align*}
Choosing $\eta$ small enough such that  $2q(1-c(1-8\eta/\alpha)),2\sqrt{2\eta}/\alpha \le \nu$ implies the claim of this lemma.

This leaves us with proving that \eqref{assumption} also holds with high probability. As $|I_{t+1}^{(pp)}|$ conditioned on $I_t$  is a self-bounding function so is $|I_{t+1}^{(pp),i}\setminus I_t^i|$ for all $ i\in \Pi\setminus F$ and therefore also $|I_{t+1}^{(pp),i}\backslash (I_t\cup \mathcal{H}_{i,j'})|=: Y_{t+1,i,j'}$ for all $i\in \Pi\setminus F $ and $ j'\in \Pi_i$. Note that $Y_{t+1,i,j'}=X_{t+1,i,j'} - X_{t,i}$;  Lemma \ref{lugosi} yields that
$$P_t\left[Y_{t+1,i,j'}\geq \left(1-\mathbb{E}_t[Y_{t+1,i,j'}]^{-1/3}\right)\mathbb{E}_t[Y_{t+1,i,j'}]\right]\geq 1- \mathbb{E}_t[Y_{t+1,i,j'}]^{-1/3}$$
and therefore setting 
$$Z_{t,i}=c \cdot q\sum_{j\in\Pi_i}
\left(\frac{d_{ij}-6\eta}{\sum_{1\le \ell \le k}d_{i\ell}}+\frac{d_{ij}-6\eta}{\sum_{1\le \ell \le k}d_{\ell j}}\right)X_{t,j}\quad \text{for all }i\in \Pi\setminus F$$
and using \eqref{assumption}, i.e. $\mathbb{E}_t[Y_{t+1,i,j'}]\geq Z_{t,i}$ for all $i\in\Pi\setminus F$ and $j'\in\Pi_i$ we get with probability at least $1-k^3Z_{t,i}^{-1/3}$
$$Y_{t+1,i,j'}\geq (1-Z_{t,i}^{-1/3})Z_{t,i}\quad\text{for all }i\in\Pi\setminus F \text{ and }j'\in\Pi_i.$$
\begin{comment}
As this bound is uniform in $j$ applying a union bound implies that with probability at least $1-k\cdot Y_{t,i}^{-1/3}$
$$X_{t+1,i}\geq (1-Y_{t,i}^{-1/3})Y_{t,i}.$$
Setting $\mathcal{A}_{t+1}=\bigwedge_{i\in [k]}\left(X_{t+1,i}\geq (1-Y_{t,i}^{-1/3})Y_{t,i}\right)$ and again applying a union bound yields
$$P_t\left[\mathcal{A}_{t+1}\right]\geq 1-\sum_{1\leq i\leq k}kY_{t,i}^{-1/3}\geq 1-\sum_{1\leq i\leq k}kX_{t,i}^{-1/3}=1-o(1),$$
as $X_{t,i}=\omega(1).$ 
\end{comment}
This and $|I_t^i|\geq X_{t,i}$ for all $i\in\Pi\setminus F$ implies that~\eqref{assumption} also holds with high probability for a marginally smaller $c$, as claimed.
\end{proof}

\paragraph{Extension.}
We now solve the linear recurrence relation above and extend it to more than one round to get an upper bound on the runtime of \pushpull. We first state a Chernoff Bound that will be very useful in the next lemma.
\begin{lemma}~\cite{Janson2014}
\label{chernoff}  Let $\varepsilon, \delta>0$.
Suppose that $X_1,...,X_n$ are independent geometric random variables with parameter $\delta$, so $\mathbb{E}\left[X_i\right]= 1/\delta$ for each i. Let $X:=\sum_{1\leq i \leq n}X_i, \mu = \mathbb{E}\left[X\right]=n/\delta$. Then 
$$P\left[X\geq (1+\varepsilon)\mu\right] \le e^{-n(\varepsilon - \log(1+\varepsilon))}\leq e^{-\varepsilon^2n/2(1+\varepsilon)}$$
\end{lemma}
 Together with Lemma \ref{startup} the following lemma  implies  Theorem \ref{pushPullIsRobust} and Theorem \ref{firstPhasePushPullRobust}.
\begin{lemma}\label{pushpullResUpper}
Consider the setting of Theorems \ref{pushPullIsRobust}, \ref{firstPhasePushPullRobust} and let $I_t=I_t^{(pp)}$. The following statements hold whp.
\begin{enumerate}[label={(\alph*)},ref={\thetheorem~(\alph*)}]
\item \label{RegularStart}
\itemsep0em 
Let $I\subseteq V_n$ satisfying  $|I|=\Theta(n)$, then there is $t=\Theta(\log\log n)$, such that whp $|I_{t}|\geq |I_{t}\cap I|\geq \log\log n$.
\item Let $\log\log n\leq|I_t| \leq n/\log n$. Then there is $\tau\leq\log_{1+2q} (n/|I_t|)+o(\log n)$ such that $|I_{t+\tau}| >n/\log n.$
\item Let $n/\log n\leq|I_t| \leq n -n/\log n$. Then there is $\tau = o(\log n)$ such that $|I_{t+\tau}| >n-n/\log n.$
\item Let $|I_t|  \geq n - n/\log n$ and $q= 1$. Then there is $\tau = o(\log n)$ such that $|I_{t+\tau}| =n.$
\end{enumerate}
\end{lemma}
\begin{proof}
As $|I_t^{(pp)}|\geq |I_t^{(\pull)}| $ clearly $c)$ and $d)$ follow from Lemma \ref{pullResUpper}.
We  show $a)$ by determining a lower bound for the probability that an arbitrary vertex gets informed after a constant number of rounds. Set $\beta=\min \{\alpha, \epsilon\}$, let $I_0 = \{u\}$ and choose $w\in V, w\neq u$. By Lemma \ref{CountingPaths} there is $d\leq 8/\beta^2+2$ and $c=(\beta^4/64)^{8/\beta^2 +3} \in (0,1)$ such that there are at least $c n^{d-1}$ paths of (edge) length $d$ from $u$ to $w$. Let~$\gamma=(u,v_1, \dots, v_{d-1},w)$ be such a path from $u$ to $w$, and denote by~$A_\gamma$ the event that $w$ is informed via~$\gamma$ after exactly $d$ rounds performing only \push operations, i.e.,~$A_\gamma$ is the event that in the first round the randomly selected neighbour of $u$ is $v_1$, in the second round the randomly selected neighbour of $v_1$ is $v_2$ and so forth, until in the $d$th round the randomly selected neighbour of $v_{d-1}$ is $w$. Obviously, the probability of~$A_\gamma$ is  bounded from below by $n^{-d}$. Let further $\gamma'\neq \gamma$ be another path from $u$ to $w$ with length $d$. As $\gamma$ and $\gamma'$ differ by at least one edge  we readily obtain that $P[A_\gamma\cap A_{\gamma'}]=0$. Let $\Gamma$ denote the set of all paths with length $d$ from $u$ to $w.$
Having done these preparations we use them to conclude for all $w\in V$ and $t\geq 0$
\begin{equation}\label{lowerBoundProbPush}
P_t[w \in I_{t+d}]\ge P_t\left[\bigcup_{\gamma\in \Gamma}A_\gamma\right] \ge \sum_{\gamma\in\Gamma}P_t[A_\gamma]\geq \sum_{\gamma\in\Gamma}n^{-d}\geq \frac{c}{n}.
\end{equation}
We define a modified protocol as follows. Wait $d:=\lceil 8/\beta^2+2\rceil$ rounds, after that with probability $c$ choose one uninformed vertex uniformly at random and set it as informed. Repeat. Call the vertices informed by this algorithm $I_t^\star$. Then the probability for any vertex to be informed after $d$ rounds is 
$$P_t[v\in I_{t+d}^\star|v\notin I_t^\star]= {c}/{n}.$$  
Thus for any $t\ge 0$ 
$$P_t[v\in I_{t+d}|v\in U_t]\geq P_t[v\in I_{t+d}^\star|v\notin I_t^\star]= {c}/{n}.$$
Note that for any $s \in \mathbb{N}$ the set $I_{sd}^\star$ is generated by a very simple procedure: $s$ times independently, with probability $c$, we choose a random vertex and put it into $I_{sd}^\star$. Thus $|I_{sd}^\star \cap I|$ is binomially distributed with $s$ trials, where each one has success probability $c|I|/n = \Theta(c)$; it follows readily that $|I_{sd}^\star \cap I|$ concentrates around a multiple of $s$ for large $s$, and the claim follows by choosing $s = \Theta(\log\log n )$.
%
%
%We use these ingredients to apply Lemma \ref{mulzer}. Let $X_i, i\in V$ be the indicator random variable with $X_i=1$ if and only if $i\notin I_{c'd\log\log n}^\star$, then $P[X_i=1]=1-p$ for all $i\in V$ and these random variables are negatively correlated.
%Let $I\subseteq [n], |I|=\Theta(n)$ and $X=\sum_{i\in I}X_i$ then Lemma \ref{mulzer} yields
%\begin{align*}
%P\left[|I_{c'd\log\log n}\cap I|\leq \frac{1}{2}p|I|\right]&\leq P\left[X\geq \Bigl(1-\frac{1}{2}\Bigr)p|I|\right]
%\leq \exp^{-1}\left(\frac{p}{8(1-p)}|I|\right)=o(1)
%\end{align*}
% Summarizing, there is  $t=\Theta(\log\log n)$ such that whp $|I_t|\geq |I_t^{\star}\cap I|\geq \log\log n$. 

 This leaves $b)$ to be shown.
 Part $a)$ implies that there is some $t_0=o(\log n)$ such that $X_{t_0,i}=\Theta(\log \log n) $ for all $ i\in \Pi\setminus F$ by choosing $I=V_i\setminus (N_i\cup \mathcal{E}_{i,j}),j\in \Pi_i$ and applying a union bound over $i$ and $j$. 
 Thus we can apply Lemma \ref{RegularRecursion}. It gives whp, say with probability $1-g(n)=1-o(1)$, that $X_{t+1}\geq(A+\Delta A)X_t$, $A$ has maximal eigenvalue $\lambda_{\max}(A)\geq 1+2q-\nu$ and $\|\Delta A\|_F\leq \nu$. Then $B:=A+\Delta A$ has maximal eigenvalue $\lambda_{\max}(B)\geq \lambda_{\max}(A)-\|\Delta A\|_F\geq 1+2q-2\nu$ (Theorem of Wielandt-Hoffmann, compare e.g. \cite{Hoffman1953}) . 
 
Set $f(n):=(\log(n/\log n))^{2/3}$.  Our assumptions guarantee that $f(n) = \omega(1)$ and $f(n) = o(\log n)$. Moreover, set
$$\tau := 1/(1-g(n))\log(n/\log n)/\log (\lambda_{\max}(B))  + f(n)= \log(n)/\log (\lambda_{\max}(B))+o(\log n).$$ 
Let $(X_i)_{i \in \N}$ be i.i.d.~geometric random variables with expectation $1-g(n)$. Set $X=X_1+X_2 + \dots + X_{\mathcal{T}}$ with $\mathcal{T} = \log(n/\log n)/\log (\lambda_{\max}(B))$. We show that $P[X\leq  \tau]= 1-o(1)$. To see this, note first that by linearity of expectation $\mathbb{E}[X]=\tau - f(n)$. Then with Lemma \ref{chernoff}
\[
	P[X \le \tau]
	= P\left[X \le \Big(1 + \frac{f(n)}{\tau + f(n)}\Big)\mathbb{E}[X]\right]
	\ge 1 - \exp\left(-\Theta\Bigl(\frac{f(n)^2}{\tau}\Bigr)\right)
	= 1- o(1).
\]
Thus we have whp
$$|I_{t+\tau}|\geq\|X_{t+\tau}\|_1\geq \|B^\mathcal{T} X_{t_0}\|_1.$$
Let $v$ be an eigenvector of $B$ to $\lambda_{\max}(B)$. As $v\neq 0$ there is an index $\ell$ such that $v_\ell\neq 0$. Without loss of generality we can assume that $v_\ell=1$, as $v/v_\ell $ is also an eigenvector to $\lambda_{\max}(B)$.
% (Theorem of Perron–Frobenius, compare e.g. \cite{Wielandt1950}, as $B$ is a positive matrix). 
%Assuming such a vector does not exist. Let $\lambda_1\geq \lambda_2 \geq \dots \geq \lambda_k$ be the eigenvalues of $B$ and $k=\text{dim}(B)$. There is a positive vector $w$ such that $w=\sum_{1\leq \ell\le k}a_\ell v_\ell$, where $a_\ell\in\mathbb{R}$ and $v_\ell$ is an eigenvector to to $\lambda_\ell$ and $a_1\neq 0$. Let $m\in\mathbb{N}$, then $(B^mw)_i\geq 0$ for all $1\le i\le k$ and $m%\in\mathbb{N}$. But  
%$$B^mw=B^m\sum_{1\le\ell\le k}a_\ell v_\ell=\sum_{1\le\ell\le k}a_\ell\lambda_\ell^m v_\ell=(1+o(1/m))a_1\lambda_1^mv$$
%and there are indices $1\leq i,j\leq k$ such that $v_i<0<v_j$ which is a contradiction to $(B^mw)_i\geq 0$ for all $1\le i\le k$.
Thus $(B^\mathcal{T} v)_\ell=\lambda_{\max}(B)^\mathcal{T}$, $(B^\mathcal{T} (X_{t_0}-v))_i\geq 0$ for all $1\leq i\leq k$ and therefore
\begin{align*}
|I_{t+\tau}|
%\geq \|B^\mathcal{T}X_{t_0}\|_\ell
\ge (B^\mathcal{T}X_{t_0} )_\ell\ge (B^\mathcal{T}(v+X_{t_0}-v))_\ell=(B^\mathcal{T} v)_\ell+(B^\mathcal{T}(X_{t_0}-v))_\ell\geq (B^\mathcal{T} v)_\ell \geq \lambda_{\max}(B)^\mathcal{T}.
\end{align*}
Our choice of $\mathcal{T}$ yields whp $|I_{t+\tau}|\geq \lambda_{\max}(B)^\mathcal{T}\ge {n}/{\log n}$. Note that, since $\nu>0$ was chosen arbitrarily, we actually have that $\tau\leq \log_{1+2q}(n)+o(\log n)$, and the proof is completed.
\end{proof}

\subsection{Proof of Theorem \ref{pushPullIsNotRobustq} --- edge deletions may slow down \pushpull}
\label{edge_deletions_slow_down_push_pull_for_q_smaller_than_1}

 For any $0 < \varepsilon < 1/2$, $q\in(0,1)$ we consider a sequence of graphs $(G_n(\varepsilon))_{n\in\mathbb{N}}=((V_n,E_n))_{n\in \mathbb{N}}$ that is similar to the one studied in the proof of Theorem \ref{lastPhasePushNotRobust}. Let $V_n=A_n\cup B_n$ with $ A_n:=\{1,\dots,\lfloor n/2 \rfloor\}, B_n:=\{\lfloor n/2 \rfloor+1, \dots, n\}$ and  deg$(v)=n-1$ for all $v \in A_n$. Let the induced subgraph of $B_n$ be a random graph in which each edge is included independently with probability $p= 1-2\varepsilon$. We know and it is easy to show, see for example \cite[Section IV]{fountoulakis2010reliable}, that whp this subgraph is almost regular, i.e.,
\begin{align}\label{label3}
d_{B_n}(v)=(1+o(1))(1-2\varepsilon)n/2 \quad\text{for all } v\in B_n,
\end{align}
and is an expander, which means that for every $S_n\subseteq B_n$, $1 \leq |S_n| \leq n/4$ and $d_{B_n}:=(1-2\varepsilon)n/2$ we have 
\begin{align}\label{label2}
e(S_n,B_n\backslash S_n) =(1+o(1)) \frac{d_{B_n} |S_n| |B_n \setminus S_n|}{|B_n|}
= (1-2\varepsilon + o(1))|S_n||B_n\setminus S_n|.
\end{align}
At first we give a statement that describes the expected number of informed vertices after performing one round of \pushpull.
\begin{lemma}Let $G_n(\varepsilon)=(A_n\cup B_n,E_n)$ be as above.
\begin{enumerate}[label={(\alph*)},ref={\thetheorem~(\alph*)}]
\itemsep0em 
\item \label{push_pull_exa_exp} Let $\sqrt{\log n}\leq |I_t| \leq n/\log n$ and set
$$
X_t
=\Big(\big|I_t^{(pp),(A)}\big|,\big|I_t^{(pp),(B)}\big|\Big)
:=\Big(\big|I_t^{(pp)}\cap A_n\big|,\big|I_t^{(pp)}\cap B_n\big|\Big).$$
Then $\mathbb{E}_t[X_{t+1}]=(1+o(1))MX_t$, where 
\begin{align*}
M=\begin{pmatrix}
1+q&q\bigl(1+\varepsilon/(2-2\varepsilon)\bigr)\\q\bigl(1+\varepsilon/(2-2\varepsilon)\bigr) & 1+q\bigl(1-2\varepsilon/(2-2\varepsilon)\bigr)
\end{pmatrix}.
\end{align*}
\item \label{push_pull_exa_exp_b}Let $|U_t^{(pp)}|\leq n/\log n$. Then
$\mathbb{E}_{t}[|U_{t+1}^{(pp)}|] \leq (1+o(1)) e^{-q(1/2+(1/2-\varepsilon)/(1-\varepsilon))}\left(1-q\right)|U_t|. $
\end{enumerate}
\end{lemma}
\begin{proof}
For $J \in \{A,B\},J_n \in \{A_n,B_n\}$ set $U_t^{(J)}:=U_t\cap J_n, I_t^{(J)}:=I_t\cap J_n$ and $I_{t+1}^{(pp),(J)}=I_{t+1}^{(pp)}\cap J_n$. 
We first prove $a)$ by computing the expected number of informed vertices after a single round. Since $d(u) = \Omega(n)$ for all $u\in V_n$ and $|I_t|\leq n/\log n$, the probability of $u\in U_t$ being informed by \pull is
$$P_t\left[u \in I_{t+1}^{(\pull)}\setminus I_t\right]=\frac{q|N(u)\cap I_t|}{|N(u)|}=o(1).$$ 
As the events of $u$ being informed by \push and \pull are independent we have $P_t[u \in (I_{t+1}^{(\push)}\cap I_{t+1}^{(\pull)})\setminus I_t]=o(1)P_t[u\in I_{t+1}^{(\push)}\setminus I_t]$. Thus 
\begin{align*}
\mathbb{E}_t\left[\big|I_{t+1}^{(pp)}\setminus I_t\big|\right]=(1+o(1))\left(\mathbb{E}_t\left[\big|I_{t+1}^{(push)}\setminus I_t\big|\right]+\mathbb{E}_t\left[\big|I_{t+1}^{(pull)}\setminus I_t\big|\right]\right).
\end{align*}
We look at \pull in detail first. Recall that deg$(v)=n-1$ for all $v \in A_n$ and deg$(v)=(1+o(1))(1-\varepsilon)n$ for all $v \in B_n$. Moreover, using  \eqref{label2} we obtain
\begin{align*}
\mathbb{E}_t\left[\big|I_{t+1}^{(pull)}\setminus I_t\big|\right]&=\sum_{u\in U_t}q\frac{|N(u)\cap I_t|}{|N(u)|}=\sum_{u\in U_t^{(A)}}q\frac{|N(u)\cap I_t|}{|N(u)|}+\sum_{u\in U_t^{(B)}}q\frac{|N(u)\cap I_t|}{|N(u)|}\\&=(q+o(1))\frac{n}{2}\left(\frac{|I_t^{(A)}|+|I_t^{(B)}|}{n}+\frac{|I_{t}^{(A)}|+(1-2\varepsilon)|I_t^{(B)}|}{(1-\varepsilon)n}\right)
\end{align*}
and thus
\begin{equation*}
\begin{aligned}
\mathbb{E}_t\left[|I_{t+1}^{(pull),(A)}\setminus I_t|\right] & =(q+o(1))\frac{|I_t^{(A)}|+|I_t^{(B)}|}{2}, \\
\mathbb{E}_t\left[|I_{t+1}^{(pull),(B)}\setminus I_t|\right] & =(q+o(1))\frac{|I_{t}^{(A)}|+(1-2\varepsilon)|I_t^{(B)}|}{2(1-\varepsilon)}.
\end{aligned}
\end{equation*}
Next we consider \push. We obtain by using that $(1-1/n)^n= e^{-1+o(1)}$%$(1-1/n)^n\le e^{-1}\le(1-1/(n-1))^n $, $1-1/n=(1+o(1))e^{-1/n}$
\begin{align*}
\mathbb{E}_t\left[\big|I_{t+1}^{(push)}\setminus I_t\big|\right]&=\sum_{u\in U_t}1-\prod_{i\in N(u)\cap I_t}\left(1-\frac{q}{|N(i)|}\right)
\\&=
\sum_{u\in U_t}1-\left(1-\frac{q}{n}\right)^{|I_t^{(A)}|} 
\cdot  \left(1-\frac{(1+o(1))q}{(1-\varepsilon)n}\right)^{\mathbb1{\left[u \in U_t^{(A)}\right]}|I_t^{(B)}| + \mathbb1{\left[u \in U_t^{(B)}\right]}|N(u)\cap I_t^{(B)}|}
\\&=
\sum_{u\in U_t}1-\exp\left(-(q+o(1))\Biggl(\frac{|I_t^{(A)}|}{n}+\frac{\mathbb1{\left[u \in U_t^{(A)}\right]}|I_t^{(B)}| + \mathbb1{\left[u \in U_t^{(B)}\right]}|N(u)\cap I_t^{(B)}|}{(1-\varepsilon)n}\Biggr)\right).
\end{align*}
Using that $1-1/n=(1+o(1))e^{-1/n}$ we get
\begin{align*}
\mathbb{E}_t\left[\big|I_{t+1}^{(push)}\setminus I_t\big|\right]&=  (q+o(1))\sum_{u\in U_t}\Biggl(\frac{|I_t^{(A)}|}{n}+\frac{\mathbb1{\left[u \in U_t^{(A)}\right]}|I_t^{(B)}| + \mathbb1{\left[u \in U_t^{(B)}\right]}|N(u)\cap I_t^{(B)}|}{(1-\varepsilon)n}\Biggr)
%\\&=
%(q+o(1))\left(\biggl(\frac{|I_t^{(A)}|}{2}+\frac{|I_t^{(B)}|}{2}+\frac{\varepsilon|I_t^{(B)}|}{2(1-\varepsilon)}\biggr)+\biggl(\frac{|I_t^{(A)}|}{2}+\frac{|I_t^{(B)}|}{2}-\frac{\varepsilon|I_t^{(B)}|}{2(1-\varepsilon)}\biggr)\right).
\end{align*}
and thus with $|U_t^{(A)}|,|U_t^{(B)}|=(1-o(1))n/2$ and  \eqref{label2},
\begin{equation*}
\begin{aligned}
\mathbb{E}_t\left[|I_{t+1}^{(push),(A)}\setminus I_t|\right]&=(q+o(1))\biggl(\frac{|I_t^{(A)}|}{2}+\frac{|I_t^{(B)}|}{2}+\frac{\varepsilon|I_t^{(B)}|}{2(1-\varepsilon)}\biggr)\\
\mathbb{E}_t\left[|I_{t+1}^{(push),(B)}\setminus I_t|\right]&=(q+o(1))\biggl(\frac{|I_t^{(A)}|}{2}+\frac{|I_t^{(B)}|}{2}-\frac{\varepsilon|I_t^{(B)}|}{2(1-\varepsilon)}\biggr).
\end{aligned}
\end{equation*}
Accumulating the calculated expectations for \pull and \push yields the claim.
\begin{comment}
\begin{equation}\label{xt}
\begin{aligned}
\mathbb{E}_t\left[|I_{t+1}^{(pp),(A)}|\right]&=|I_t^{(A)}|+(q+o(1))\left( |I_t^{(A)}|+|I_t^{(B)}|+\frac{\varepsilon|I_t^{(B)}|}{2(1-\varepsilon)}\right)\\
\mathbb{E}_t\left[|I_{t+1}^{(pp),(B)}|\right]&=|I_t^{(B)}|+(q+o(1))\left( |I_t^{(A)}|+|I_t^{(B)}|-\frac{2\varepsilon|I_t^{(B)}|}{2(1-\varepsilon)}+\frac{\varepsilon|I_{t}^{(A)}|}{2(1-\varepsilon)}\right)
\end{aligned}
\end{equation}
which is the same as 
$\mathbb{E}_t[x_{t+1}]=(1+o(1))M x_t.$  
\end{comment}

Next we show $b)$. The assumption implies that $|I_t|=(1-o(1))n$ and therefore $|I_t^{(A)}|=|I_t^{(B)}|=(1-o(1))n/2$. Let $A_u$ be the event that an uninformed vertex $u$ does not get informed by the \push algorithm, let $B_u$ be the corresponding event for \pull. Then $A_u$ and $B_u$ are independent and $A_u\cap B_u$ is the event that $u$ does not get informed in the current round. Let $u\in U_t^{(A)}$, then 
\begin{align*}
P_t[A_u]&=\prod_{v\in I_t^{(A)}}\left(1-\frac{q}{|N(v)|}\right)\prod_{v\in I_t^{(B)}}\left(1-\frac{q}{|N(v)|}\right)=(1-o(1))\left(1-\frac{q}{n}\right)^{|I_t^{(A)}|}\left(1-\frac{q}{(1-\varepsilon)n}\right)^{|I_t^{(B)}|}
\\&=e^{-q(1/2+ 1/(2(1-\varepsilon)))}+o(1)\leq e^{-q(1/2+ (1-2\varepsilon)/(2(1-\varepsilon)))}+o(1)
\end{align*}
and
$$P_t[B_u]= 1-\frac{q|N(u)\cap |I_t||}{|N(u)|}= 1-\frac{q|I_t|}{n-1}=1-q+o(1).$$
%Thus combining the results for $u\in U_t^{(A)}$ gives
%\begin{equation*}
%\begin{aligned}
%\sum\limits_{u\in U_t^{(A)}} P_{t}[A_u\cap B_u]=\sum\limits_{u\in U_t^{(A)}}P_{t}[A_u] \cdot P_{t}[B_u]\leq(1+o(1)) e^{-q(1/2+ (1/2-\varepsilon)/(1-\varepsilon))}(1-q) |U_t^{(A)}|.
%\end{aligned}
%\end{equation*}
Consider now $u\in U_t^{(B)}$, then according to \eqref{label3} we have $|N(u)\cap I_t^{(B)}|=|N(u)\cap B_n|-|N(u)\cap U_t^{(B)}|=(1+o(1))(1-2\varepsilon)n/2$; therefore 
\begin{align*}
P_t[A_u]&=\prod_{v\in I_t^{(A)}}\left(1-\frac{q}{|N(v)|}\right)\prod_{v\in N(u)\cap I_t^{(B)}}\left(1-\frac{q}{|N(v)|}\right)=(1-o(1))e^{-q/2}\left(1-\frac{q}{(1-\varepsilon)n}\right)^{|N(u)\cap I_t^{(B)}|}\\&=e^{-q(1/2+(1-2\varepsilon)/(2(1-\varepsilon)))}+o(1)
\end{align*}
and 
\begin{align*}
P_t[B_u]&= 1-\frac{q|N(u)\cap |I_t||}{|N(u)|}= 1-(1+o(1))\frac{q(|I_t^{(A)}|+|N(u)\cap I_t^{(B)}|)}{(1-\varepsilon)n}
%=1-(1+o(1))\frac{q(n/2+%(1-2\varepsilon)n/2)}{(1-\varepsilon)n}\\&
=1-q+o(1).
\end{align*}
Combining the results for $u\in U_t^{(A)}$ and $u\in U_t^{(B)}$ we get
\begin{equation*}
\begin{aligned}
\mathbb{E}_{t}[|U_{t+1}|]=\sum\limits_{u\in U_t} P_{t}[A_u]P_t[B_u]\le (1+o(1)) e^{-q(1/2+(1/2-\varepsilon)/(1-\varepsilon))}\left(1-q\right)|U_t|.
\end{aligned}
\end{equation*}
\end{proof}
\begin{remark}\label{lambdaMax}
Let $\lambda_{\max}$ be the greatest eigenvalue of $M$ as defined in Lemma \ref{push_pull_exa_exp}. Then
$$\lambda_{max}= 1+2q+(2q(\sqrt{(\varepsilon^2/2-\varepsilon+1)}-1)+q\varepsilon )/(2-2\varepsilon)>1+2q.$$
\end{remark}
Next comes a lemma that bounds the runtime of \pushpull on $G_n(\varepsilon)$. In particular, Lemma \ref{pushpull_slow} $a)$ and $c)$ provide a lower bound on the runtime and Lemma \ref{pushpull_slow} $a), b)$ and $d)$ together with Lemma \ref{RegularStart} provide an upper bound.
\begin{lemma}\label{pushpull_slow}
Let $I_t=I_t^{(pp)}$, $\varepsilon>0$ and $\lambda=\lambda_{\max}(M)$ be the greatest eigenvalue of $M$ as given in Lemma \ref{push_pull_exa_exp}. Consider $G_n(\varepsilon)$.
\begin{enumerate}
%\item Let $1\leq|I_t|$. Then there is $\tau=o(\log n)$ such that $|I_{t+\tau}| >\sqrt{\log n}$.
\itemsep0em 
\item Let $\sqrt{\log n}\leq|I_t| \leq n/\log n$. Then there are $\tau_1,\tau_2= \log_{\lambda} (n/|I_t|)+o(\log n)$ such that $|I_{t+\tau_1}| < n/\log n<|I_{t+\tau_1}|.$
%\item Let $|I_t| \leq \sqrt{\log n}$. Then there is $\tau\geq \log_{\lambda} n-o(\log n)$ such that $|I_{t+\tau}| <n/\log n.$
\item Let $n/\log n\leq|I_t| \leq n -n/\log n$. Then there is $\tau = o(\log n)$ such that $|I_{t+\tau}| >n-n/\log n.$
\item Let $|I_t|\le n/\log n$. Then there is $\tau \geq \log n/\log((1-q)^{-1}\exp(q(1/2+(1/2-\varepsilon)/(1-\varepsilon)))) -o(\log n)$ such that $|I_{t+\tau}| <n.$ 
\item Let $|I_t| \geq n-n/\log n$ and $q \in (0,1)$. Then there is $\tau \leq \log n/\log((1-q)^{-1}\exp(q(1/2+(1/2-\varepsilon)/(1-\varepsilon)))) +o(\log n)$ such that $|I_{t+\tau}|=n$.
\end{enumerate}
\end{lemma}
\begin{proof}
We do not give a proof for $b)$ as it follows immediately from Lemma \ref{RegularStart}. For $J \in \{A,B\}$ set $U_t^{(A)}:=U_t\cap J_n, I_t^{(J)}:=I_t\cap J_n$. We prove $a)$ first. Let $t_0>0$ be the first round such that $|I_{t_0}|\geq \log\log n$ and set $x_t$ and $M$ as in Lemma \ref{push_pull_exa_exp}, note that Lemma \ref{RegularStart} also gives that $x_{t_0}\ge \log\log n/2$. Then for all $t\geq t_0$ such that $|I_t|\leq n/\log n$ we obtain from Lemma \ref{push_pull_exa_exp} that $\mathbb{E}_t[x_{t+1}]=(1+o(1))Mx_t$ and, in particular, $\mathbb{E}_t[(x_{t+1})_i]=\Theta(|I_t|)$ for $i\in\{1,2\}$. As every component of $x_t$ is self-bounding, Lemma \ref{lugosiappl} applies and we get for $i\in\{1,2\}$
$$P_t[|(x_{t+1})_i-\mathbb{E}_t[(x_{t+1})_i]|\geq \mathbb{E}_t[(x_{t+1})_i]^{2/3}]= O(|I_t|^{-1/3})$$ 
and by union bound, provided that $|I_t| \le n/\log n$,
\begin{align}\label{union}
P_t\left[\bigcap_{i\in \{1,2\}}\left(|(x_{t+1})_i-\mathbb{E}_t[(x_{t+1})_i]|\leq \mathbb{E}_t[(x_{t+1})_i]^{2/3}\right)\right]= 1-O(|I_t|^{-1/3}).
\end{align}
%Denote this event by $\mathcal{A}_t$. Conditioning on $\mathcal{A}_t, \dots,\mathcal{A}_{t_0}$ 
Using \eqref{union} we want to find a bound on $|I_{t+1}|$. We get as long as $|I_t| \le n/\log n$  that
\begin{align*}
\left((1-O(|I_{t_0}|^{-1/3}))M\right)^{t+1-t_0}x_{t_0}\leq x_{t+1}\leq \left((1+O(|I_{t_0}|^{-1/3}))M\right)^{t+1-t_0}x_{t_0}.
\end{align*}
As seen in Remark \ref{lambdaMax}, $M$ has maximal eigenvalue $\lambda_{\max}>1$ and as $M$ is a positive matrix there is a positive eigenvector $v$ to $\lambda_{\max}$, compare \cite{Wielandt1950}. This gives constants $c_1,c_2>0$ such that $ c_1v\log{\log n}\le x_{t_0}\le c_2 v\log{\log n} $ and for $t$ large enough
\begin{align*}
\frac{c_1}{c_2}\left((1-O(|I_{t_0}|^{-1/3}))\lambda_{\max}\right)^{t+1-t_0}x_{t_0}\leq x_{t+1}\leq \frac{c_2}{c_1}\left((1+O(|I_{t_0}|^{-1/3}))\lambda_{\max}\right)^{t+1-t_0}x_{t_0},
\end{align*}
and therefore
\begin{align*}
|I_{t+1}|\leq \frac{c_1}{c_2}((1+o(1))\lambda_{\max})^{t-t_0}|I_{t_0}|.
\end{align*}
as long as the right hand side is bounded by $n/\log n$. For all these $t$ we get additionally 
\begin{align*}
|I_{t+1}|\geq \frac{c_2}{c_1}((1-o(1))\lambda_{\max})^{t-t_0}|I_{t_0}|.
\end{align*}
Proceeding as in Examples \ref{concetrationRemark} and \ref{concetrationExample}, where we replace the events   ``$|I_{t}| \ge \mathbb{E}_{t-1}\left[|I_{t}|\right]-  \mathbb{E}_{t-1}\left[|I_{t}|\right]^{2/3}$ or $|I_t| \ge n/g(n)\text{''}$ and ``$\left||I_{t}|-\mathbb{E}_{t-1}\left[|I_{t}|\right]\right|\leq \mathbb{E}_{t-1}\left[|I_{t}|\right]^{2/3}$'' with ``$\bigcap_{i\in \{1,2\}}\left((x_{t+1})_i\ge (1-\mathbb{E}_t[(x_{t+1})_i]^{-1/3})\mathbb{E}_t[(x_{t+1})_i]\right)$ or $|I_{t}|\geq n/\log n$'' and ``$\bigcap_{i\in \{1,2\}}\left(|(x_{t+1})_i-\mathbb{E}_t[(x_{t+1})_i]|\leq \mathbb{E}_t[(x_{t+1})_i]^{2/3}\right)$'' we obtain the statement.
%Using \eqref{union} yields thus
%\begin{align*}
%P_{t_0}[\mathcal{A}_{t+1}\mid \mathcal{A}_1, \dots, \mathcal{A}_{t}]\geq 1- \lambda_{\max}^{-t/3}(c_4|I_{t_0}|)^{-1/3},
%\end{align*}
%Applying Proposition~\ref{concentration} then immediately gives that there are $\tau_1, \tau_2 = \log_c (n/|I_{t_0}|) + o(\log n)$ such that
%$$|I_{t_0+\tau_1}|\leq \frac{n}{10\log n}\leq |I_{t_0+\tau_2}|.$$
 % by applying a version of Lemma \ref{2.6lemma}, Lemma \ref{2.9} and Remark \ref{2.9b}. \\

Next we show $c)$. The assumption guarantees that less than $n/\log n$ vertices are informed. 
Thus $|U_t^{(B)}|\geq n/2-|I_t| \ge \left(1/2-{1}/{\log n}\right)n.$ We consider a modified dissemination process, where in each round, each uninformed vertex always chooses an informed neighbour (but does not necessarily get informed as the message transmission may fail),  and additionally each vertex chooses a neighbour iuar and after this round the chosen vertex is informed with probability $q$; in other words, we assume that also uninformed vertices can inform other vertices. In this modified process the probability of an uninformed vertex $u\in U_t^{(B)}$ staying uninformed after performing one round is given by the product of the probabilities of not being informed by \pull or via \push by a vertex in $A_n$ or $B_n$. Using \eqref{label2} and $(1-1/n)^n= e^{-1+o(1)}$ we get $g(n)=o(1)$ such that
\begin{align*}
P_t[u\in U_{t+1}^{(B)}]&=(1-q)\left(1- \frac{q}{n}\right)^{n/2}\left(1-\frac{q}{(1-\varepsilon)n}\right)^{|N(u)\cap B_n|} \\&= (1-q)\exp\left(-q\left(\frac{1}{2}+\frac{1/2-\varepsilon}{1-\varepsilon}\right)+g(n)\right).
\end{align*}
As we have seen in the proof of Lemma \ref{push_pull_exa_exp_b}, the probability to be in formed by \pushpull is greater for a vertex in $A_n$ than for a vertex in $B_n$. Therefore it is sensible to expect that some vertices in $B_n$ we will be the last to be informed.
Consequently denote by $E_{u}$ the event that a currently uninformed vertex $u\in U_t^{(B)}$ does not get informed in this modified version within the next $$\tau :=\frac{1}{\log({(1-q)^{-1}\exp(q(1/2+(1/2-\varepsilon)/(1-\varepsilon)-g(n)))})} \log(n) - h(n)$$ rounds where $h = o(\log n)$ and $h = \omega(1)$.
\begin{comment}
 Let $$c_{\varepsilon} = \frac{q-\log\left(1-q\right)}{q(1-1.5 \varepsilon)/(1-\varepsilon)-\log \left(1-q\right)};$$ then $c_{\varepsilon}>1+q/(2(q-\log(1-q))^2)\varepsilon>1$. Note that $$\tau =\frac{1}{(q (1-1.5 \varepsilon)/(1-\varepsilon)-\log \left(1-q\right)}\log n -h(n) =\frac{c_{\varepsilon}}{q-\log\left(1-q\right)} \log n-h(n).$$  
 \end{comment}
Therefore we have
\begin{align*}
P_t[E_{u}]&=
%\left((1-q)\left(1- \frac{q}{n}\right)^{n/2}\left((1-\frac{q}{(1-\varepsilon)n}\right)^{(1/2 - \varepsilon+g(n))n}\right)^{\tau} \\&= 
\left((1-q)\exp\left(-q\left(\frac{1}{2}+\frac{1/2-\varepsilon}{1-\varepsilon}\right)+g(n)\right)\right)^{\tau} = \frac{1}{n} e^{\omega(1)}.
\end{align*}
In this modified model the events $\{E_{u}\mid u\in U_t^{(B)}\}$ satisfy that there is $p =\omega(n^{-1})$ such that $P_t[E_{u}\mid \{\overline{E_v}:v\in U\}]\geq p$ for all $u\in B_n$ and $U\subseteq V\setminus \{u\}$. This follows immediately by the above calculations. Thus as $|U_t^{(B)}|=\Theta(n)$
\begin{align*}
P_t\left[\bigwedge\limits_{u \in U_t^{(B)}} \overline{E_{u}}\right] 
%\leq P\left[\bigwedge\limits_{u \in V} \overline{E_{u}}\right]
\le \prod\limits_{u \in U_t^{(B)}} (1-p)
%= \prod\limits_{u \in V} (1-P[E_{u}])
 \leq \exp\left(- \sum\limits_{u \in U_t^{(B)}} p\right)=o(1).
\end{align*}
%We used that $(1-\frac{q}{n})^{n} \rightarrow e^{-1}$ converges fast enough.
%Note that the probability of $E_{u_B}$ for \pushpull is then also $\omega(n^{-1})$; Lemma \ref{finish} concludes this part.
Finally we show $d)$.  
By Lemma \ref{push_pull_exa_exp_b}, we obtain that for any $\tau \in \mathbb{N}$, $$\mathbb{E}_{t}[|U_{t+\tau}|] \le \left( (1+o(1)) e^{-q(1/2+(1/2-\varepsilon)/(1-\varepsilon))}\left(1-q\right)\right)^{\tau}|U_t|. $$ 
Then for some\begin{align*}
%\tau &= \frac{\log(n)}{\log((1+o(1))(1-q)^{-1}\exp(q(1/2+(1/2-\varepsilon)/(1-\varepsilon))))} 
%\\& 
\tau := \frac{\log(n)}{\log((1-q)^{-1}\exp(q(1/2+(1/2-\varepsilon)/(1-\varepsilon))))} +o(\log n)
\end{align*}
we obtain  that, say,
$\mathbb{E}_{t}[|U_{t+\tau}|]\leq |U_t|/n \leq {1}/{\log n}.$
Thus $P_{t}[|U_{t+\tau}|\geq 1]\leq o(1)$ by Markov's inequality.
\end{proof}
Lemma \ref{pushpull_slow} together with Lemma \ref{startup} give that
$$T_{pp}(G_n(\varepsilon),q)=\log_{\lambda}n+\frac{1}{q(1-1.5 \varepsilon)/(1-\varepsilon)-\log \left(1-q\right)}\log n+o(\log n)$$
where $\lambda= 1+2q+(2q(\sqrt{(\varepsilon^2/2-\varepsilon+1)}-1)+q\varepsilon )/(2-2\varepsilon)>1+2q$. To see wether \pushpull actually slowed down (in terms of order $ \log n$) one has to compare the runtime on this sequence of graphs 
 to $c_{pp}\log n$; the runtime on expander sequences. In the figure below we can see that it slows down for nearly all values of $\varepsilon$ and $q$ in question; however, there are admissible values of $\varepsilon$ and $q$ such that the process even speeds up.
 
%\begin{comment}
 \begin{figure}[htbp]
 \centering
\begin{tikzpicture}

\begin{axis}[view={135}{35},axis lines* = left,tick label style={font=\small,align = center,rotate = 70},
label style={font=\small},
 %x tick label style = {font = \small, text width = 1.7cm, align = center, ,anchor = south west,text width = 1.7cm },
    xlabel=$q$, ylabel=$\varepsilon$,zlabel=$100\cdot\Delta$	
]
\addplot3[
	surf,
	domain=0.9:1,
	domain y=0:0.5,
] 
	{(1/ln((1/4)*(6*x*y+4*y-4*x-4-2*sqrt(2*y^2*x^2-4*y*x^2+4*x^2))/(-1+y))+1/(x*(1-(3/2)*y)/(1-y)-ln(1-x))-1/ln(1+2*x)-1/(x-ln(1-x)))*100};
	\addplot3 [black] table {
%0.9 0 0
1 0 0
1 0.5 0
0.9 0.5 0

};

\end{axis}
\end{tikzpicture}
\caption{Plotted values of $\Delta$ in $T_{pp}(G_n(\varepsilon),q)-c_{pp}\log n=\Delta \log n+o(\log n)$, for $0.9<q<1$ and $0<\varepsilon<1/2$. }
\end{figure}
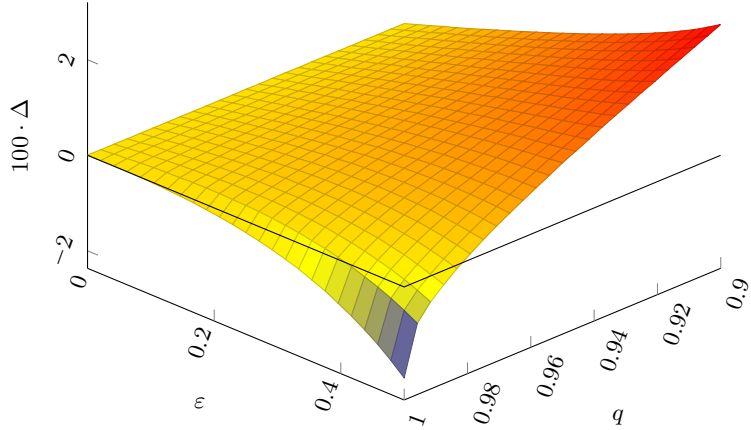
%\end{comment}

\phantomsection

\small
\bibliography{literature}

\begin{thebibliography}{10}

\bibitem{acan2017push}
H.~Acan, A.~Collevecchio, A.~Mehrabian, and N.~Wormald.
\newblock {On the push\&pull protocol for rumour spreading}.
\newblock In {\em {Extended Abstracts Summer 2015}}, pages 3--10. Springer,
  2017.

\bibitem{angel2017string}
O.~Angel, A.~Mehrabian, and Y.~Peres.
\newblock {The string of diamonds is tight for rumor spreading}.
\newblock In {\em {Approximation, Randomization, and Combinatorial
  Optimization. Algorithms and Techniques (APPROX/RANDOM 2017)}}. Schloss
  Dagstuhl-Leibniz-Zentrum fuer Informatik, 2017.

\bibitem{Boucheron2004}
S.~Boucheron, G.~Lugosi, and O.~Bousquet.
\newblock {Concentration inequalities}.
\newblock In {\em {Advanced Lectures on Machine Learning}}, pages 208--240.
  Springer, 2004.

\bibitem{boyd2006randomized}
S.~Boyd, A.~Ghosh, B.~Prabhakar, and D.~Shah.
\newblock {Randomized gossip algorithms}.
\newblock {\em IEEE/ACM Transactions on Networking (TON)}, 14(SI):2508--2530,
  2006.

\bibitem{censor2012global}
K.~Censor-Hillel, B.~Haeupler, J.~Kelner, and P.~Maymounkov.
\newblock Global computation in a poorly connected world: fast rumor spreading
  with no dependence on conductance.
\newblock In {\em Proceedings of the forty-fourth annual ACM symposium on
  Theory of computing}, pages 961--970. ACM, 2012.

\bibitem{chierichetti2018rumor}
F.~Chierichetti, G.~Giakkoupis, S.~Lattanzi, and A.~Panconesi.
\newblock Rumor spreading and conductance.
\newblock {\em Journal of the ACM (JACM)}, 65(4):17, 2018.

\bibitem{daum_rumor_2016}
S.~Daum, F.~Kuhn, and Y.~Maus.
\newblock {Rumor Spreading with Bounded In-Degree}.
\newblock In {\em {Structural Information and Communication Complexity - 23rd
  International Colloquium, {SIROCCO} 2016, Helsinki, Finland, July 19-21,
  2016, Revised Selected Papers}}, pages 323--339, 2016.

\bibitem{Dellamonica2008}
D.~Dellamonica, Y.~Kohayakawa, M.~Marciniszyn, and A.~Steger.
\newblock {On the Resilience of Long Cycles in Random Graphs}.
\newblock {\em Electr. J. Comb.}, 15(1), 2008.

\bibitem{Demers1988}
A.~Demers, D.~Greene, C.~Houser, W.~Irish, J.~Larson, S.~Shenker, H.~Sturgis,
  D.~Swinehart, and D.~Terry.
\newblock Epidemic algorithms for replicated database maintenance.
\newblock {\em ACM SIGOPS Operating Systems Review}, 22(1):8--32, 1988.

\bibitem{doerr2011social}
B.~Doerr, M.~Fouz, and T.~Friedrich.
\newblock {Social networks spread rumors in sublogarithmic time}.
\newblock In {\em {Proceedings of the forty-third annual ACM symposium on
  Theory of computing}}, pages 21--30. ACM, 2011.

\bibitem{doerr2017randomized}
B.~Doerr and A.~Kostrygin.
\newblock {Randomized Rumor Spreading Revisited}.
\newblock In {\em {44th International Colloquium on Automata, Languages, and
  Programming, {ICALP} 2017, July 10-14, 2017, Warsaw, Poland}}, pages
  138:1--138:14, 2017.

\bibitem{doerr2014tight}
B.~Doerr and M.~K{\"u}nnemann.
\newblock {Tight analysis of randomized rumor spreading in complete graphs}.
\newblock In {\em {Proceedings of the Meeting on Analytic Algorithmics and
  Combinatorics}}, pages 82--91. Society for Industrial and Applied
  Mathematics, 2014.

\bibitem{elsasser2009runtime}
R.~Els{\"a}sser and T.~Sauerwald.
\newblock {On the runtime and robustness of randomized broadcasting}.
\newblock {\em Theoretical Computer Science}, 410(36):3414--3427, 2009.

\bibitem{fountoulakis2010reliable}
N.~Fountoulakis, A.~Huber, and K.~Panagiotou.
\newblock {Reliable broadcasting in random networks and the effect of density}.
\newblock In {\em {2010 Proceedings IEEE INFOCOM}}, pages 1--9. IEEE, 2010.

\bibitem{fountoulakis2010rumor}
N.~Fountoulakis and K.~Panagiotou.
\newblock {Rumor spreading on random regular graphs and expanders}.
\newblock In {\em {Approximation, Randomization, and Combinatorial
  Optimization. Algorithms and Techniques}}, pages 560--573. Springer, 2010.

\bibitem{fountoulakis2012ultra}
N.~Fountoulakis, K.~Panagiotou, and T.~Sauerwald.
\newblock {Ultra-fast rumor spreading in social networks}.
\newblock In {\em {Proceedings of the twenty-third annual ACM-SIAM symposium on
  Discrete Algorithms}}, pages 1642--1660. SIAM, 2012.

\bibitem{Friedrich2013}
T.~Friedrich, T.~Sauerwald, and A.~Stauffer.
\newblock {Diameter and Broadcast Time of Random Geometric Graphs in Arbitrary
  Dimensions}.
\newblock {\em Algorithmica}, 67(1):65--88, Sep 2013.

\bibitem{frieze1985shortest}
A.~M. Frieze and G.~R. Grimmett.
\newblock {The shortest-path problem for graphs with random arc-lengths}.
\newblock {\em Discrete Applied Mathematics}, 10(1):57--77, 1985.

\bibitem{giakkoupis_tight_2011}
G.~Giakkoupis.
\newblock {Tight bounds for rumor spreading in graphs of a given conductance}.
\newblock In {\em {28th International Symposium on Theoretical Aspects of
  Computer Science, {STACS} 2011, March 10-12, 2011, Dortmund, Germany}}, pages
  57--68, 2011.

\bibitem{giakkoupis_tight_2014}
G.~Giakkoupis.
\newblock {Tight Bounds for Rumor Spreading with Vertex Expansion}.
\newblock In {\em {Proceedings of the Twenty-Fifth Annual {ACM-SIAM} Symposium
  on Discrete Algorithms, {SODA} 2014, Portland, Oregon, USA, January 5-7,
  2014}}, pages 801--815, 2014.

\bibitem{greenberg2014tight}
S.~Greenberg and M.~Mohri.
\newblock {Tight lower bound on the probability of a binomial exceeding its
  expectation}.
\newblock {\em Statistics \& Probability Letters}, 86:91--98, 2014.

\bibitem{haeupler2015simple}
B.~Haeupler.
\newblock Simple, fast and deterministic gossip and rumor spreading.
\newblock {\em Journal of the ACM (JACM)}, 62(6):47, 2015.

\bibitem{Hoffman1953}
A.~J. Hoffman and H.~W. Wielandt.
\newblock {The variation of the spectrum of a normal matrix}.
\newblock {\em Duke Math. J.}, 20:37--39, 1953.

\bibitem{hoory2006expander}
S.~Hoory, N.~Linial, and A.~Wigderson.
\newblock {Expander graphs and their applications}.
\newblock {\em Bulletin of the American Mathematical Society}, 43(4):439--561,
  2006.

\bibitem{Janson2014}
S.~{Janson}.
\newblock {Tail bounds for sums of geometric and exponential variables}.
\newblock {\em ArXiv e-prints:1709.08157}, Sept. 2017.

\bibitem{Karp2000}
R.~M. Karp, C.~Schindelhauer, S.~Shenker, and B.~V{\"o}cking.
\newblock {Randomized Rumor Spreading}.
\newblock In {\em {41st Annual Symposium on Foundations of Computer Science,
  {FOCS} 2000, 12-14 November 2000, Redondo Beach, California, {USA}}}, pages
  565--574, 2000.

\bibitem{Panagiotou2015}
K.~Panagiotou, X.~P{\'e}rez{-}Gim{\'e}nez, T.~Sauerwald, and H.~Sun.
\newblock {Randomized Rumour Spreading: The Effect of the Network Topology}.
\newblock {\em Combinatorics, Probability {\&} Computing}, 24(2):457--479,
  2015.

\bibitem{panagiotou2017asynchronous}
K.~Panagiotou and L.~Speidel.
\newblock {Asynchronous rumor spreading on random graphs}.
\newblock {\em Algorithmica}, 78(3):968--989, 2017.

\bibitem{rodl2010regularity}
V.~R{\"o}dl and M.~Schacht.
\newblock {Regularity lemmas for graphs}.
\newblock In {\em {Fete of combinatorics and computer science}}, pages
  287--325. Springer, 2010.

\bibitem{Sudakov2008}
B.~Sudakov and V.~H. Vu.
\newblock {Local resilience of graphs}.
\newblock {\em Random Struct. Algorithms}, 33(4):409--433, 2008.

\bibitem{Wielandt1950}
H.~Wielandt.
\newblock {Unzerlegbare, nicht negative Matrizen}.
\newblock {\em Mathematische Zeitschrift}, 52(1):642--648, Dec 1950.

\end{thebibliography}
\end{document}